\documentclass[onefignum,onetabnum]{siamart190516}
\usepackage{lipsum}
\usepackage{amsfonts}
\usepackage{graphicx}
\usepackage{epstopdf}
\usepackage{booktabs}
\usepackage{mathtools}

\usepackage{algpseudocode}
\algnewcommand{\Initialize}[1]{
  \State \textbf{Initialize:}
  \Statex \hspace*{\algorithmicindent}\parbox[t]{.8\linewidth}{\raggedright #1}
}
\algnewcommand{\Indent}[2]{
  \State {#1}
  \vspace{-2mm}
  \Statex \hspace*{\algorithmicindent}\parbox[t]{.9\linewidth}{\raggedright #2}
}

\ifpdf
  \DeclareGraphicsExtensions{.eps,.pdf,.png,.jpg}
\else
  \DeclareGraphicsExtensions{.eps}
\fi

\usepackage{enumitem}
\usepackage[caption=false]{subfig}
\usepackage{amssymb}
\newcommand{\vect}[1]{\boldsymbol{#1}}
\newcommand{\vectt}[1]{\boldsymbol{\mathbf{#1}}}

\newcommand{\dx}{\mathrm{d}x}
\newcommand{\dy}{\mathrm{d}y}

\newcommand{\fdx}{\frac{\mathrm{d}}{\dx}}

\newcommand{\bQ}{{\bf Q}}

\def\Q{Q^{t,(a,b,c)}}

\def\bQ{{\bf Q}^{t,(a,b,c)}}

\def\dbar{\bar\partial}

\def\XXint#1#2#3{{\setbox0=\hbox{$#1{#2#3}{\int}$}
\vcenter{\hbox{$#2#3$}}\kern-.5\wd0}}

\usepackage{color}
\definecolor{deepblue}{rgb}{0,0,0.5}
\definecolor{deepred}{rgb}{0.6,0,0}
\definecolor{deepgreen}{rgb}{0,0.5,0}

\newmuskip\pFqmuskip

\newcommand*\pFq[6][8]{%
  \begingroup % only local assignments
  \pFqmuskip=#1mu\relax
  \mathchardef\normalcomma=\mathcode`,
  \mathcode`\,=\string"8000
  \begingroup\lccode`\~=`\,
  \lowercase{\endgroup\let~}\pFqcomma
  {}_{#2}F_{#3}{\left(\genfrac..{0pt}{}{#4}{#5};#6\right)}%
  \endgroup
}
\newcommand{\pFqcomma}{{\normalcomma}\mskip\pFqmuskip}

\usepackage{todonotes}

\newsiamremark{remark}{Remark}
\newsiamremark{hypothesis}{Hypothesis}
\crefname{hypothesis}{Hypothesis}{Hypotheses}
\newsiamthm{claim}{Claim}

\headers{Hierarchies of semiclassical Jacobi polynomials}{I.~P.~A.~Papadopoulos, T.~S.~Gutleb, R.~M.~Slevinsky, S.~Olver}

\title{Building hierarchies of semiclassical Jacobi polynomials for spectral methods in annuli \thanks{Submitted DATE.
\funding{This work was completed with the support of the EPSRC grant EP/T022132/1 ``Spectral element methods for fractional differential equations, with applications in applied analysis and medical imaging" and the Leverhulme Trust Research Project Grant RPG-2019-144 ``Constructive approximation theory on and inside algebraic curves and surfaces". IPAP was also supported by the Deutsche Forschungsgemeinschaft (DFG, German Research Foundation) under Germany's Excellence Strategy -- The Berlin Mathematics Research Center MATH+ (EXC-2046/1, project ID: 390685689).}}}

\author{Ioannis P.~A.~Papadopoulos\thanks{Department of Mathematics, Imperial College London, London, UK \newline  \hspace*{4mm} (\email{ioannis.papadopoulos13@imperial.ac.uk});}
\and Timon S.~Gutleb \thanks{Mathematical Institute, University of Oxford, UK, (\email{timon.gutleb@maths.ox.ac.uk});}
\and Richard M.~Slevinsky \thanks{Department of Mathematics, University of Manitoba,
Canada, \newline  \hspace*{4mm} (\email{Richard.Slevinsky@umanitoba.ca});}
\and Sheehan Olver \thanks{Department of Mathematics, Imperial College London, UK, (\email{s.olver@imperial.ac.uk})}.}

\usepackage{amsopn}

\ifpdf
\hypersetup{
  pdftitle={Building hierarchies of semiclassical Jacobi polynomials for spectral methods in annuli},
  pdfauthor={I. P. A. Papadopoulos, T.~S.~Gutleb, R.~M.~Slevinsky, S.~Olver}
}
\fi

\def\addtab#1={#1\;&=}

\def\meeq#1{\def\ccr{\\\addtab}
 \begin{align*}
 \addtab#1
 \end{align*}
  }  
  
  \def\leqaddtab#1\leq{#1\;&\leq}

\def\pr(#1){\left({#1}\right)}
\def\br[#1]{\left[{#1}\right]}
\def\fbr[#1]{\!\left[{#1}\right]}

\def\ip<#1>{\left\langle{#1}\right\rangle}
\def\iip<#1>{\left\langle\!\langle{#1}\right\rangle\!\rangle}

\def\fpr(#1){\!\pr({#1})}

\def\Re{{\rm Re}\,}
\def\Im{{\rm Im}\,}

\def\floor#1{\left\lfloor#1\right\rfloor}
\def\ceil#1{\left\lceil#1\right\rceil}

\def\mapengine#1,#2.{\mapfunction{#1}\ifx\void#2\else\mapengine #2.\fi }

\def\map[#1]{\mapengine #1,\void.}

\def\mapenginesep_#1#2,#3.{\mapfunction{#2}\ifx\void#3\else#1\mapengine #3.\fi }

\def\mapsep_#1[#2]{\mapenginesep_{#1}#2,\void.}

\def\vcbr[#1]{\pr(#1)}

\def\bvect[#1,#2]{
{
\def\dots{\cdots}
\def\mapfunction##1{\ | \  ##1}
	\sopmatrix{
		 \,#1\map[#2]\,
	}
}
}

\def\vect[#1]{
{\def\dots{\ldots}
	\vcbr[{#1}]
}}

\def\vectt[#1]{
{\def\dots{\ldots}
	\vect[{#1}]^{\top}
}}

\def\Vectt[#1]{
{
\def\mapfunction##1{##1 \cr} 
\def\dots{\vdots}
	\begin{pmatrix}
		\map[#1]
	\end{pmatrix}
}}

\def\E{{\rm e}}
\def\I{{\rm i}}
\def\D{{\rm d}}
\def\dx{\D x}
\def\dy{\D y}

\def\tF_#1{{\tt F}_{#1}}

\def\tFC_#1{{\tt T}_{#1}}

\def\elllRpz_#1{\ell_{#1{\rm z}}^{(\lambda,R),p}}

\def\sopmatrix#1{\begin{pmatrix}#1\end{pmatrix}}

\def\Problem#1#2\par{\begin{problem}\label{pb:#1} #2\end{problem}}
\def\Theorem#1#2\par{\begin{theorem}\label{th:#1} #2\end{theorem}}
\def\Conjecture#1#2\par{\begin{conjecture}\label{conj:#1} #2\end{conjecture}}
\def\Proposition#1#2\par{\begin{proposition}\label{prop:#1} #2\end{proposition}}
\def\Definition#1#2\par{\begin{definition}\label{def:#1} #2\end{definition}}
\def\Corollary#1#2\par{\begin{corollary}\label{cr:#1} #2\end{corollary}}
\def\Lemma#1#2\par{\begin{lemma}\label{lm:#1} #2\end{lemma}}
\def\Example#1#2\par{\begin{example}\label{ex:#1} #2\end{example}}
\def\Remark #1\par{\begin{remark*}#1\end{remark*}}

\def\Proof{\begin{proof}}
\def\mqed{\end{proof}}

\def\Figuretwow[#1,#2]#3#4\par{
\begin{figure}[tb]
\begin{center}{
\includegraphics[width=#3]{Figures/#1}\includegraphics[width=#3]{Figures/#2}}
\end{center}
\caption{#4}\label{fig:#1} 
\end{figure}
}
\begin{document}

\maketitle

% REQUIRED
\begin{abstract}
We discuss computing with hierarchies of families of (potentially weighted) semiclassical Jacobi polynomials which arise in the construction of multivariate orthogonal polynomials. In particular, we outline how to build connection and differentiation matrices with optimal complexity and compute analysis and synthesis operations in quasi-optimal complexity. We investigate a particular application of these results to constructing orthogonal polynomials in annuli, called the generalised Zernike annular polynomials, which lead to sparse discretisations of partial differential equations. We compare against a scaled-and-shifted Chebyshev--Fourier series showing that in general the annular polynomials converge faster when approximating smooth functions and have better conditioning.  We also construct a sparse spectral element method by combining disk and annulus cells, which  is highly effective for solving PDEs with radially discontinuous variable coefficients and data.
\end{abstract}

% REQUIRED

% REQUIRED

\section{Introduction}
\label{sec:introduction}

Semiclassical Jacobi polynomials are univariate polynomials orthogonal with respect to the weight $x^a (1-x)^b (t-x)^c$ on $[0,1]$ where $t > 1$ and $a,b > -1$.  The semiclassical Jacobi polynomials are used to give explicit expressions for the so-called generalised Zernike annular polynomials. These are multivariate orthogonal polynomials (in $x$ and $y$) on the annulus $\Omega_\rho = \{(x,y) \in \mathbb{R}^2 : 0<\rho \leq r \leq 1\}$, where $r^2 = x^2 + y^2$, orthogonal with respect to a Jacobi-like inner product:
\begin{align}
\ip<f,g>_{\rho, (a,b)}=\iint_{\Omega_\rho} f(x,y) g(x,y) (1-r^2)^a(r^2-\rho^2)^b \, \dx \dy,
\label{eq:annulus-weight}
\end{align}
where $a,b > -1$.  We shall utilise the connection between semiclassical Jacobi polynomials and Zernike annular polynomials to introduce quasi-optimal complexity means for discretising and thereby solving partial differential equations in annuli. This requires the construction of operators associated with a hierarchy of semiclassical Jacobi polynomials where $c \in \{0, 1, \dots, C\}$.

Tatian \cite{Tatian1974} and Mahajan \cite{Mahajan1981} were the first to introduce the (non-generalised) Zernike annular polynomials  orthogonal with respect to the unweighted $L^2$-inner product, $a=b=0$, on the annulus. Sometimes referred to as Tatian--Zernike or fringe-Tatian polynomials, these have proven very popular in the optics and other communities, cf.~\cite{deWinter2020, deWinter2022, Maguire2023, Rolland2021}. Zernike annular polynomials also underpin the recently introduced gyroscopic orthogonal polynomials \cite{Ellison2023}, which are used for solving PDEs in cylinders of varying heights. In all the aforementioned works,  the synthesis, analysis, connection and differentiation operators  are constructed via the Christoffel--Darboux formula combined with quadrature, which does not achieve optimal complexity. More specifically, using Christoffel--Darboux, which is equivalent to the Cholesky factorisation technique described in \cref{sec:semiclassicalJacobi}, they calculate $N$ Jacobi matrices in $\mathcal{O}(N^2)$ flops. For our applications this is numerically unstable when $N \gg 0$ and alternative methods via QR factorisations are preferable cf.~\cref{sec:semiclassicalJacobi}. They also calculate the Laplacian and identity operators in $\mathcal{O}(N^3)$ flops as opposed to our methods which require $\mathcal{O}(N^2)$ flops.

By constructing multivariate orthogonal polynomials with respect to the non-uniform weight in \cref{eq:annulus-weight}, one may build sparse spectral methods to solve partial differential equations (PDEs) on such regimes. In \cref{sec:examples}, we focus on using  these multivariate orthogonal polynomials to solve PDEs on disks and annuli via a \emph{sparse} spectral method and compare against a method based on the Chebyshev--Fourier series, though the results extend to the setting of gyroscopic orthogonal polynomials. We also construct a spectral element method for problems with radial direction discontinuities in the variable coefficients and right-hand side, as exemplified in \cref{fig:spectral-element-plots}. The boundary conditions and continuity across cells are enforced via a tau-method \cite{Lanczos1938, Burns2020}.

\begin{figure}[h!]
\centering
\subfloat[Right-hand side $f_1$ \cref{eq:f-element1}.]{\includegraphics[width =0.32 \textwidth]{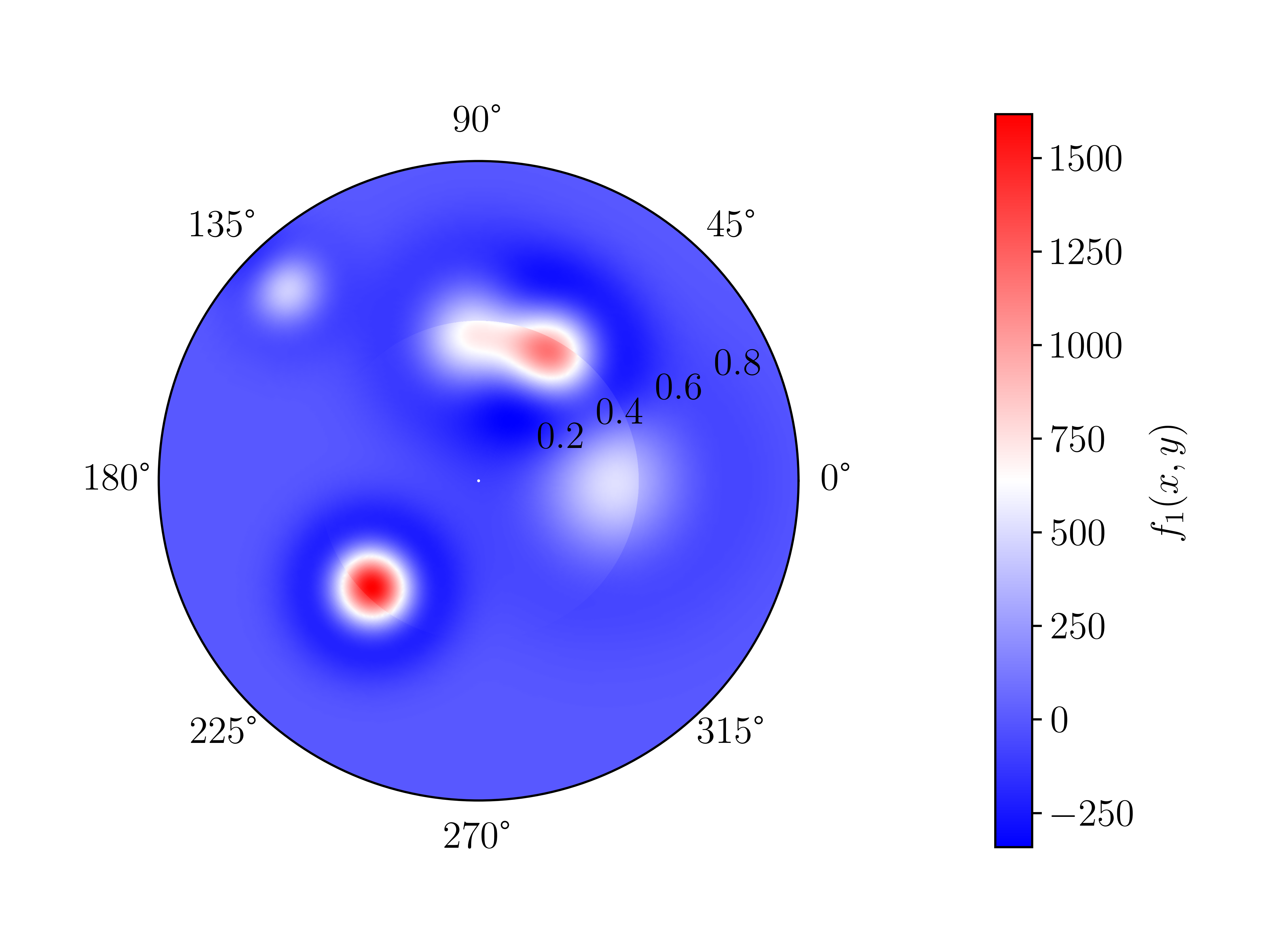}}
\subfloat[Right-hand side $f_2$ \cref{eq:f-element2}.]{\includegraphics[width =0.32 \textwidth]{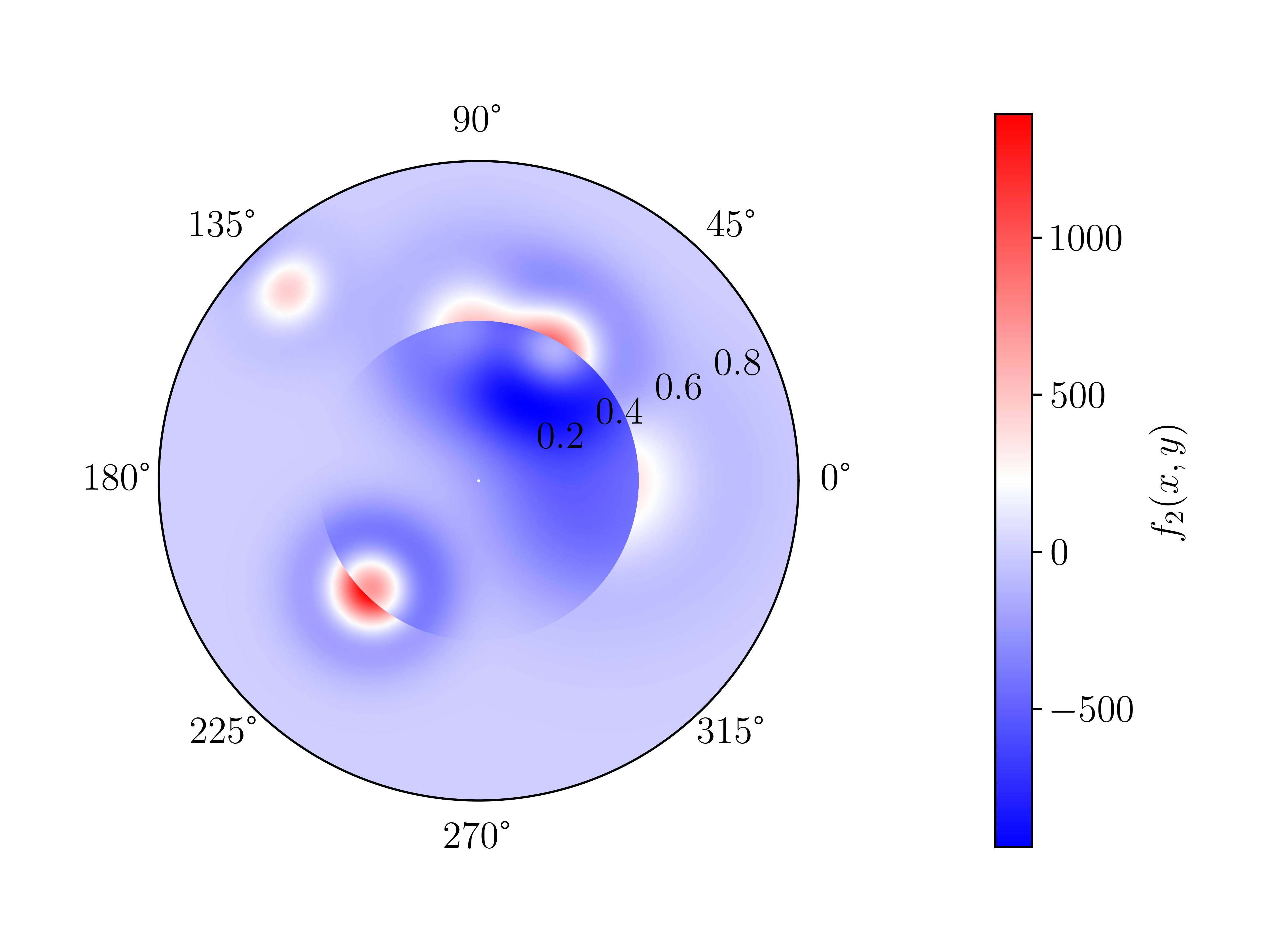}}
\subfloat[Solution $u$ \cref{eq:u-element}.]{\includegraphics[width =0.32 \textwidth]{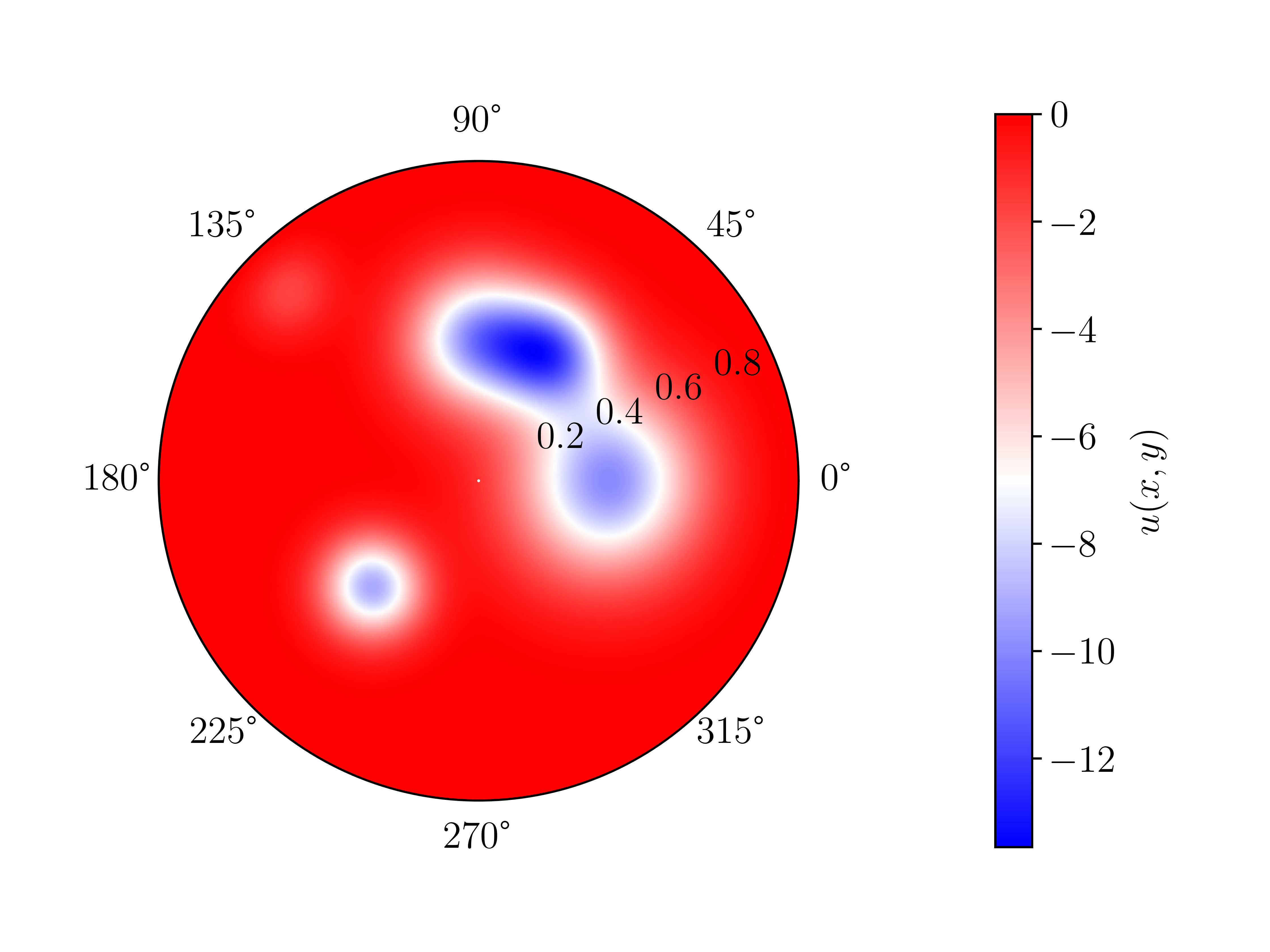}}
\caption{Plots of the right-hand sides with discontinuities in the radial direction and the corresponding solution of \cref{eq:spectral-poisson}--\cref{eq:spectral-helmholtz}. The setup of the problem is given in \cref{sec:disc-data}. These functions are resolved to machine precision utilising the spectral element method described in \cref{sec:spectral-element} based on Zernike and Zernike annular polynomials.} \label{fig:spectral-element-plots}
\end{figure}

Spectral methods for the disk/ball \cite{Boyd2011, VasilDisk, Wilber2017,Boulle2020,Meyer2021,Li2014,Atkinson2019,Ellison2022} and sphere \cite{Vasil2019, Lecoanet2019, Townsend2016} have been thoroughly studied. On the annulus many methods utilise variations of a Chebyshev--Fourier series, e.g.~\cite{Molina2020}. Barakat \cite{barakat1980optimum} constructs a basis for the annulus but it is not polynomial in Cartesian coordinates.  In 2011, Boyd and Yu \cite{Boyd2011} compared seven spectral methods for a disk. Out of the seven, they found that Zernike polynomials and Chebyshev--Fourier series had the best approximation properties where, for non-trivial examples, Zernike polynomials required half the degrees of freedom of the next best method. This motivates the construction of multivariate orthogonal polynomials for domains of interest. Until recently, a common critique of Zernike polynomials was that no fast transform for the radial direction was known \cite[Sec.~6.1]{Boyd2011}. However, this was resolved by Slevinsky \cite{Slevinsky2019, FastTransforms}, see also \cite{Olver2020}. This technique was extended to generalised Zernike annular polynomials in \cite[Sec.~4.4]{Gutleb2023}.

 Spectral element methods for the disk, where the cells are an inner disk and concentric annuli, have been previously studied. For instance, the spectral method software \texttt{Dedalus} \cite{Burns2020} was recently utilized to construct such a discretisation in the context of fully compressible magnetohydrodynamics \cite{anders2023}. Their discretisation uses a Zernike polynomial basis for the disk cell and the equivalent of a Chebyshev--Fourier series for the annuli cells. Continuity is enforced via a tau-method. We compare our spectral element method with an almost equivalent discretisation as the one considered in \cite{anders2023} in \cref{sec:examples}. 

%As an aside, the extension of a quadratic argument Chebyshev--Fourier series to the annulus is not obvious. In \cref{sec:annulusOPs}, we argue for the use of the two-band polynomials which are orthogonal to the weight $(1-|x|^2)^a(|x|^2-\rho^2)^b |x|^{2c}$, $a,b > -1$, and $c \in \mathbb{R}$ on $\rho \leq |x| \leq 1$. One may then represent functions defined on an annulus via a two-band--Fourier series.

%: Zernike polynomials, Logan--Shepp ridge polynomials, Chebyshev--Fourier series, cylindrical Robert functions, Bessel--Fourier expansion, square-to-disk conformal mappings, and radial basis functions.
%They noted that cylindrical Robert functions are severely ill-conditioned \cite[Sec.~7]{Boyd2011}, a Bessel--Fourier series is usually not spectrally-accurate \cite[Sec.~8]{Boyd2011}, square-to-disk conformal mappings artificially cluster points in four regions associated with the corners of the square \cite[Fig.~10]{Boyd2011}, and Logan--Shepp ridge polynomials do not enjoy the use of the fast Fourier transform (FFT). 

%The scaled-and-shifted Chebyshev--Fourier series clusters points closer to the origin \cite{Boyd2011}. In many functions of interest defined on a disk, a singularity will occur closer to the outer boundary (since the area of the disk grows at a rate of $r^2$). A quadratic-argument version of the Chebyshev--Fourier series better accounts for such occurrences. On the annulus, solutions often feature a logarithmic component and methods that cluster points closer to the origin may be advantageous.

Note that as Zernike (annular) polynomials and the Chebyshev--Fourier series feature a Fourier mode component, any operator that commutes with rotations (such as the Laplacian or identity operator) decouples across the Fourier modes. That is, we can decompose a two-dimensional PDE solve on the annulus into $2M+1$ one-dimensional solves where $M$ is the highest Fourier mode considered. For each one-dimensional solve, the scaled-and-shifted Chebyshev--Fourier series results in an almost-banded matrix. Aside from the two dense rows associated with the boundary conditions, one recovers matrices with a bandwidth of five and nine for the Poisson and Helmholtz equations, respectively.  By contrast the Zernike annular polynomials result in a tridiagonal and a pentadiagonal matrix per Fourier mode for the Poisson and Helmholtz equations, respectively, as depicted in \cref{fig:spy-poisson}: a smaller bandwidth.  Furthermore, we note that the Laplacian matrix for the Zernike annular polynomials is better conditioned than the Laplacian matrix of the Chebyshev--Fourier series, for increasing Fourier mode $m$, as plotted in \cref{fig:poisson-conditioning}.

Due to the decoupling of the considered PDEs across Fourier modes, we require $\mathcal{O}(N^2)$ for the linear system solve. Thus the overall complexity, from approximation of the right-hand side through to evaluating the approximated solution on a grid is $\mathcal{O}(N^2 \log(N))$.  We note that Bremer \cite{bremer2020} designed an algorithm that solves the two-dimensional variable coefficient Helmholtz with radial symmetry in $\mathcal{O}(k \log k)$ complexity, where $k$ is the wavenumber. The scheme is for problems posed on $\mathbb{R}^2$ but can be extended to disk and annuli domains. Although we do not achieve the same complexity for the class of equations considered by Bremer, our method can be extended to general variable coefficients, albeit at a higher computational complexity, and is also amendable to more general classes of equations. 
% The two-band--Fourier series results in a pentadiagonal and nonadiagonal matrix for the Laplacian and Helmholtz operators, respectively.

%Smaller bandwidths, in particular a symmetric tridiagonal Laplacian, results in many favourable benefits. For instance, spectral decompositions of tridiagonal matrices are easier to compute e.g.~via the QR, divide-and-conquer, and the MRRR algorithms. Moreover, matrix irreducibility is easier to detect. 

\begin{figure}[h!]
\centering
\subfloat[All modes.]{\includegraphics[width =0.24 \textwidth]{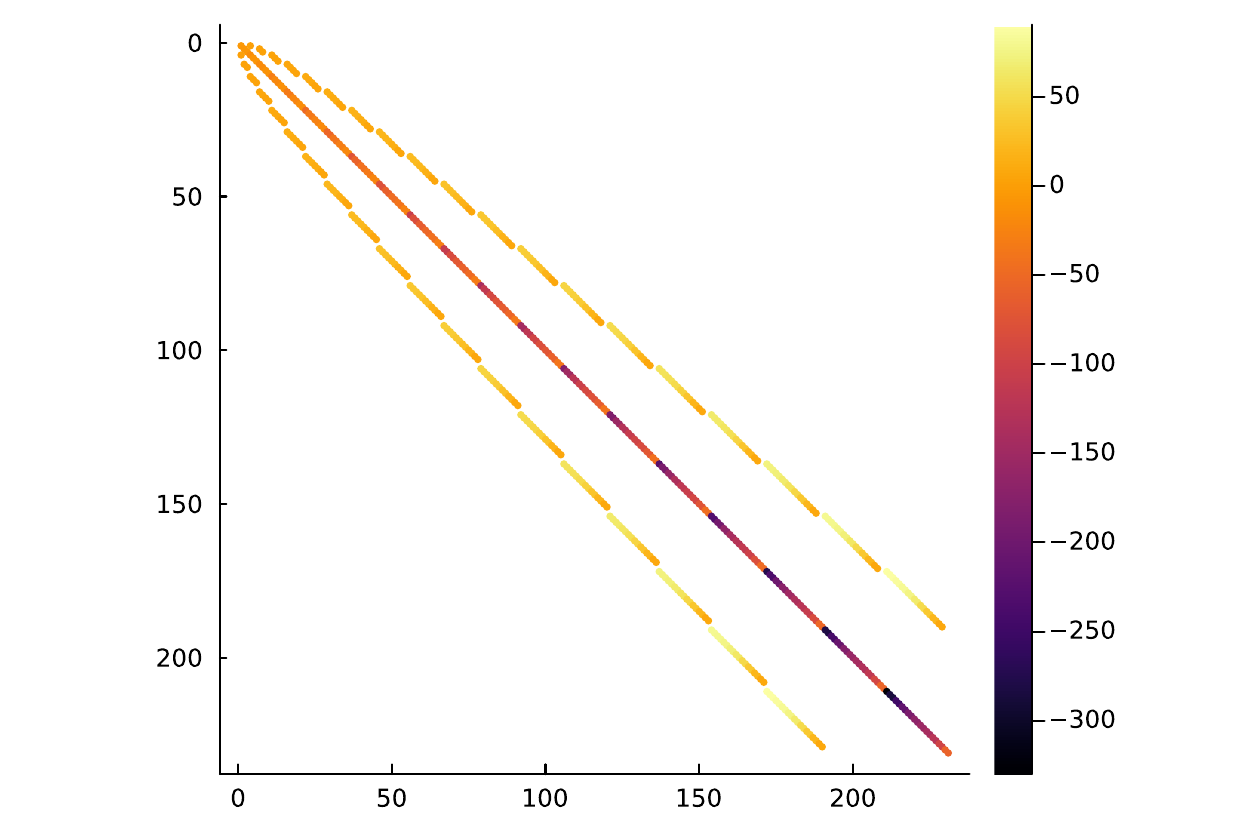}}
\subfloat[$0^{\mathrm{th}}$-Fourier mode.]{\includegraphics[width =0.24 \textwidth]{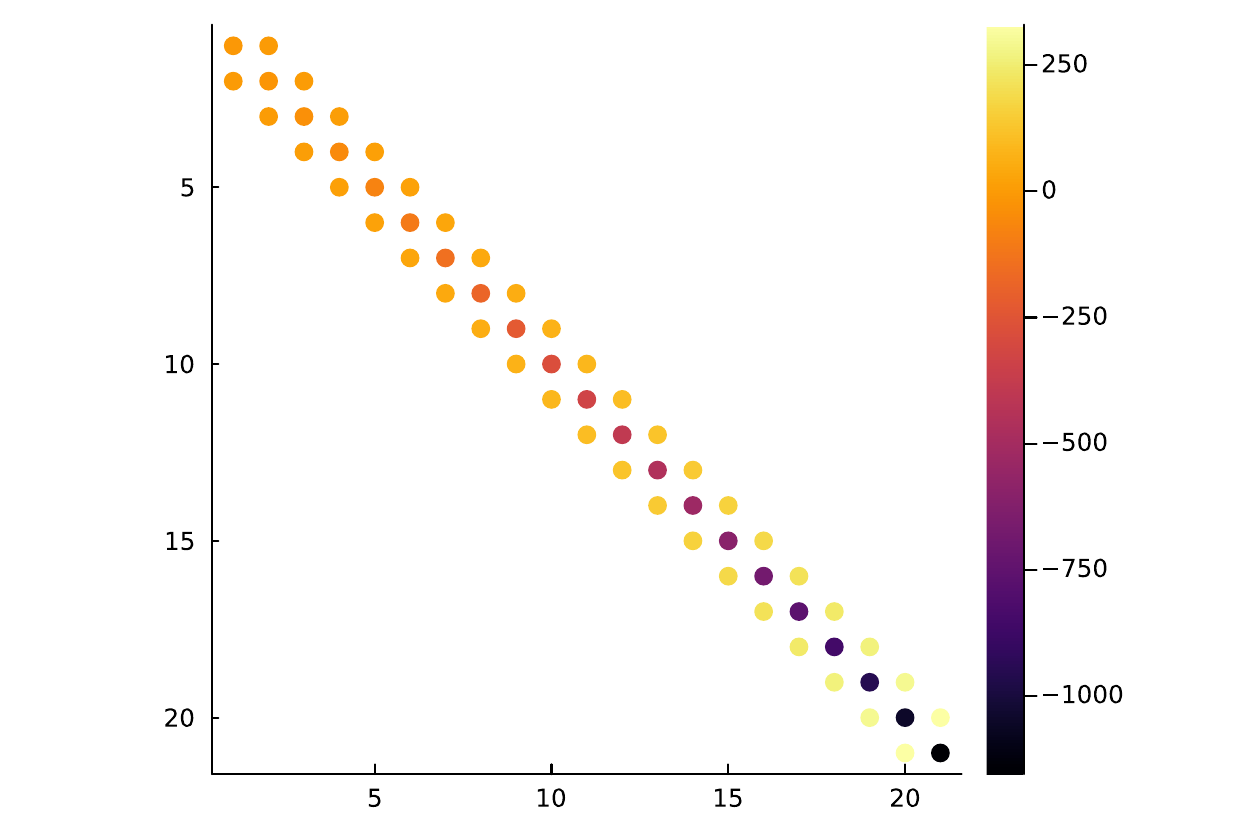}}
\subfloat[All modes.]{\includegraphics[width =0.24 \textwidth]{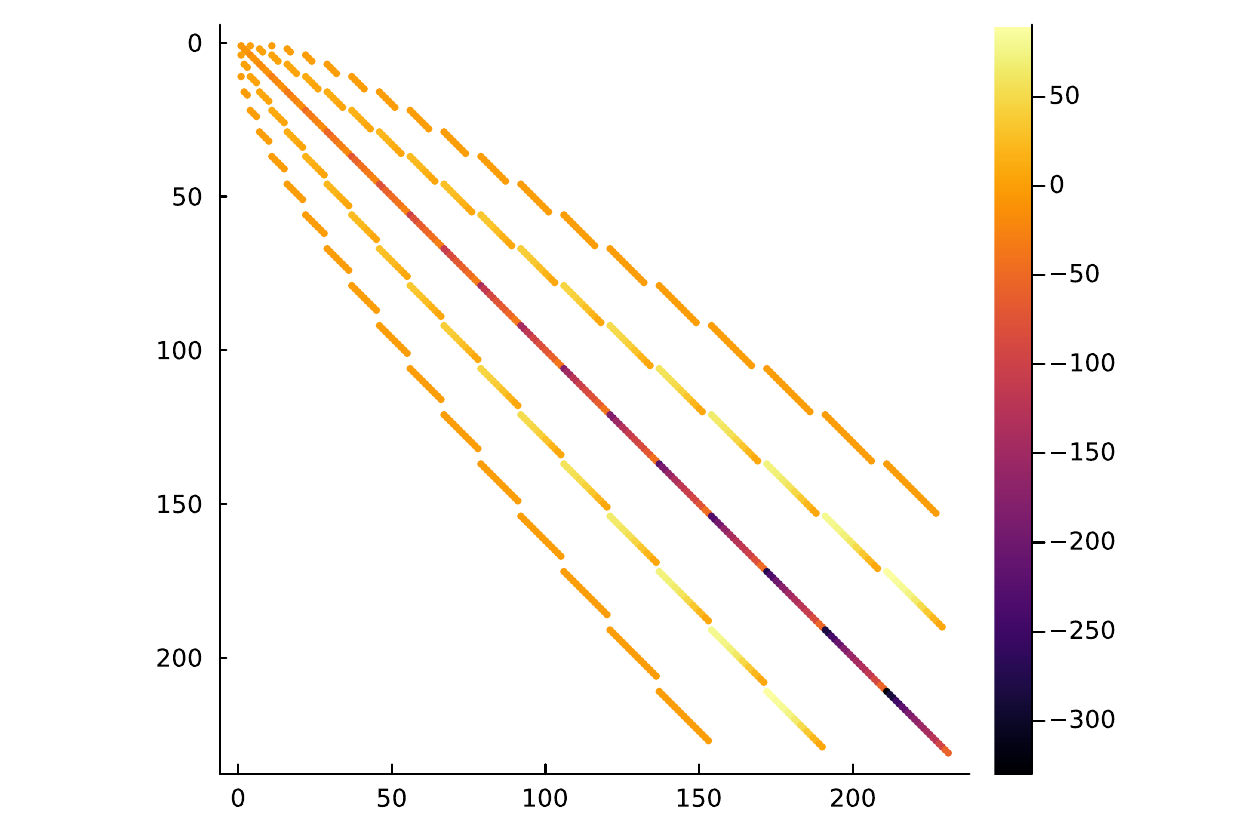}}
\subfloat[$0^{\mathrm{th}}$-Fourier mode.]{\includegraphics[width =0.24 \textwidth]{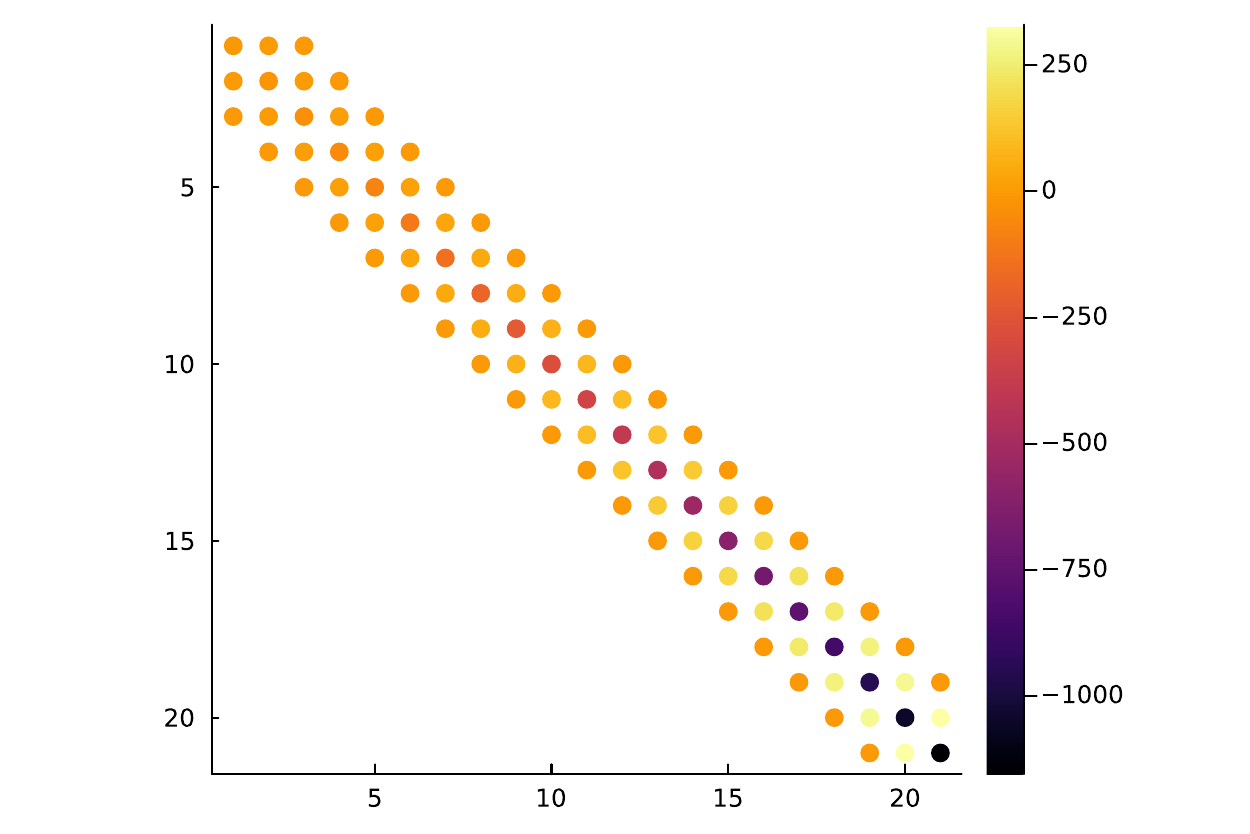}}
\caption{\texttt{Spy} plots of the matrices after discretising the Laplacian, $\Delta$, ((a) and (b)) and the Helmholtz operator $(\Delta + \mathcal{I})$ ((c) and (d)) on the 2D annulus $\Omega_{1/2}$ truncated at polynomial degree 20 with generalised Zernike annular polynomials. (a) and (c) show all the Fourier modes, whereas (b) and (d) show the matrix when one decouples the Fourier modes and focuses on the $m=0$ mode in isolation.}\label{fig:spy-poisson}
\end{figure}
%\begin{figure}[h!]
%\centering
%
%\caption{\texttt{Spy} plots of the matrices after discretising the operator $(-\Delta + \mathcal{I})$ on the 2D annulus $\Omega_{1/2}$ truncated at polynomial degree 19. The Helmholtz operator is pentadiagonal.}\label{fig:spy-helmholtz}
%\end{figure}

\begin{figure}[h!]
\centering
\subfloat[Condition number.]{\includegraphics[width =0.35\textwidth]{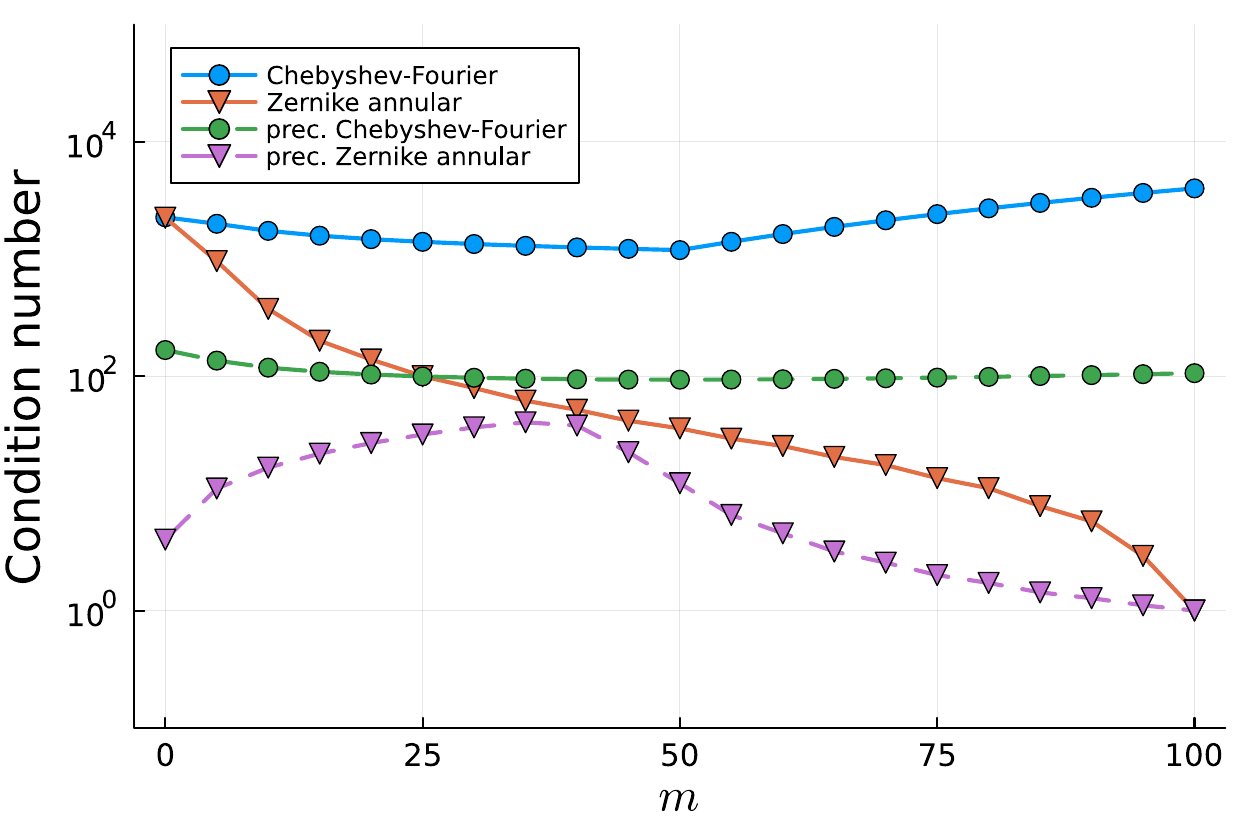}}
\subfloat[Matrix size.]{\includegraphics[width =0.35\textwidth]{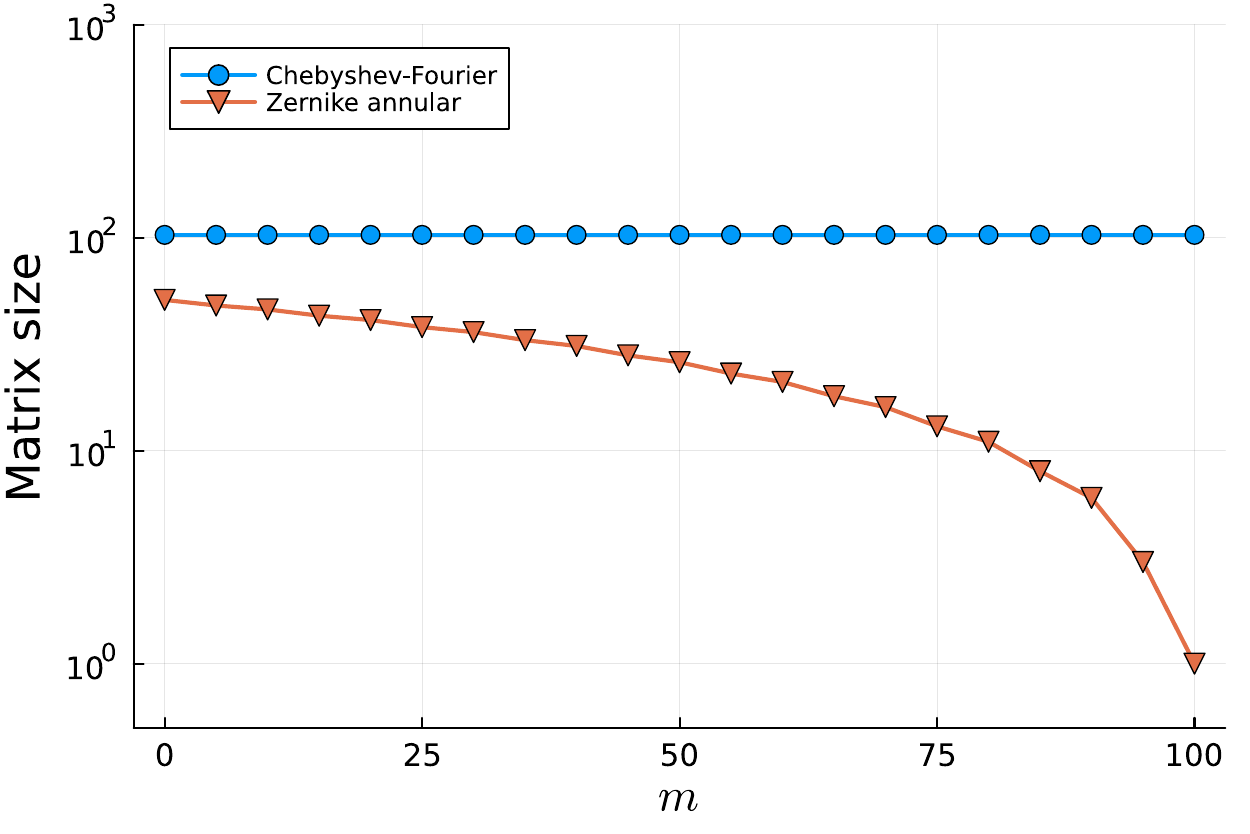}}
\caption{The condition number and matrix size of the Laplacian matrix, for increasing Fourier mode $m$, induced by the Chebyshev--Fourier series and the Zernike annular polynomials. We also consider the Laplacian matrix preconditioned with the inverse of the diagonal of the original matrix. We observe that both the original and preconditioned Laplacian matrices for the Zernike annular polynomials are better conditioned and smaller in size.}\label{fig:poisson-conditioning}
\end{figure}

We describe the scaled-and-shifted Chebyshev--Fourier series in \cref{sec:tensorproduct}. In \cref{sec:semiclassicalJacobi} we introduce the semiclassical Jacobi polynomials and discuss factorisation techniques for deriving the relationships between different family parameters. In \cref{sec:annulusOPs}, we introduce the generalised Zernike annular polynomials and derive various relationships between families of parameters that may be computed with optimal complexity. We outline the sparse spectral methods for solving the Helmholtz equation in \cref{sec:pdes}. In \cref{sec:spectral-element} we construct a spectral element method where the cells are an inner disk and concentric annuli. Finally, in \cref{sec:examples} we consider various PDE examples and compare the spectral methods. 
%In \cref{sec:weighted-jacobi}, we construct the weighted Jacobi--Fourier series which may be used to discretise a modified weak form of the Helmholtz equation. 
%By utilising a Cholesky factorisation we show how these orthogonal polynomials may be used to construct Sobolev orthogonal polynomials in \cref{sec:sobolev-ops}. 

\begin{remark}[Weak formulation]
\label{rem:weak}
A major benefit of the generalised Zernike (annular) polynomials is that they are easily amendable to discretising PDEs posed in weak form. This allows one to construct sparse spectral element methods for variational problems. A thorough construction of such a method is beyond the scope of this work, however, we make some comments in \cref{sec:annulusOPs}.
\end{remark}

\begin{remark}
Zernike annular polynomials also allow one to construct random functions on the annulus via the methods of Filip et al.~\cite{Filip2019}.
\end{remark}

\begin{remark}[Generalisation to 3D]
Just as generalised Zernike polynomials on the disk can be extended to the ball \cite[Prop.~5.2.1]{dunkl2014orthogonal},  orthogonal polynomials on annuli naturally extend to higher dimensional spherical shells with spherical harmonics in place of Fourier modes. For simplicity we restrict our attention to 2D.
\end{remark}

\begin{remark}[Generalisation to vector-valued PDEs]
For brevity we focus on scalar PDEs. Whilst it is possible to use a scalar basis of orthogonal polynomials on annuli in each component of a vector-valued problem, this will not automatically capture rotational invariance and hence the complexity will be sub-optimal. It is likely the construction in \cite{VasilDisk} can be extended to annuli, {\it mutatis mutandi}, by an appropriate replacement of the Jacobi polynomials with semiclassical Jacobi polynomials.
\end{remark}

\section{Scaled-and-shifted Chebyshev--Fourier series}
\label{sec:tensorproduct}

The Chebyshev--Fourier series is constructed via a tensor product of Chebyshev polynomials of the first kind (denoted $T_n(x)$, $n\geq 0$) \cite[Sec.~18.3]{dlmf} in the radial direction with the Fourier series in the angular direction. 
%On a disk, the Chebyshev polynomials take the form:
%\begin{enumerate}
%\itemsep=0pt
%\item Scaled-and-shifted: $T_n(2r-1)$ on $r \in [0,1]$;
%\item Quadratic argument: $T_n(2r^2-1)$ on $r \in [-1,1]$.
%\end{enumerate}

Following \cite{Olver2020}, we use quasimatrix notation, i.e.
\begin{align}
{\bf T}(x) \coloneqq \bvect[T_0(x), T_1(x), T_2(x), \dots].
\end{align}
This is convenient for expressing recurrence relationships.  For example, the three term recurrence is expressed as
\begin{align}
x {\bf T}(x) = {\bf T}(x) 
\begin{pmatrix}
0 & 1/2 & & & \\
1 & 0 & 1/2 & & \\
& 1/2 & 0 & 1/2 &\\
& & \ddots & \ddots & \ddots
\end{pmatrix}.
\end{align}
Consider the annulus $\Omega_\rho = \{ (x,y) \in \mathbb{R}^2: 0< \rho \leq |(x,y)| \leq 1\}$. Define
\begin{align}
r_\rho = \frac{2}{1-\rho}\left(r-\frac{1+\rho}{2}\right). \label{eq:r-affine}
\end{align} 
Let $F_{m,j}$ denote the Fourier series at mode $m$ where $j \in \{0,1\}$ if $m \in \{1, 2, \dots\}$, else $j=1$ if $m=0$. For clarity, $F_{0,1}(\theta) = 1$ and for $m \geq 1$,
\begin{align}
F_{m,j}(\theta) \coloneqq
\begin{cases}
\sin(m\theta) & \text{if} \; j =0,\\
\cos(m\theta) & \text{if} \; j = 1.
\end{cases}
\end{align} 
Then the scaled-and-shifted Chebyshev--Fourier series on the annulus is given by the tensor product:
%\begin{enumerate}
%\item ${\bf T}\left(\frac{2}{1-\rho}(r-\frac{1+\rho}{2})\right)\times{\bf F}(\theta) \\ \coloneqq \left\{T_n\left(\frac{2}{1-\rho}(r-\frac{1+\rho}{2})\right)F_{m,j}(\theta) : n \in \mathbb{N}_0, m = 0,1,\dots, \floor{n/2}, j = 0,1 \right\}$; 
%\item ${\bf T}\left(\frac{2}{1-\rho^2}(r^2-\frac{1+\rho^2}{2}) \right)\times{\bf F}(\theta)$. 
%\end{enumerate}
\begin{align}
\begin{split}
&{\bf T}(r_\rho) \otimes {\bf F}(\theta)  \\
%\coloneqq \left\{T_n(r_\rho)F_{m,j}(\theta) : n \in \mathbb{N}_0, (m,j) \in (\mathbb{N}_0 \times \{0,1\}) \backslash \{(0,1)\} \right\}.
&  \coloneqq \left(T_0(r_\rho) F_{0,1}(\theta) \, | \, T_1(r_\rho) F_{0,1}(\theta) \, | \, \cdots \, |\, T_0(r_\rho)F_{1,0}(\theta) \,  | \, \cdots \, | \, T_{2n}(r_\rho) F_{n,0}(\theta) \, | \, \cdots \right).
\end{split}
\label{def:tensor-product}
\end{align} 
%In the case of a disk, where $\rho=0$, it is well documented that the tensor product clusters grid points near $r=0$ \cite[Sec.~6]{Boyd2011} and thus may capture singularities near the origin. However, a random singularity is more likely to occur near the boundary since the area of the disk grows at a rate of $r^2$. Thus it is often recommended to consider a quadratic argument in the ChebyshevT polynomial \cite[Sec.~6]{Boyd2011}, i.e.
%\begin{align}
%{\bf T}(2r^2-1) \times {\bf F}(\theta), \;\; r \in [-1,1].  
%\end{align}
%We propose an extension of the quadratic argument to the annulus in \cref{sec:annulusOPs:tensor-product}.
An expansion in this basis takes the form:
\begin{align*}
u(r,\theta) &= 
\sum_{n=0}^\infty u_{n, 0,  1 } T_n(r_\rho) + \sum_{n=0}^\infty \sum_{m=1}^\infty \left[ u_{n, m, 0} T_n(r_\rho) \sin(m\theta) + u_{n, m, 1} T_n(r_\rho) \cos(m\theta) \right]\\
%& = \left[{\bf T}(r_\rho) \otimes {\bf F}(\theta)\right] 
%\begin{pmatrix}
%u_{0,0,0}\\
%u_{1,0,0}\\
%\vdots\\
%u_{2n, m, j}\\
%\vdots
%\end{pmatrix}
& = \left[{\bf T}(r_\rho) \otimes {\bf F}(\theta)\right] 
\begin{pmatrix}
u_{0,0, 1 }&
u_{1,0,0}&
\cdots&
u_{2n, m, j}&
\cdots
\end{pmatrix}^\top
= \left[{\bf T}(r_\rho) \otimes {\bf F}(\theta)\right]  {\bf u}.
\end{align*}
If we rearrange the coefficients in ${\bf u}$ into a matrix such that
\begin{align}
U = 
\begin{pmatrix}
u_{0,0,1} & u_{0,1,0} & u_{0,1,1} & u_{0,2,0} & \cdots\\
u_{1,0,1} & u_{1,1,0} & u_{1,1,1} & u_{1,2,0} & \cdots\\
\vdots & \vdots & \vdots &\vdots & \ddots
\end{pmatrix},
\end{align}
then $\left[{\bf T}(r_\rho) \otimes {\bf F}(\theta)\right]  {\bf u} = {\bf T}(r_\rho) U {\bf F}(\theta)^\top$. The choice of truncation degree $N$ and Fourier mode $M$ are independent and may be custom chosen for each problem individually. In this work we pick $M = N$ resulting in the truncated coefficient matrix $U$ of size $(N+1) \times (2N+1)$.

\section{Semiclassical Jacobi polynomials}
\label{sec:semiclassicalJacobi}

We denote the orthonormalised Jacobi polynomials by $P^{(a,b)}_n(x)$, $n\geq 0$ which are orthonormal with respect to the weight  $(1-x)^a (1+x)^b$ such that $a, b >-1$ \cite[Sec.~18.3]{dlmf}.  The term \emph{weighted Jacobi polynomials} refers to the Jacobi polynomial $P_n^{(a,b)}(x)$ multiplied by its weight, i.e.~$W^{(a,b)}_n(x) \coloneqq (1-x)^a (1+x)^b P_n^{(a,b)}(x)$.

The building blocks for the Zernike annular polynomials are the so-called semiclassical Jacobi polynomials. Recall that these are univariate orthogonal polynomials on the interval $(0,1)$ with respect to the inner product
\begin{align}
\int_0^1 f(x) g(x) x^a (1-x)^b (t-x)^c \D x,
\end{align}
where $t > 1$ and $a,b > -1$. First introduced by Magnus\footnote{The weight considered by Magnus is the slightly different $(1-x)^a x^b (t-x)^c$, which has the unfortunate property that the singularities are neither listed left-to-right or right-to-left. We have changed the order to be left-to-right.} in \cite[Sec.~5]{Magnus1995}, we denote the orthonormal semiclassical Jacobi polynomials as
$$
\Q_n(x) = k_n x^n + O(x^{n-1}),
$$
where, for concreteness, $k_n > 0$. Note that, when $c = 0$, these become scaled-and-shifted orthonormalised Jacobi polynomials and we drop the $t$ dependence. That is, we have for any $t$,
\begin{align}
Q_n^{t,(a,b,0)}(x) = (-1)^{n} 2^{(a+b+1)/2}  P_n^{(a,b)}(1-2x) = 2^{(a+b+1)/2} P_n^{(b,a)}(2x-1).
\label{eq:scaling}
\end{align}
Many beautiful connections between semiclassical Jacobi polynomials and Painlev\'e equations are known, and we refer the readers to \cite{Magnus1995}.

\begin{remark}
The semiclassical Jacobi polynomials are a special case of the generalised Jacobi polynomials found in \cite[Sec.~3]{Ellison2023}. 
\end{remark}

We require recurrence relationships between families of the Zernike annular polynomials in the construction of spectral methods. These recurrence relationships are derived from recurrence relationships that relate hierarchies of semiclassical Jacobi polynomial families. We denote the connection matrix mapping the parameter family from $(\alpha,\beta,\gamma)$ to $(a,b,c)$ by $R^{t,(\alpha,\beta,\gamma)}_{(a,b,c)}$. To indicate that maps are from a (partially) weighted family to another (partially) weighted family, we utilise the Roman script $\mathrm{a}$, $\mathrm{b}$, and $\mathrm{c}$ or combinations thereof. For instance: 
\begin{align}
\begin{split}
\bQ(x) &= {\bf Q}^{t,(\alpha,\beta,\gamma)}(x) R^{t,(\alpha,\beta,\gamma)}_{(a,b,c)},\\
x^a \bQ &= x^\alpha {\bf Q}^{t,(\alpha,\beta,\gamma)} R^{t,(\alpha,\beta,\gamma)}_{\mathrm{a}, (a,b,c)},\\
%x^a  (1-x)^b \bQ &=x^\alpha (1-x)^\beta {\bf Q}^{t,(\alpha,\beta,\gamma)} R^{t,(\alpha,\beta,\gamma)}_{\mathrm{ab}, (a,b,c)},\\
x^a (1-x)^b (t-x)^c \bQ(x) &=x^\alpha (1-x)^\beta (t-x)^\gamma {\bf Q}^{t,(\alpha,\beta,\gamma)}(x) R^{t,(\alpha,\beta,\gamma)}_{\mathrm{abc}, (a,b,c)}.
\end{split}\label{eq:connection}
\end{align}
Similarly, we denote the differentiation connection matrix between families of semiclassical Jacobi polynomials by $D^{t,(\alpha,\beta,\gamma)}_{(a,b,c)}$, e.g.
\begin{align}
\begin{split}
\fdx \bQ(x) &= {\bf Q}^{t,(\alpha,\beta,\gamma)}(x) D^{t,(\alpha,\beta,\gamma)}_{(a,b,c)},\\
\fdx x^a \bQ &= x^\alpha {\bf Q}^{t,(\alpha,\beta,\gamma)} D^{t,(\alpha,\beta,\gamma)}_{\mathrm{a}, (a,b,c)},\\
%\fdx x^a  (1-x)^b \bQ &=x^\alpha (1-x)^\beta {\bf Q}^{t,(\alpha,\beta,\gamma)} D^{t,(\alpha,\beta,\gamma)}_{\mathrm{ab}, (a,b,c)},\\
\fdx [x^a (1-x)^b (t-x)^c \bQ(x)] &=x^\alpha (1-x)^\beta (t-x)^\gamma {\bf Q}^{t,(\alpha,\beta,\gamma)}(x) D^{t,(\alpha,\beta,\gamma)}_{\mathrm{abc}, (a,b,c)}.
\end{split}\label{eq:diff}
\end{align}

%In recent work by Gutleb et al.~\cite{Gutleb2023}, a systematic technique was introduced for computing the entries of $R^{t,(\alpha,\beta,\gamma)}_{\boldsymbol{\cdot}, (a,b,c)}$ and $D^{t,(\alpha,\beta,\gamma)}_{\boldsymbol{\cdot}, (a,b,c)}$ based on the QR and Cholesky factorisations. 

There exist choices of $(\alpha, \beta, \gamma)$ and $(a,b,c)$ such that the matrices $R^{t,(\alpha,\beta,\gamma)}_{\boldsymbol{\cdot}, (a,b,c)}$ and $D^{t,(\alpha,\beta,\gamma)}_{\boldsymbol{\cdot}, (a,b,c)}$ are sparse. In the next three subsections, we introduce the factorisation techniques that allow one to apply analysis and synthesis operators with quasi-optimal complexity as well as compute Jacobi, connection and differentiation matrices.

\subsection{Connection matrices}

Given the quasimatrix of a semiclassical Jacobi family evaluated at a point, ${\bf Q}^{t,(a,b,c)}(x)$, the goal is to find the (truncated) infinite-dimensional matrix of coefficients in  \cref{eq:connection} that relates the evaluation to a different semiclassical Jacobi family evaluated at the same point.  In recent work by Gutleb et al.~\cite{Gutleb2023}, the authors analyse six infinite-dimensional matrix factorisations for this purpose. The main contribution of their work is summarised in \cite[Tab.~1]{Gutleb2023}. We replicate the relevant results in \cref{th:factorisations}.

\begin{theorem}[Tab.~1 in \cite{Gutleb2023}]
\label{th:factorisations}
Let $w(x)$ denote a nonnegative and bounded weight on $\mathbb{R}$ and consider the measure $\mathrm{d}\mu$ such that both $\mathrm{d}\mu$ and $w(x) \mathrm{d}\mu$ are positive Borel measures on the real line whose support contains an infinite number of points and has finite moments.  Suppose that ${\bf P}(x)$ and ${\bf Q}(x)$ are the quasimatrices of two families of orthonormal polynomials with respect to $\mathrm{d}\mu$ and $w(x) \mathrm{d}\mu$, respectively. Let $X_P$ denote the Jacobi matrix of ${\bf P}(x)$, i.e. $x {\bf P}(x) = {\bf P}(x) X_P$. Then
\begin{align}
{\bf P}(x) = {\bf Q}(x) R \;\; \text{and} \;\; w(x) {\bf Q}(x) = {\bf P}(x) R^\top \;\; \text{where} \;\; w(X_P) = R^\top R.
\label{eq:cholesky}
\end{align}
Here $R$ is upper-triangular, i.e.~$R^\top R$ is a Cholesky factorisation. Moreover, if $w(x)$ is a polynomial, then $w(X_P)$ and $R$ are banded. Furthermore, if $\sqrt{w}(x)$ is also a polynomial, then
\begin{align}
{\bf P}(x) = {\bf Q}(x) R \;\; \text{and} \;\; \sqrt{w}(x) {\bf Q}(x) = {\bf P}(x) Q \;\; \text{where} \;\; \sqrt{w}(X_P) = QR.
\label{eq:QR}
\end{align}
Here $Q$ is orthogonal and $R$ is upper-triangular, i.e.~they are the factors of a positive-phase QR-factorisation (and are unique). 
\end{theorem}

%\begin{table}[ht]
%\small
%\centering
%\begin{tabular}{ll}
%\toprule
%Factorisations & Connections \\ \midrule
%$U = R^\top R$ & ${\bf P}(x) = {\bf Q}(x) R$ \newline $u(x) {\bf Q}(x) = {\bf P}(x) R^\top$\\
%$\sqrt{U} = QR$ & ${\bf P}(x) = {\bf Q}(x) R$ \newline $\sqrt{u(x)} {\bf Q}(x) = {\bf P}(x) Q$\\
%\bottomrule
%\end{tabular}
%\caption{.} 
%\label{tab:factorisations}
%\end{table}

Suppose that $\alpha- a \geq 0$, $\beta - b \geq 0$, $\gamma - c \geq 0$ are integers. Let $X_{t,(a,b,c)}$ denote the Jacobi matrix of the orthormalised semiclassical Jacobi quasimatrix ${\bf Q}^{t,(a,b,c)}(x)$, i.e.
\begin{align*}
x {\bf Q}^{t,(a,b,c)}(x) = {\bf Q}^{t,(a,b,c)}(x) X_{t,(a,b,c)}.
\end{align*} 
Suppose one computes the following Cholesky factorisation
\begin{align}
X_{t,(a,b,c)}^{\alpha-a} (I-X_{t,(a,b,c)})^{\beta - b} (tI - X_{t,(a,b,c)})^{\gamma - c} = R^\top R,
\label{eq:chol1}
\end{align}
where $I$ is the identity matrix of conforming size. Then \cref{eq:cholesky} in \cref{th:factorisations} reveals that $R^{t,(\alpha,\beta,\gamma)}_{(a,b,c)} = R$ and $R^{t,(a,b,c)}_{\mathrm{abc},(\alpha,\beta,\gamma)} = R^\top$. Furthermore, we note that $R$ is upper-triangular with bandwidth $(\alpha-a)+(\beta-b)+(\gamma-c) + 1$. For more general choices of $\alpha, \beta, a, b > -1$, and $\gamma, c \in \mathbb{R}$, we refer the interested reader to \cite{Gutleb2023}. 

%Note that general rational modifications do not necessarily lead to banded conversion matrices.

%When considering analysis and synthesis operators for Zernike annular polynomials we shall require hierarchies of semiclassical Jacobi families where the $a$ and $b$ parameters are fixed but $c$ runs from zero to the truncation degree $N$ in steps of one. In such a case, we only increment $c$ in steps of one or two\sotodo{Not clear}, meaning the matrix in \cref{eq:chol1} (reducing to $tI - X_{t,(a,b,c)}$) remains numerically positive-definite.

When $\alpha \gg a$, $\beta \gg b$, or $\gamma \gg c$, the matrix on the left-hand side of \cref{eq:chol1} becomes increasingly less positive-definite and eventually a Cholesky factorisation fails due to numerical roundoff. In such a case, we recommend computing the connection matrix via an intermediate semiclassical Jacobi family. Alternatively, one may consider \cref{eq:QR} in \cref{th:factorisations}:
\begin{align}
{\bf Q}^{t, (a,b,c)}(x) = {\bf Q}^{t, (a,b,c+2)}(x) R, \;\;  (t-x){\bf Q}^{t, (a,b,c+2)}(x) = {\bf Q}^{t, (a,b,c)}(x) Q. \label{eq:qr1}
\end{align}

In summary ${\bf Q}^{t, (a,b,0)}(x) ={\bf Q}^{t, (a,b,1)}(x) R_1$ where $R_1^\top R_1 = tI - X_{t,(a,b,0)}$ and ${\bf Q}^{t, (a,b,0)}(x) ={\bf Q}^{t, (a,b,2)}(x) R_2$ where $QR_2= tI - X_{t,(a,b,0)}$. Connection matrices for higher values of $c$ may be computed by chaining together these factorisations.

\subsection{Jacobi matrices}

The goal is to compute the entries of the Jacobi matrix $X_{t,(a,b,c)}$ with $\mathcal{O}(cN)$ complexity, where $c \in \mathbb{N}$ and $N$ is the truncation degree. We first note that the main difficulty arises in the parameter $c$. If $c=0$ then explicit formulae for the Jacobi matrix exist. 

\begin{remark}
Given the base Jacobi matrix $X_{t,(a,b,c)}$ for any $c \in \mathbb{R}$, $c \geq 0$, then one may compute the Jacobi matrix $X_{t,(a,b,c+m)}$, for any $m \in \mathbb{N}$, with the following techniques. Thus we are not restricted to the case where $c \in \mathbb{N}$.
\end{remark}

The identities in \cref{eq:chol1} and \cref{eq:qr1} provide three alternatives for computing the Jacobi matrix when $c \neq 0$, $c \in \mathbb{N}$:
\begin{enumerate}
\itemsep=0pt
\item via the Cholesky factorisation;
\item via the $R$ in a QR factorisation;
\item via the $Q$ in a QR factorisation.
\end{enumerate}
\noindent \underline{Via Cholesky}. One may find the connection matrix between the case $c=0$ and any other choice of $c \in \mathbb{N}$ utilising \cref{eq:chol1}. If $c \in \mathbb{N}$ then the connection matrix is banded. We note that in principle one may compute the connection matrix to any $c \in \mathbb{N}$ in one step. However, one must find the Cholesky factorisation of $(tI - X_{t, (a,b,0)})^c$ which is numerically indefinite for even moderate values of $c$. Thus it is \emph{highly} recommended that one proceeds in steps of one for $c$. If $R^\top R =  t I-X_{t,(a,b,0)}$, then
\begin{align}\label{eq:choleskymethod}
\begin{split}
x  {\bf Q}^{t,(a,b,1)}(x) &= x  {\bf Q}^{t,(a,b,0)}(x)R^{-1}
= {\bf Q}^{t,(a,b,0)}(x) X_{t,(a,b,0)} R^{-1} \\
&=  {\bf Q}^{t,(a,b,1)}(x) R X_{t,(a,b,0)}R^{-1}.
\end{split}
\end{align}
Hence $X_{t,(a,b,1)} = R X_{t,(a,b,0)}R^{-1}$. When $c>1$, we increment the $c$ parameter in steps of one until we reach the desired Jacobi matrix. 

\noindent \underline{Via QR}. If one utilises a QR factorisation via \cref{eq:qr1} then it is possible to increment the $c$ parameter in steps of two. If $QR = t I-X_{t,(a,b,0)}$, then
\begin{align}
\begin{split}
x (t-x) {\bf Q}^{t,(a,b,2)}(x) 
& = x  {\bf Q}^{t,(a,b,0)}(x) Q = {\bf Q}^{t,(a,b,0)}(x) X_{t,(a,b,0)} Q \\
&= (t-x) {\bf Q}^{t,(a,b,2)}(x) Q^\top X_{t,(a,b,0)} Q.
\end{split}
\label{eq:analysis}
\end{align}
Hence, $X_{t,(a,b,2)} =  Q^\top X_{t,(a,b,0)} Q$. Alternatively, note that
\begin{align}\label{eq:qrmethodR}
\begin{split}
 {\bf Q}^{t,(a,b,0)}(x) X_{t,(a,b,0)} &= x {\bf Q}^{t,(a,b,0)}(x) = x {\bf Q}^{t,(a,b,2)}(x) R\\
&= {\bf Q}^{t,(a,b,2)}(x) X_{t,(a,b,2)} R =  {\bf Q}^{t,(a,b,0)}(x) R^{-1} X_{t,(a,b,2)} R.
\end{split}
\end{align}
Thus $X_{t,(a,b,2)} =R  X_{t,(a,b,0)}  R^{-1}$.

%Whereas in the Cholesky case it is ideal to compute the Jacobi matrix by incrementing the parameter in steps of one, the most stable method for computing the Jacobi matrix using the QR factorisation is to increment the $c$ parameter in steps of two. For an odd parameter choice $c$, one utilises a one-step Cholesky factorisation from ${\bf Q}^{t,(a,b,0)}$ to ${\bf Q}^{t,(a,b,1)}$. At this point, one switches to the QR-factorisation, incrementing in steps of two. However, we note that, unlike the Cholesky factorisation, the QR-method does not catastrophically fail if one were to to take a large step in the parameter $c$: QR factorisations do not require positive definite matrices to succeed. Hence, although not recommended, one could obtain Jacobi matrices for arbitrary integer parameters $c$ in a single QR factorisation (even $c$) or a one-step Cholesky factorisation and then one QR factorisation (odd $c$).

For either of the three routes, one computes the triple-matrix-product in $\mathcal{O}(N)$ complexity. See \cref{sec:complexity} for more details. Thus to compute the Jacobi matrix for an arbitrary parameter $c \in \mathbb{N}$ requires $\mathcal{O}(cN)$ complexity. As mentioned before, when building multivariate orthogonal polynomials, one typically requires a hierarchy of Jacobi matrices, one for each value of $c \in \{c_0, c_0+1, c_0+2, \dots\}$, $c_0 \geq 0$. Hence, in practice, these intermediate Jacobi matrices are required.

Out of these three approaches, the most theoretically stable method is the $Q$-factor variant of the QR factorisation as $\|Q\|_2=1$, whereas the upper triangular factors (and their inverses) in Cholesky and QR do not have unit $2$-norm.

\subsection{Analysis \& synthesis}
\label{sec:semiclassical-analysis}
The analysis and synthesis operators map a function to the coefficient vector of an expansion and vice versa. More precisely, let $\mathrm{d}\mu$ denote a positive Borel measure on the real line whose support contains an infinite number of points and has finite moments. Let $L^2(\mathbb{R}, \mathrm{d}\mu)$ denote the space of square integrable functions with respect to the measure $\mathrm{d}\mu$ and let $\ell^2$ denote the space of square summable sequences. Suppose that ${\bf P}(x)$ denotes a quasimatrix that is orthogonal with respect to $\mathrm{d}\mu$. Then the analysis operator $\mathcal{A}$ and synthesis operator $\mathcal{S}$ satisfy:
\begin{align}
&\mathcal{A} : L^2(\mathbb{R}, \mathrm{d}\mu) \to \ell^2,  &&f \mapsto \int_{\mathbb{R}} {\bf P}(x)^\top f(x) \mathrm{d}\mu(x) = {\bf f},\\ 
&\mathcal{S} : \ell^2 \to L^2(\mathbb{R}, \mathrm{d}\mu), &&{\bf f} \mapsto {\bf P}{\bf f} = f.
\end{align}
For a given basis, it is favourable if there exists a fast transform between function evaluations and coefficient vector expansions, i.e.~an ability to compute the analysis and synthesis operators in a fast manner.

Suppose that the goal is to apply the synthesis operator to $(t-x)^{c/2} {\bf Q}^{t,(a,b,c)}(x) {\bf f}$. We consider a ``square-root'' weighted expansion as it directly connects to the synthesis operator for Zernike annular polynomials in \cref{sec:zernike-analysis-synthesis}. Suppose that $c$ is even. Then, via \cref{eq:qr1},
%Then, by utilising \cref{eq:qr1},
\begin{align}
\begin{split}
(t-x)^{c/2}{\bf Q}^{t, (a,b,c)}(x) {\bf  f}
&=  (t-x)^{c/2-1}{\bf Q}^{t,(a,b,c-2)}(x) Q_{c-2} {\bf  f}\\
& =   {\bf Q}^{t,(a,b,0)}(x) Q_0 Q_2 \cdots Q_{c-2} {\bf  f}\\
& =  {\bf T}(1-2x) R_T S^{-1}_{(a,b)} Q_0 Q_2 \cdots Q_{c-2} {\bf  f}.
\end{split}
\end{align}
Here $R_T$ is the upper-triangular conversion matrix as computed via the Chebyshev--Jacobi transform \cite{Slevinsky2018}, $S_{(a,b)}$ is the diagonal scaling as defined by \cref{eq:scaling}, and $Q_\gamma R_\gamma = tI - X_{t,(a,b,\gamma)}$. When $c$ is odd, one converts down to  $(t-x)^{1/2} {\bf Q}^{t,(a,b,1)}(x)$ with QR and then a final Cholesky factorisation to recover $(t-x)^{1/2} {\bf Q}^{t,(a,b,0)}(x)$. We then convert ${\bf Q}^{t,(a,b,0)}(x)$ to ${\bf T}(1-2x)$ and there exists a fast transform for the synthesis operator of $(t-x)^{1/2}{\bf T}(1-2x)$. Thus the complexity for computing the Chebyshev coefficients is $\mathcal{O}(cN)$ where $N$ is the truncation degree. The synthesis operator may then be applied in $\mathcal{O}(N \log N)$ complexity via the DCT. We note that the analysis operator is the reverse process of the synthesis operator.

For an unweighted expansion, we instead utilise the $R$ factors in the QR factorisation. If $c$ is even, we find that 
\begin{align}
\begin{split}
{\bf Q}^{t, (a,b,c)}(x) {\bf  f} =  {\bf T}(1-2x) R_T S^{-1}_{(a,b)} R_0^{-1} R_2^{-1} \cdots R_{c-2}^{-1} {\bf  f}.
\end{split}
\end{align}

\subsection{Differentiation}

The goal is to compute the entries of the banded matrix $D^{t,(a+1,b+1,c+1)}_{(a,b,c)}$ and weighted counterparts. A core concept for computing the differentiation matrices with low complexity is the following sparsity result.

\begin{theorem}[Theorem 2.20 in \cite{Gutleb2023}]
\label{th:banded-derivatives}
The differentiation matrix in the right-hand side of
\begin{align*}
\fdx {\bf Q}^{t,(a,b,c)}(x) = {\bf Q}^{t,(a+1,b+1,c+1)}(x) D^{t,(a+1,b+1,c+1)}_{(a,b,c)}
\end{align*}
only has nonzero entries on the first two super-diagonals.

\end{theorem}
Similar results holds for the (partially) weighted counterparts. 
$$
D^{t,(a-1,b+1,c+1)}_{\mathrm{a},(a,b,c)}, D^{t,(a-1,b-1,c+1)}_{\mathrm{ab},(a,b,c)}\hbox{, and }D^{t,(a-1,b-1,c-1)}_{\mathrm{abc},(a,b,c)}
$$
all have a bandwidth of two and similarly for the other combination of weights, ${_\mathrm{b}}$,  ${_\mathrm{c}}$, ${_\mathrm{bc}}$, and ${_\mathrm{ac}}$.

Given the differentiation and Jacobi matrix for ${\bf Q}^{t,(a,b,c)}(x)$ we note that one obtains the differentiation matrix for ${\bf Q}^{t,(a,b,c+1)}(x)$ in $\mathcal{O}(N)$ complexity as follows:
\begin{align}
\begin{split}
\fdx {\bf Q}^{t,(a,b,c+1)}(x)
&= \fdx {\bf Q}^{t,(a,b,c)}(x) (R^{t,(a,b,c+1)}_{(a,b,c)})^{-1}\\
&= {\bf Q}^{t,(a+1,b+1,c+1)}(x) D^{t,(a+1,b+1,c+1)}_{(a,b,c)} (R^{t,(a,b,c+1)}_{(a,b,c)})^{-1}\\
&= {\bf Q}^{t,(a+1,b+1,c+2)}(x) R^{t,(a+1,b+1,c+2)}_{(a+1,b+1,c+1)} D^{t,(a+1,b+1,c+1)}_{(a,b,c)} (R^{t,(a,b,c+1)}_{(a,b,c)})^{-1}.
\end{split}\label{eq:diff1}
\end{align}
The connection coefficients are Cholesky factors \cite{Gutleb2023}:
\begin{align}
(R^{t,(a,b,c+1)}_{(a,b,c)})^\top R^{t,(a,b,c+1)}_{(a,b,c)} &= tI - X_{t,(a,b,c)}, \\
(R^{t,(a+1,b+1,c+2)}_{(a+1,b+1,c+1)})^\top R^{t,(a+1,b+1,c+2)}_{(a+1,b+1,c+1)} &= tI - X_{t,(a+1,b+1,c+1)}.
\end{align}
The unknown Jacobi matrix $X_{t,(a+1,b+1,c+1)}$ is obtained by noting
\begin{align}
(R^{t,(a+1,b+1,c+1)}_{(a,b,c)})^\top R^{t,(a+1,b+1,c+1)}_{(a,b,c)} &= X_{t,(a,b,c)} (I- X_{t,(a,b,c)}) (tI - X_{t,(a,b,c)}) 
\end{align}
and
\begin{align}
X_{t,(a+1,b+1,c+1)} = R^{t,(a+1,b+1,c+1)}_{(a,b,c)} X_{t,(a,b,c)} (R^{t,(a+1,b+1,c+1)}_{(a,b,c)})^{-1}.
\end{align}
Thanks to \cref{th:banded-derivatives} the triple matrix product on the right-hand side of \cref{eq:diff1} has a bandwidth of two. Hence, the entries of the differentiation matrix may be computed in linear complexity as one knows the upper bandwidth and thus we are not required to compute every entry of each column. A similar computation is conducted for (partially) weighted quasimatrices except the parameter associated with the weight is incremented down instead of up.

If one does not have the differentiation and Jacobi matrix of ${\bf Q}^{t,(a,b,c)}(x)$, then one recursively decrements the $c$ parameter until they hit a case where the the matrices are known. In the worst case, they reach the case where $c=0$ which is a scaling of the classical Jacobi polynomial case. In particular
\begin{align}
\fdx {\bf Q}^{t,(a,b,0)}(x) ={\bf Q}^{t,(a+1,b+1,0)}(x) D^{t,(a+1,b+1,0)}_{(a,b,0)}
\end{align}
where $D^{t,(a+1,b+1,0)}_{(a,b,0)}$ has an explicit expression and has a bandwidth of one \cite[Eq.~18.9.15]{dlmf}. Thus the worst case complexity to compute the differentiation matrix for an arbitrary ${\bf Q}^{t,(a,b,c)}(x)$ quasimatrix, $c \in \mathbb{N}$,  with zero pre-computation is $\mathcal{O}(cN)$. %However, in many cases of interest, in particular when one builds multivariate orthogonal polynomials, one compute hierarchies of semiclassical Jacobi polynomial families where $a$ and $b$ are fixed and $c$ runs from zero to a truncation mode $m$. Hence, one requires the differentiation matrix for each $c$ and thus there is no wasted computation.

\subsection{Computing the triple matrix products in $\mathcal{O}(N)$ complexity}
\label{sec:complexity}

In this section we briefly discuss how to compute the triple matrix products in \eqref{eq:choleskymethod}, \eqref{eq:analysis} and \eqref{eq:qrmethodR} in $\mathcal{O}(N)$ complexity and provide a comparison of the Cholesky method and the two methods based on QR decomposition. For the $R$-based methods of Cholesky or QR, the object of concern is naturally the right-application of the inverse $R^{-1}$ which, in general, will not be banded. For the $Q$-based method via QR, the inverse computation is trivial via $Q^\top$, but in exchange neither $Q$ nor $Q^\top$ are banded.

As usual in computational contexts, the main idea for optimal complexity for the $R$ methods is to never explicitly construct the inverse matrix $R^{-1}$. Instead one notes that the bandwidth of the matrix product $RX$ to which we wish to right-apply $R^{-1}$ is known from the bandwidths of $R$ and $X$, where $X$ is a Jacobi matrix. Since the resulting matrix $RXR^{-1}$ is a Jacobi matrix and must thus be tridiagonal, one only needs to compute three entries per column with $N$-independent complexity, i.e. in $\mathcal{O}(1)$. Computing the entries required for an $N \times N$ principal subblock of the new Jacobi matrix thus has complexity $\mathcal{O}(N)$.

Obtaining optimal complexity for the computation of $Q^\top X Q$ in the $Q$ method is slightly more work, in particular, since $Q^\top$ generally has infinitely many subdiagonals. Applying $Q^\top$ na\"ively to the tridiagonal $X$ would thus require $O(N^2)$ operations. To avoid this we require access to the Householder matrices that comprise the orthogonal matrix $Q$. Fortunately, competitive QR factorisations of symmetric tridiagonal (and banded) matrices are already stored in such a way that $Q$ is not directly constructed and instead the information to compute the individual Householder matrices is retained. Using this information, the action of each Householder triple matrix product $H_j^\top XH_j$ only affects a small block (determined by the size of the \emph{lower} bandwidth of the $Q$ factor) on the band and elements in the top left of $H_j^\top XH_j$ are progressively finalised and no longer modified by applications of the remaining Householder matrices. This means that the algorithm only has to update a small block of $X$ along the band starting in the top left by multiplying with the active subblock of the Householder matrix on both sides. The size of the blocks to be updated does not change with $N$ if we further use the knowledge that the resulting matrix, as a Jacobi matrix of univariate orthogonal polynomials, must be tridiagonal and thus the modification to $X$ from each triple Householder product $H_j^\top XH_j$ can be computed in $\mathcal{O}(1)$ complexity.

\begin{remark}
The $Q$ and $R$ variants correspond to the two similarity transformations implied by the QR algorithm with a constant shift for symmetric tridiagonal matrices. Similarly, the $R$ method from Cholesky corresponds to the LR algorithm  with the same constant shift.
\end{remark}

In \cref{fig:choleskyqrcomparison:a} we show a plot demonstrating the $\mathcal{O}(N)$ computational complexity of the above-described algorithms and compare the performance of the three available approaches: Cholesky as in \eqref{eq:choleskymethod}, $Q$-based QR as in \eqref{eq:analysis} and $R$-based QR as in \eqref{eq:qrmethodR}. In \cref{fig:choleskyqrcomparison:b} we show the relative 2-norm error comparing single and double precision computations of computed Jacobi matrix principal subblocks. The stability of all three methods is self-evident. In our experiments, we have not observed more than a constant factor between all three approaches; certain parameter values lead to one method outperforming the others. By analogy to the QR and LR eigenvalue algorithms, our default preference is the $Q$-method.

\begin{figure}[h!]
\centering
\subfloat[Computational cost]{\includegraphics[width =0.475\textwidth]{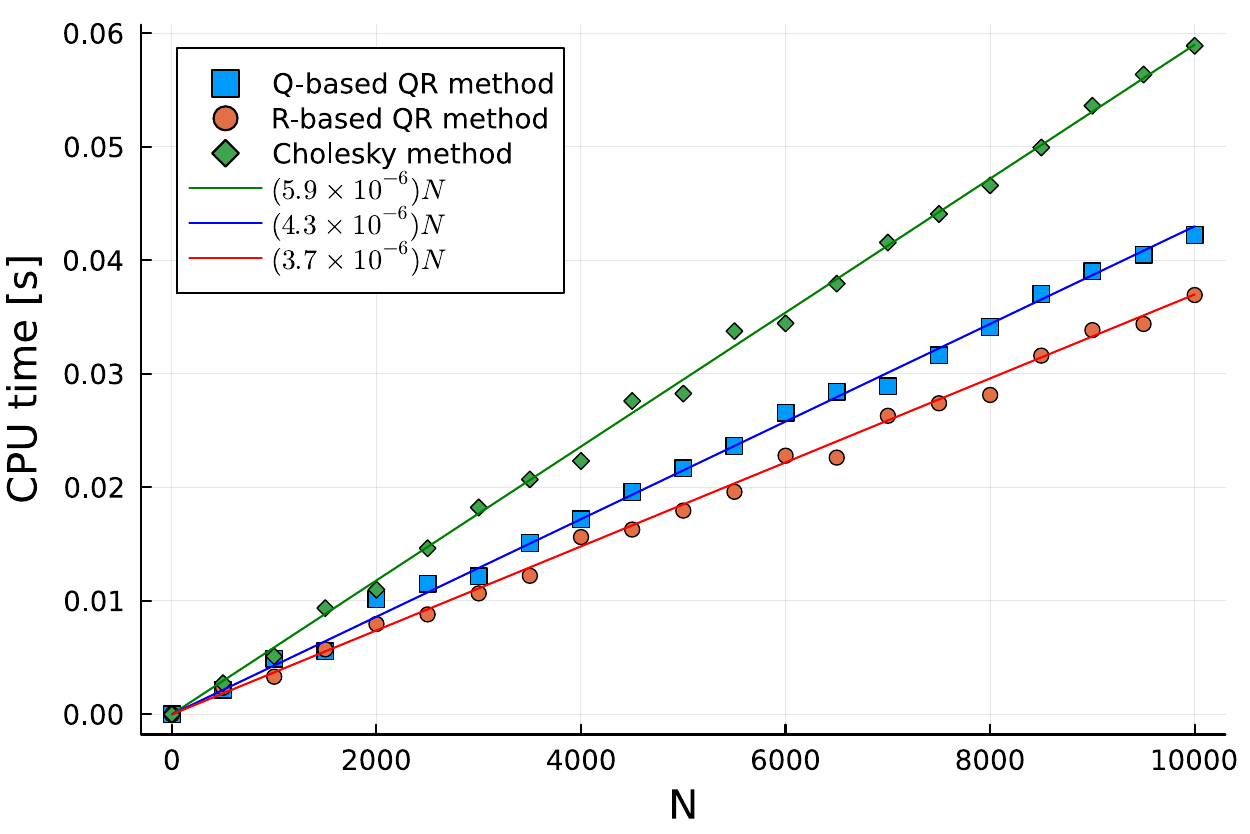}\label{fig:choleskyqrcomparison:a}}
\subfloat[2-norm relative error (single precision)]{\includegraphics[width =0.475\textwidth]{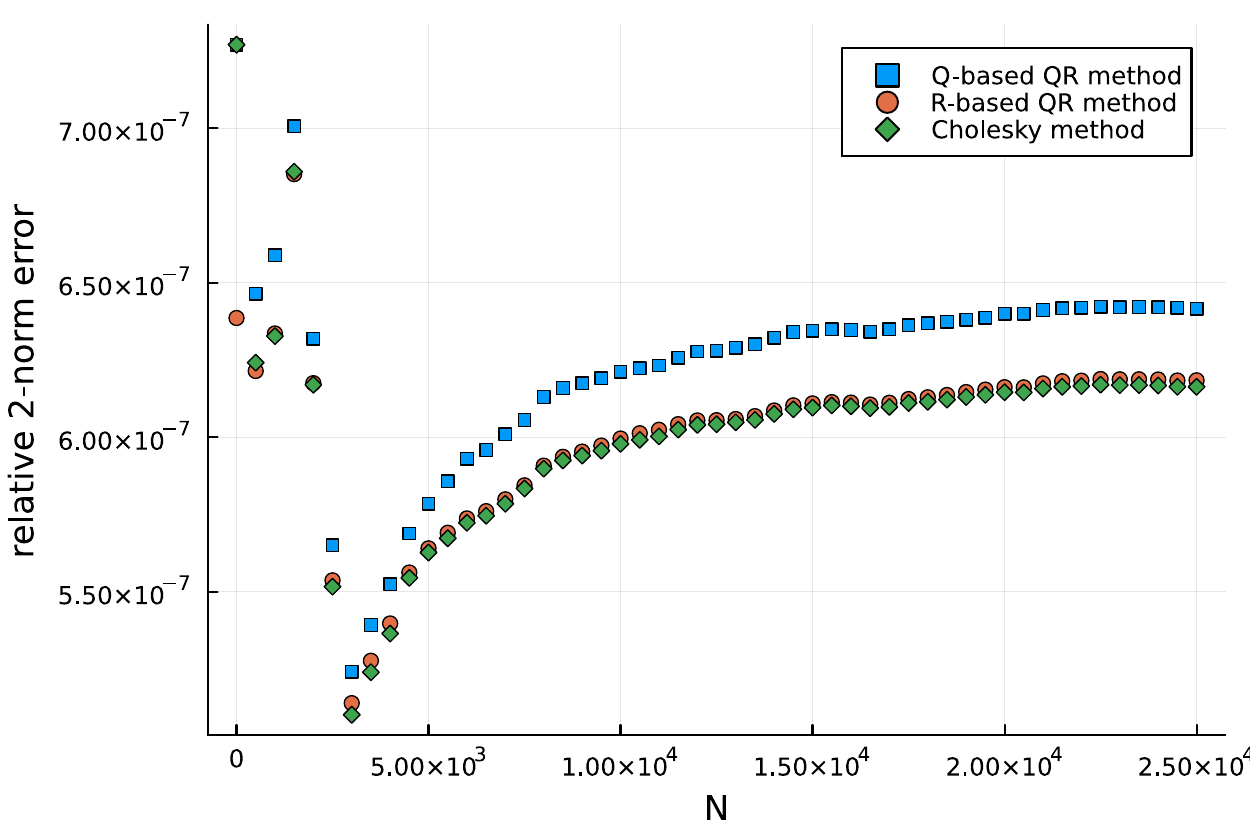}\label{fig:choleskyqrcomparison:b}}
\caption{(a) shows the $\mathcal{O}(N)$ complexity of the Cholesky as well as $Q$ and $R$ based QR methods for computing the $N \times N$ principal Jacobi matrix subblock of the semiclassical Jacobi polynomials $Q_n^{t,(1,1,22)}(x)$ from $Q_n^{t,(1,1,20)}(x)$. The Cholesky method requires twice the amount of work in this example as it proceeds in two steps. CPU timings were obtained on a Lenovo ThinkPad X1 Carbon laptop with a 12th Gen Intel i5-1250P CPU. (b) shows the relative 2-norm error for $N\times N$ principal subblocks of the Jacobi matrices comparing single to double precision computations.}\label{fig:choleskyqrcomparison}
\end{figure}

\section{Orthogonal polynomials on the annulus}
\label{sec:annulusOPs}

We now have the ingredients to define the orthogonal polynomials on the annulus as well as compute the connections between various families. Consider an inner radius  $0 < \rho < 1$ and without loss of generality take the outer radius to be one. Throughout this section $R^{t,(\alpha,\beta,\gamma)}_{\boldsymbol{\cdot}, (a,b,c)}$ and $D^{t,(\alpha,\beta,\gamma)}_{\boldsymbol{\cdot}, (a,b,c)}$ denote the lowering and differentiation matrix from the (potentially weighted) semiclassical Jacobi quasimatrix with parameters $(\alpha,\beta,\gamma)$ to $(a,b,c)$ as they were defined in \cref{sec:semiclassicalJacobi}.

On the two-dimensional disk ($\rho = 0$) we have the (complex-valued) generalised Zernike polynomials, written in terms of Jacobi polynomials as
\begin{align}
Z^{(b)}_{n,m}(x,y) \coloneqq r^{m} \E^{\I m \theta} P_{(n-m)/2}^{(b,m)}(2r^2-1), 
\end{align}
where $r^2 = x^2 + y^2$, $0 \leq m \leq n$ and 
\begin{align}
\begin{cases}
m = 0 & \text{if} \; n =0\\
m = 1,3, \dots, n & \text{if $n$ is odd},\\
m = 0, 2, \ldots, n & \text{if $n$ is even}.
\end{cases}
\label{eq:m2n}
\end{align}
These are polynomials in Cartesian coordinates $x$ and $y$ since $r^m \E^{\pm \I m \theta} = (x \pm \I y)^m$. Moreover, they are orthogonal on the unit disk with respect to the weight $(1-r^2)^b$. 

Note the real analogues are deduced by replacing the complex exponentials with sines and cosines. It is convenient for generalising to other dimensions to write these in terms of {\it harmonic polynomials}
\begin{align}
Y_{m,0}(x,y) &= \Im(x + \I y)^m = r^m \sin(m \theta), \\
Y_{m,1}(x,y) &= \Re(x + \I y)^m = r^m \cos(m \theta).
\end{align}
Thus the real-valued Zernike polynomials are:
\begin{align}
Z^{(b)}_{n,m,j}(x,y) =  Y_{m,j}(x,y) P_{(n-m)/2}^{(b,m)}(2r^2-1).  
\end{align}
For $n$ odd,  $m = 1, 3, \ldots, n$, and for $n$ even, $m = 0, 2, \ldots, n$. If $m = 0$ then $j = 1$, otherwise $j \in \{0,1\}$.  %The $d$-dimensional analogue of the real-valued Zernike polynomials is then \cite{Dunkl}:
%$$
%P_{n,k,j}(\vc x) = 
%$$

\begin{definition}[Generalised Zernike annular polynomials]
\label{def:def:2Dannuli}
For $t = (1-\rho^2)^{-1}$, $r^2 = x^2 + y^2$, we define the \emph{generalised Zernike annular polynomials}, $Z^{\rho,(a,b)}_{n,m,j}$, as
\begin{align}
Z^{\rho,(a,b)}_{n,m,j}(x, y) \coloneqq Y_{m,j}(x,y) Q_{(n-m)/2}^{t,(a,b,m)}\fpr({1-r^2 \over 1 - \rho^2}) ,
\label{def:2Dannuli}
\end{align}
where for $n$ odd,  $m = 1, 3, \ldots, n$, and for $n$ even, $m = 0, 2, \ldots, n$. If $m = 0$ then $j = 1$, otherwise $j \in \{0,1\}$. 
\end{definition}
First note the functions as defined in \cref{def:2Dannuli} are all polynomials. 

\begin{proposition}[Orthogonality of generalised Zernike annular polynomials, cf.~{\cite[Eq.~69]{Ellison2023}}]
Let $r^2=x^2+y^2$ and $\Omega_\rho = \{ (x,y) : 0< \rho \leq r \leq 1 \}$. Then
$
Z^{\rho,(a,b)}_{n,m,j}
$
are  orthogonal polynomials with respect to
\begin{align}
\ip<f,g>_{\rho, (a,b)} = \iint_{\Omega_\rho} f(x,y) g(x,y) (1-r^2)^a (r^2 - \rho^2)^b  \dx \dy.
\end{align}
\end{proposition}

%\begin{proof}
%
%Let $\tau = {1-r^2 \over 1 - \rho^2}$. In polar coordinates we have that
%\begin{align}
%&\ip<Z^{\rho,(a,b)}_{n,m,j},Z^{\rho,(a,b)}_{\nu,\omega,k}>_{\rho, (a,b)}\\
%& =\int_0^{2 \pi} \int_\rho^1   (1-r^2)^a (r^2 - \rho^2)^b  Z^{\rho,(a,b)}_{n,m,j} Z^{\rho,(a,b)}_{\nu,\omega,k} r\, \D r \D \theta \\
%\begin{split}
%&= \int_0^{2\pi} \frac{Y_{m,j} Y_{\omega,k}}{r^{m+\omega}} \D \theta\\
%& \indent \times \int_\rho^1 (1-r^2)^a (r^2 - \rho^2)^b r^{m+\omega+1} Q^{t,(a,b,m)}_{(n-m)/2}(\tau)  Q^{t,(a,b,m)}_{(\nu-m)/2}(\tau)  \, \D r.
%\end{split} \label{eq:2d-ortho-1}
%\end{align}
%If $m \neq \omega$ or $j\neq k$, then the first term on the right-hand side of \cref{eq:2d-ortho-1} equals zero. Suppose that $m = \omega$ and $j=k$. We have that
%\begin{align}
%1-r^2 &= t^{-1}\tau, \;\; r^2 = t^{-1}(t - \tau), \;\; r^2 - \rho^2 = t^{-1}(1 - \tau).
%\end{align}
%Thus
%\begin{align}
%\begin{split}
%& \int_\rho^1 (1-r^2)^a (r^2 - \rho^2)^b r^{2m+1} Q^{t,(a,b,m+c)}_{(n-m)/2}(\tau)  Q^{t,(a,b,m+c)}_{(\nu-m)/2}(\tau) \, \D r \\
%& \indent = {1 \over 2  t^{a + b + m + c + 1}} \int_0^1 \tau^a (1-\tau)^b  (t- \tau)^{m+c} Q^{t,(a,b,m+c)}_{(n-m)/2}(\tau)  Q^{t,(a,b,m+c)}_{(\nu-m)/2}(\tau) \, \D \tau.
%\label{eq:2d-ortho-2}
%\end{split}
%\end{align}
%By orthogonality of the semiclassical Jacobi polynomials $Q_n^{t,(a,b,m+c)}$, \cref{eq:2d-ortho-2} equals zero if $n \neq \nu$. Hence, $\ip<Z^{\rho,(a,b)}_{n,m,j},Z^{\rho,(a,b)}_{\nu,\omega,k}>_{\rho, (a,b)} = 0$ unless $n=\nu$, $m = \omega$, and $j = k$. 
%\end{proof}

\subsection{Connection and differentiation matrices}

In this subsection, our goal is to find the matrices $D$ and $L$ such that
\begin{align}
\Delta {\bf W}^{\rho, (1, 1)}(x,y)&= {\bf Z}^{\rho, (1, 1)}(x,y) D, \label{eq:annulus-Lap}\\
{\bf W}^{\rho, (1, 1)}(x,y)&= {\bf Z}^{\rho, (1, 1)}(x,y) R, \label{eq:annulus-Id}
\end{align}
where ${\bf W}^{\rho,(a,b)}(x,y)$ denotes the weighted version of ${\bf Z}^{\rho,(a,b)}(x,y)$. We discover that $D$ and $R$ are block tridiagonal and pentadiagonal, respectively, where each block is diagonal. 

Let ${\bf Z}^{\rho, (a,b)}_{m,j}(x,y)$ denote the quasimatrix containing only the ``$m$, $j$''-mode 2D annulus orthogonal polynomials with increasing $n$. For instance:
\begin{align}
{\bf Z}^{\rho, (a,b)}_{m,j} =
\left(
Z^{\rho, (a,b)}_{m,m,j}
\ | \ Z^{\rho, (a,b)}_{m+2,m,j}
\ | \  Z^{\rho, (a,b)}_{m+4,m,j} 
\ | \ \cdots
\right).
\end{align}
${\bf W}^{\rho, (1,1)}_{m,j}$ is defined analogously.

\begin{proposition}[Lowering of weighted generalised Zernike annular polynomials]
\label{prop:2d-low}
\begin{align}
{\bf W}^{\rho,(1,1)}_{m,j}(x,y) = (1-\rho^2)^2 {\bf Z}^{\rho,(1,1)}_{m,j}(x,y) R^{t,(1,1,m)}_{(0,0,m)} R^{t,(0,0,m)}_{\mathrm{ab}, (1,1,m)}.
\end{align}
\end{proposition}
$R^{t,(1,1,m)}_{(0,0,m)} R^{t,(0,0,m)}_{\mathrm{ab}, (1,1,m)}$ is a pentadiagonal matrix. Again the different $m$ and $j$ modes do not interact. When interlacing the different Fourier modes, we result in a block pentadiagonal matrix where each block is diagonal.
\begin{proof}
Let $z = x + \I y$.
\begin{align}
\begin{split}{\bf W}^{\rho,(1,1)}_{m,1}(x,y)
&= (1-r^2)(r^2-\rho^2)\Re[z^m] {\bf Q}^{t,(1,1,m)}(\tau)\\
&  = t^{-2} \Re[z^m] \tau(1-\tau) {\bf Q}^{t,(1,1,m)}(\tau)\\
& = t^{-2} \Re[z^m] {\bf Q}^{t,(0,0,m)}(\tau) R^{t,(0,0,m)}_{\mathrm{ab}, (1,1,m)}\\
&  = t^{-2} \Re[z^m] {\bf Q}^{t,(1,1,m)}(\tau) R^{t,(1,1,m)}_{(0,0,m)} R^{t,(0,0,m)}_{\mathrm{ab}, (1,1,m)}\\
& = (1-\rho^2)^2 {\bf Z}^{\rho,(1,1)}_{m,1}(x,y) R^{t,(1,1,m)}_{(0,0,m)} R^{t,(0,0,m)}_{\mathrm{ab}, (1,1,m)}.
\end{split}
\end{align}
The proof is similar for $j=0$. 
\end{proof}

In order to compute the Laplacian of the generalised Zernike annular polynomials, we follow \cite{VasilDisk} and use the decomposition
$$
\Delta =  4 \partial \bar\partial =4 \bar\partial  \partial 
$$
for the two linear operators\footnote{We have used slightly different operators from \cite{VasilDisk}, which used $D_\pm := {\partial \over \partial r} \pm {\I \over r} {\partial \over \partial \theta}$. This is to ensure we mapped from polynomials to polynomials.}
%$$
%D_\pm := {\partial \over \partial r} \pm {\I \over r} {\partial \over \partial \theta}.
%$$
%Or better yet
\begin{align*}
	2 \partial  = \partial_x - \I \partial_y  = \E^{- \I \theta} \pr(\partial_r - {\I \over r}\partial_\theta) \;\; \text{and} \;\;
	2 \dbar  = \partial_x + \I \partial_y  = \E^{- \I \theta} \pr(\partial_r + {\I \over r}\partial_\theta).
\end{align*}
\begin{lemma}
\label{lem:partial}
For $\tau = {1-r^2 \over 1-\rho^2}$, $t=(1-\rho^2)^{-1}$, $z = r \E^{\I \theta}$, and a differentiable function $f$, we have
\meeq{
	 \partial z^m f(\tau) = z^{m-1} (m f(\tau) - tr^2 f'(\tau)) = z^{m-1} (\tau-t)^{1-m} {\mathrm{d} \over \mathrm{d} \tau}[(\tau-t)^m f(\tau)], \ccr
	 \partial \bar z^m f(\tau) = - t \bar z^{m+1} f'(\tau), \ccr
	 \dbar z^m f(\tau) = - t z^{m+1} f'(\tau), \ccr
	 \dbar \bar z^m f(\tau) =  \bar z^{m-1} (m f(r^2) - t r^2 f'(r^2))  = \bar z^{m-1} (\tau-t)^{1-m}  {\mathrm{d} \over \mathrm{d} \tau}[(\tau-t)^m f(\tau)].
} 
\end{lemma}

\begin{proposition}[Laplacian of generalised Zernike annular polynomials]
\label{prop:2d-lap}
Let  $t=(1-\rho^2)^{-1}$, then
\begin{align}
\Delta {\bf Z}^{\rho, (a,b)}_{m, j}(x,y) = 4 (1-\rho^2)^{-1} {\bf Z}^{\rho, (a+2, b+2)}_{m, j}(x,y)  D^{t,(a+2,b+2,m)}_{\mathrm{c},(a+1,b+1,m+1)} D^{t,(a+1,b+1,m+1)}_{(a,b,m)}.
\end{align} 
\end{proposition}
\begin{proof}
Consider $j=1$ and $z = x + \I y$. Note that $\Delta = 4 \partial \dbar$ is a real operator and, therefore, commutes with the real operator $\Re$.  Hence,
\begin{align}
\begin{split}
&\Delta {\bf Z}^{\rho, (a,b)}_{m, 1}(x,y) \\
&\indent = 4 \partial \dbar \Re[z^m] {\bf Q}^{t,(a,b,m)}(\tau)\\
&\indent= 4 \Re \partial \left(-t z^{m+1} {\bf Q}^{t, (a+1,b+1,m+1)}(\tau) D^{t,(a+1,b+1,m+1)}_{(a,b,m)} \right)\\
& \indent= -4t \Re z^m (\tau-t)^{-m} {\mathrm{d} \over \mathrm{d} \tau} \left [(\tau-t)^{m+1} {\bf Q}^{t,(a+1,b+1,m+1)}(\tau) \right]D^{t,(a+1,b+1,m+1)}_{(a,b,m)}\\
&\indent = 4t\Re[z^m] {\bf Q}^{t,(a+2,b+2,m)}(\tau) D^{t,(a+2,b+2,m)}_{\mathrm{c},(a+1,b+1,m+1)} D^{t,(a+1,b+1,m+1)}_{(a,b,m)}\\
&\indent = 4t {\bf Z}^{\rho, (a+2, b+2)}_{m, 1}(x,y) D^{t,(a+2,b+2,m)}_{\mathrm{c},(a+1,b+1,m+1)} D^{t,(a+1,b+1,m+1)}_{(a,b,m)}.
\end{split}
\label{eq:LapR}
\end{align}
%where 
%\begin{align*}
%D_1 &= {\bf Q}^{t, (a+1,b+1,m+1)}(\tau) \backslash \frac{\mathrm{d}}{\mathrm{d}\tau}{\bf Q}^{t, (a,b,m)}(\tau),\\
%D_2 &= H\{:c\}{\bf Q}^{t, (a+2,b+2,m)}(\tau) \backslash \frac{\mathrm{d}}{\mathrm{d}\tau} H\{:c\}{\bf Q}^{t, (a+1,b+1,m+1)}(\tau) .
%\end{align*} 
The second and third lines in \cref{eq:LapR} hold thanks to \cref{lem:partial}. By replacing $\Re[z^m]$ with $\Im[z^m]$, the exact same steps hold for $j=0$.
\end{proof}

$D^{t,(a+2,b+2,m)}_{\mathrm{c},(a+1,b+1,m+1)} D^{t,(a+1,b+1,m+1)}_{(a,b,m)}$ is an upper-triangular matrix with a bandwidth of three. We derive a similar relationship for the weighted version $W^{\rho,(1,1)}_{n,m,j}(x,y)$.
\begin{proposition}[Laplacian of weighted generalised Zernike annular polynomials]
\label{prop:2d-Lap}
\begin{align}
\Delta {\bf W}^{\rho, (1,1)}_{m, j}(x,y) = 4(1-\rho^2) {\bf Z}^{\rho, (1, 1)}_{m, j}(x,y) D^{t,(1,1,m)}_{\mathrm{c},(0,0,m+1)}  D^{t,(0,0,m+1)}_{\mathrm{ab},(1,1,m)}.
\end{align}
\end{proposition}
\begin{proof}
Consider $j=1$ and $z = x + \I y$. Then,
\begin{align}
\begin{split}
&\Delta {\bf W}^{\rho, (1,1)}_{m,1}(x,y)\\
& \indent = 4 \partial \dbar \left[(1-r^2)(r^2-\rho^2)\Re[z^m] {\bf Q}^{t,(1,1,m)}(\tau)\right]\\
& \indent = 4t^{-2} \Re \partial \dbar \left[ z^m \tau(1-\tau) {\bf Q}^{t,(1,1,m)}(\tau)\right]\\
& \indent= -4t^{-1}\Re \partial \left[z^{m+1} {\bf Q}^{t,(0,0,m+1)}(\tau) D^{t,(0,0,m+1)}_{\mathrm{ab},(1,1,m)} \right]\\
& \indent = -4t^{-1} \Re z^m (\tau-t)^{-m} {\mathrm{d} \over \mathrm{d} \tau} \left [(\tau-t)^{m+1} {\bf Q}^{t,(0,0,m+1)}(\tau) \right]  D^{t,(0,0,m+1)}_{\mathrm{ab},(1,1,m)}\\
& \indent = 4t^{-1}\Re[z^{m}] {\bf Q}^{t,(1,1,m)}(\tau)  D^{t,(1,1,m)}_{\mathrm{c},(0,0,m+1)} D^{t,(0,0,m+1)}_{\mathrm{ab},(1,1,m)}\\
& \indent = 4t^{-1} {\bf Z}^{\rho, (1, 1)}_{m, 1}(x,y) D^{t,(1,1,m)}_{\mathrm{c},(0,0,m+1)} D^{t,(0,0,m+1)}_{\mathrm{ab},(1,1,m)}.
\end{split}
\label{eq:LapR-2}
\end{align} 
The third and fourth lines in \cref{eq:LapR-2} hold thanks to \cref{lem:partial}. The case $j=0$ follows similarly.
%where 
%\begin{align*}
%D_3 &= H\{:ab\}{\bf Q}^{t, (0,0,m+1)}(\tau) \backslash \frac{\mathrm{d}}{\mathrm{d}\tau}H\{:ab\}{\bf Q}^{t, (1,1,m)}(\tau),\\
%D_4 &= H\{:c\}{\bf Q}^{t, (1,1,m)}(\tau)  \backslash \frac{\mathrm{d}}{\mathrm{d}\tau} H\{:c\}{\bf Q}^{t, (0,0,m+1)}(\tau) .
%\end{align*}
\end{proof}
The matrix $D^{t,(1,1,m)}_{\mathrm{c},(0,0,m+1)}  D^{t,(0,0,m+1)}_{\mathrm{ab},(1,1,m)}$  is tridiagonal. An important feature is the different $m$ and $j$ modes do not interact. Thus we obtain tridiagonal Laplacians for each Fourier mode.

\subsection{Analysis \& synthesis}
\label{sec:zernike-analysis-synthesis}

By computing the connection, Jacobi, and differentiation matrices for hierarchies of semiclassical Jacobi families as described in \cref{sec:semiclassicalJacobi}, we may compute the connection and differentiation matrices for the Zernike annular polynomials in optimal complexity. In particular the connection and differentiation matrices for all the Fourier modes up to the truncation mode and degree $M=N$ in Propositions \labelcref{prop:2d-low} and \labelcref{prop:2d-lap} are computed in $\mathcal{O}(N^2)$ complexity.

We now discuss the technique for $\mathcal{O}(N^2 \log (N))$ complexity analysis and synthesis operators. Suppose that the goal is to apply the synthesis operator to the coefficient vector ${\bf f}$ to evaluate ${\bf Z}^{\rho,(a,b)}(x,y){\bf f}$ on the grid $(x,y) \in \{(r_k \cos(l \theta), r_k \sin(l \theta))\}$ where
\begin{align}
r_k &= \sqrt{\cos^2\left(\frac{2k+1}{4K}\pi\right) + \rho^2 \sin^2\left(\frac{2k+1}{4K}\pi\right)} \;\; \text{for} \;\; 0 \leq k \leq K-1,\\
\theta_l & = \frac{2l}{L} \pi \;\; \text{for} \;\; 0 \leq l \leq L-1.
\label{eq:zernike-grid}
\end{align}
Here $K = \floor{(N+1)/2}+1$, $L = 4K-3$, and $N$ is the truncation degree. Given the even truncation degree $N$, we first rearrange the coefficients $f_{n,m,j}$ in ${\bf f}$ into the $N/2 \times (N+1)$ rectangular matrix:
\begin{align}
\begin{pmatrix}
f_{0,0,1} & f_{1,1,0} & f_{1,1,1} & f_{2, 2, 0} & f_{2, 2, 1} & \cdots & f_{N, N, 0} & f_{N, N, 1} \\
f_{2,0,1} & f_{3,1,0} & f_{3,1,1} & f_{4, 2, 0} & f_{4, 2, 1} & \cdots & 0 & 0 \\
\vdots & \vdots & \vdots & \vdots & \vdots & \ddots & \vdots & \vdots \\
f_{N-2,0,1} & f_{N-1,1,0} & f_{N-1,1,1} & f_{N, 2, 0} & f_{N, 2, 1} & \cdots & 0 & 0 \\
f_{N,0,1} & 0 & 0 & 0 & 0 & \cdots & 0 & 0
\end{pmatrix}.
\label{eq:Z-coeffs}
\end{align}
Each column corresponds to a separate Fourier mode $(m,j)$. The matrix of coefficients \cref{eq:Z-coeffs} is converted to a matrix of Chebyshev--Fourier series coefficients for which a fast transform exists via a DCT and FFT. The conversion is best understood column-wise. Consider the last column and let ${\bf  f}_{N, N, 1} = ( f_{N, N, 1} \; 0 \; \cdots \; 0)^\top \in \mathbb{R}^{N/2}$. Note that
\begin{align}
\begin{split}
{\bf Z}^{\rho, (a,b)}_{N,1}(x,y) {\bf  f}_{N, N, 1}
&= r^{N} \cos(N\theta) {\bf Q}^{t,(a,b,N)}(\tau) {\bf  f}_{N, N, 1}\\
&=  t^{-N/2} (t-\tau)^{N/2} \cos(N\theta) {\bf Q}^{t,(a,b,N)}(\tau) {\bf  f}_{N, N, 1}.
\end{split}
\end{align}
Thus, as derived in \cref{sec:semiclassical-analysis}, the final column in the matrix for the equivalent Chebyshev--Fourier coefficients is given by $t^{-N/2} R_T S^{-1}_{(a,b)} Q_0 Q_2\cdots Q_{N-2} {\bf  f}_{N, N, 1}$. Therefore, the final column is a product of $\mathcal{O}(N)$ $Q$-factors, each of which costs $\mathcal{O}(N)$ to apply, culminating in $\mathcal{O}(N^2)$ work. Na\"ively, the complexity for computing the Chebyshev--Fourier coefficients for all the columns is $\mathcal{O}(N^3)$. However, this conversion has considerable speedups if one utilises Givens rotations and a butterfly application resulting in $\mathcal{O}(N^2 \log N)$ complexity. We refer the interested reader to \cite{Slevinsky2019} for a thorough description and numerical analysis of these algorithms for spherical harmonic polynomials.

\subsection{Non-constant coefficients}
As for classical sparse spectral methods, band-limited operators are available for non-constant coefficients (NCCs). In particular one retains sparsity if the NCC can be accurately represented as a low-degree polynomial in $x$, $y$, or $r^2$. The implementation of NCCs in the spectral method hinges on the two propositions in this subsection.

\begin{proposition}[$r^2$-Jacobi matrix]
\label{prop:jacobi-matrix-r2}
Let $t = (1-\rho^2)^{-1}$ and $X_{t,(a,b,c)}$ denote the Jacobi matrix for the semiclassical Jacobi quasimatrix ${\bf Q}^{t,(a,b,c)}$. Then
\begin{align}
r^2 {\bf Z}_{m,j}^{\rho,(a,b)}(x,y) &=  {\bf Z}_{m,j}^{\rho,(a,b)}(x,y) (I-t^{-1}X_{t,(a,b,m)}), \label{eq:r2-2}
\end{align}
where $I$ is the identity matrix.
\end{proposition}
\begin{proof}
Consider $\tau = t(1-r^2)$. Then
\begin{align}
\begin{split}
r^2 {\bf Z}_{m,j}^{\rho,(a,b)}(x,y) 
&= r^2 Y_{m,j}(x,y) {\bf Q}^{t,(a,b,m)}(\tau) \\
&=    Y_{m,j}(x,y) (1-t^{-1}\tau) {\bf Q}^{t,(a,b,m)}(\tau) \\
&=    Y_{m,j}(x,y) {\bf Q}^{t,(a,b,m)}(\tau) (I - t^{-1} X_{t,(a,b,m)}) \\
&=   {\bf Z}_{m,j}^{\rho,(a,b)}(x,y)  (I - t^{-1} X_{t,(a,b,m)}).
\end{split}
\end{align}
Thus \cref{eq:r2-2} holds.
\end{proof}

\begin{proposition}[$x$, $y$-Jacobi matrix]
\label{prop:jacobi-matrix-x-y}
Let $t = (1-\rho^2)^{-1}$ and $R^{t,(a,b,c+1)}_{(a,b,c)}$ denote the raising matrix for the semiclassical Jacobi families ${\bf Q}^{t,(a,b,c)} = {\bf Q}^{t,(a,b,c+1)}  R^{t,(a,b,c+1)}_{(a,b,c)}$. Then
\begin{align}
\begin{split}
&x {\bf Z}_{m,j}^{\rho,(a,b)}(x,y)\\ &\indent=  \frac{1}{2} \left[t^{-1} {\bf Z}_{m-1,j}^{\rho,(a,b)}(x,y) (R^{t,(a,b,m)}_{(a,b,m-1)})^\top + {\bf Z}_{m+1,j}^{\rho,(a,b)}(x,y) R^{t,(a,b,m+1)}_{(a,b,m)}\right],
\end{split} \label{eq:x1} \\
\begin{split}
&y {\bf Z}_{m,1}^{\rho,(a,b)}(x,y) \\ &\indent= \frac{1}{2} \left[ - t^{-1} {\bf Z}_{m-1,0}^{\rho,(a,b)}(x,y) (R^{t,(a,b,m)}_{(a,b,m-1)})^\top + {\bf Z}_{m+1,0}^{\rho,(a,b)}(x,y) R^{t,(a,b,m+1)}_{(a,b,m)}\right],  \end{split} \label{eq:x2} \\
\begin{split}
&y {\bf Z}_{m,0}^{\rho,(a,b)}(x,y) \\ &\indent =  \frac{1}{2} \left[t^{-1} {\bf Z}_{m-1,1}^{\rho,(a,b)}(x,y) (R^{t,(a,b,m)}_{(a,b,m-1)})^\top - {\bf Z}_{m+1,1}^{\rho,(a,b)}(x,y) R^{t,(a,b,m+1)}_{(a,b,m)}\right]. \end{split} \label{eq:y2} 
\end{align}
\end{proposition}
\begin{proof}
We prove \cref{eq:x1} when $j=1$ and note that the case $j=0$ as well as \cref{eq:x2}--\cref{eq:y2} all follow similarly. Let $\tau = t(1-r^2)$ and first consider \cref{eq:x1}:
\begin{align}
\begin{split}
&x {\bf Z}_{m,1}^{\rho,(a,b)}(x,y) 
= (r\cos \theta) r^m \cos(m\theta) {\bf Q}^{t,(a,b,m)}(\tau)\\
&\indent = \frac{1}{2}\left[r^{m+1} \cos((m-1)\theta) {\bf Q}^{t,(a,b,m)}(\tau)
+ r^{m+1} \cos((m+1)\theta) {\bf Q}^{t,(a,b,m)}(\tau)
\right],
\end{split}
\label{eq:jacobix1}
\end{align}
where we use the identity $2\cos\alpha \cos\beta = \cos(\alpha-\beta) +  \cos(\alpha+\beta)$. Note that, since $r^2 = t^{-1} (t-\tau)$,
\begin{align}
\begin{split}
r^2{\bf Q}^{t,(a,b,m)}(\tau) 
= t^{-1} (t-\tau) {\bf Q}^{t,(a,b,m)}(\tau)
= t^{-1} {\bf Q}^{t,(a,b,m-1)}(\tau)(R^{t,(a,b,m)}_{(a,b,m-1)})^\top,
\end{split}
\label{eq:jacobix2}
\end{align}
where the final equality follows as a consequence of \cref{th:factorisations}. Substituting \cref{eq:jacobix2} into \cref{eq:jacobix1}, we deduce that
\begin{align}
\begin{split}
&x {\bf Z}_{m,1}^{\rho,(a,b)}(x,y) \\
&\indent = \frac{1}{2} t^{-1} r^{m-1} \cos((m-1)\theta) {\bf Q}^{t,(a,b,m-1)}(\tau) (R^{t,(a,b,m)}_{(a,b,m-1)})^\top\\
&\indent \indent + \frac{1}{2}r^{m+1} \cos((m+1)\theta) {\bf Q}^{t,(a,b,m+1)}(\tau) R^{t,(a,b,m+1)}_{(a,b,m)} \\
 &\indent =  \frac{1}{2} t^{-1} {\bf Z}_{m-1,1}^{\rho,(a,b)}(x,y) (R^{t,(a,b,m)}_{(a,b,m-1)})^\top
 + \frac{1}{2} {\bf Z}_{m+1,1}^{\rho,(a,b)}(x,y) R^{t,(a,b,m+1)}_{(a,b,m)}.
\end{split}
\label{eq:jacobix3}
\end{align}
\end{proof}

\section{Sparse spectral method for PDEs}
\label{sec:pdes}
In this section we detail how one may use the Chebyshev--Fourier series and the generalised Zernike annular polynomials to construct a \emph{sparse} spectral method for the Helmholtz equation. Consider the Helmholtz  equation, for some $\lambda \in \mathbb{R}$ (potentially $\lambda \in L^\infty(\Omega_\rho)$),
\begin{align}
\begin{split}
\Delta u(x,y) + \lambda u(x,y) &= f(x,y) \;\; \text{in} \;\; \Omega_\rho, \quad
u = 0 \;\; \text{on} \;\; \partial \Omega_\rho. 
\end{split}
\label{eq:helmholtz}
\end{align}

\textbf{Boundary conditions.}
\label{sec:bcs}
We consider two techniques to enforce boundary conditions:
\begin{enumerate}
\itemsep=0pt
\item Incorporate the boundary conditions into the basis. A homogeneous Dirichlet boundary condition is automatically satisfied on the annulus if all the basis polynomials contain the factor $(1-r^2)(r^2-\rho^2)$.
\item Provided the equation decouples across Fourier modes, then one may add the boundary conditions as two additional constraints. In general, these additional constraints result in two dense rows.
\end{enumerate}

An advantage of the second technique is that it is amendable to enforcing inhomogeneous Dirichlet boundary conditions as well as Neumann and Robin boundary conditions with only small modifications \cite{Olver2013}. This results in an almost banded matrix which can be solved optimally \cite[Sec.~5]{Olver2013}. However, it may cause severe ill-conditioning if constructed na\"ively. Tau methods are a robust technique for enforcing boundary conditions in this way. Moreover if they are combined with a preconditioner (via a Schur complement) then one may recover the bandedness, cf.~\cite{Lanczos1938, Ortiz1969} and \cite[Sec.~B]{Burns2020}. In the first technique the boundary conditions are automatically satisfied, however, incorporating nontrivial boundary conditions can be difficult. 

%If the equations are posed in weak form, the Neumann boundary conditions are typically enforced without modification. Dirichlet boundary conditions may be enforced by zeroing the rows and columns of the matrix associated with basis functions that are nonzero at the boundary, and adding a one on the diagonal. This technique preserves symmetry and sparsity. 

\textbf{Chebyshev--Fourier series.}
Consider $x = r\cos(\theta)$ and $y=r\sin(\theta)$. Then a change of variables for \cref{eq:helmholtz} gives:
 \begin{align}
\left(r^2  \frac{\partial^2}{\partial r^2} + r \frac{\partial}{\partial r} + \frac{\partial^2}{\partial \theta^2} + \lambda r^2 \right ) u(r,\theta) = r^2 f(r,\theta). 
\end{align}
Let ${\bf C}^{(2)}$ denote the quasimatrix of the ultraspherical polynomials \cite[Sec.~18.3]{dlmf}. Recall the definition of $r_\rho$ from \cref{eq:r-affine}. Consider the matrices $D$, $R$, and $X$ such that
\begin{align*}
\left[r^2  \frac{\partial^2}{\partial r^2} + r \frac{\partial}{\partial r}\right] {\bf C}^{(2)}(r_\rho) &= {\bf T}(r_\rho) D, \;\; {\bf C}^{(2)}(r_\rho)R = {\bf T}(r_\rho) ,\;
 r{\bf C}^{(2)}(r_\rho) = {\bf C}^{(2)}(r_\rho)X.
\end{align*}
It is known that $D$ and $R$ have a bandwidth of five \cite{Olver2013} and $X$ is tridiagonal. The algorithm for solving the Helmholtz equation on the annulus via a scaled-and-shifted Chebyshev--Fourier series is provided in \cref{alg:scaled-and-shifted}. Boundary conditions are enforced via the second technique.
\begin{algorithm}[ht]
\caption{Scaled-and-shifted Chebyshev--Fourier series}
\label{alg:scaled-and-shifted}
\begin{algorithmic}[1]
\State{\texttt{EXPAND.} Expand the right-hand side $f(r,\theta)$ in the ${\bf T}(r_\rho)\otimes {\bf F}(\theta)$-basis. Truncate at polynomial degree $N$ and $2N-1$ Fourier modes for $(N+1) (2N+1)$ coefficients.}

The solve is then decomposed into a parallelisable sequence of one-dimensional solves:

\State{\texttt{ASSEMBLE.} (Lazily) assemble the matrices for $m \in \{0,1,\dots,N\}$:
\begin{align}
%A_{m} = \begin{pmatrix}
%1 & -1 & 1 & -1 & \cdots\\
%1 & 1 & 1 & 1 & \cdots\\
% \multicolumn{5}{c}{D - m^2 M + \lambda R^2 M}
%\end{pmatrix};
\begin{pmatrix}
{\bf T}(-1) \\
{\bf T}(1) \\
D - m^2 R + \lambda X^2 R
\end{pmatrix}.
\label{eq:chebyshev-fourier-A}
\end{align}
and truncate to form the matrix $A_m \in \mathbb{R}^{(N+2) \times (N+2)}$.} 
\State{\texttt{EXTRACT.} Extract and truncate the coefficients from ${\bf f}$ associated with the $(m, j)$-mode to form the coefficient vector ${\bf f}_{m,j} \in \mathbb{R}^N$.
}
\State{\texttt{SOLVE.} Solve the $2N+1$ linear systems of size $(N+2)\times(N+2)$, $A_m {\bf u}_{m, j} = \left(0, \; 0, \; (X^2 R {\bf f}_{m,j})^\top \right)^\top$, where $m \in \{0,1,\dots, N\}$, $j \in \{0,1\}$, $(m,j) \neq (0,0)$.}
\State{\texttt{INTERLACE.} Interlace ${\bf u}_{m, j} \in \mathbb{R}^{N+2}$ to form the coefficient vector ${\bf u}$.}
\State{$u(r,\theta) \approx \left[{\bf T}(r_\rho) \otimes {\bf F}(\theta)\right]  {\bf u}$.}
\end{algorithmic}
\end{algorithm}

%\begin{enumerate}
%\itemsep=0pt
%\item Expand the right-hand side $f(r,\theta)$ in the ${\bf T}(r_\rho)\otimes {\bf F}(\theta)$-basis.
%\end{enumerate}
%The solve is then decomposed into a parallelisable sequence of one-dimensional solves:
%\begin{enumerate}
%\setcounter{enumi}{1}
%\itemsep=0pt
%\item (Lazily) assemble the matrices for $m \in \mathbb{N}_0$:
%\begin{align}
%%A_{m} = \begin{pmatrix}
%%1 & -1 & 1 & -1 & \cdots\\
%%1 & 1 & 1 & 1 & \cdots\\
%% \multicolumn{5}{c}{D - m^2 M + \lambda R^2 M}
%%\end{pmatrix};
%A_{m} = \begin{pmatrix}
%{\bf T}(-1) \\
%{\bf T}(1) \\
%D - m^2 M + \lambda X^2 M
%\end{pmatrix};
%\label{eq:chebyshev-fourier-A}
%\end{align}
%\item Extract the coefficients from ${\bf f}$ associated with the $(m, j)$-mode to form the coefficient vector ${\bf f}_{m,j}$;
%\item Solve $A_m {\bf u}_{m, j} = \left(0, \; 0, \; (X^2 M {\bf f}_{m,j})^\top \right)^\top$, $m = 0,1,2,\dots$, $j=0,1$;
%\item Interlace ${\bf u}_{m, j}$ to form the coefficient vector ${\bf u}$; 
%\item $u(r,\theta) \approx \left[{\bf T}(r_\rho) \otimes {\bf F}(\theta)\right]  {\bf u}$. 
%\end{enumerate}
The expansion of the right-hand side $f(r,\theta)$ in step 1 of \cref{alg:scaled-and-shifted} may be achieved via a DCT and FFT for the Chebyshev and Fourier parts, respectively, resulting in a complexity of $\mathcal{O}(N^2 \log(N))$. In step 4, the first two rows in the right-hand side vector enforce the homogeneous Dirichlet boundary condition. The coefficient vector ${\bf f}_{m,j}$ is an expansion of $f(r,\theta)$ in the ${\bf T}(r_\rho) \otimes {\bf F}(\theta)$ basis and must be mapped to an expansion of $r^2 f(r,\theta)$ in the ${\bf C}^{(2)}(r_\rho) \otimes {\bf F}(\theta)$ basis. This is precisely $X^2 R {\bf f}_{m,j}$.

\textbf{Generalised Zernike annular polynomials.}
An algorithm for solving \cref{eq:helmholtz} with a generalised Zernike annular polynomial discretisation is provided in \cref{alg:2d-annuli}. The Dirichlet boundary condition is directly incorporated into the basis by expanding the solution in weighted generalised Zernike annular polynomials.

\begin{algorithm}[ht]
\caption{Generalised Zernike annular polynomials}
\label{alg:2d-annuli}
\begin{algorithmic}[1]
\State{\texttt{EXPAND.} Expand $f$ in the ${\bf Z}^{\rho,(1,1)}$-basis and obtain the coefficient vector $\mathbf{f}$.}

The solve is then decomposed into a parallelisable sequence of one-dimensional solves:

\State{\texttt{ASSEMBLE.}  Per Propositions \labelcref{prop:2d-low} and \labelcref{prop:2d-Lap}, (lazily) assemble the matrices $4 t^{-1}  D^{t,(1,1,m)}_{\mathrm{c},(0,0,m+1)}   D^{t,(0,0,m+1)}_{\mathrm{ab},(1,1,m)}+ \lambda t^{-2}  R^{t,(1,1,m)}_{(0,0,m)} R^{t,(0,0,m)}_{\mathrm{ab}, (1,1,m)}$ and truncate at degree $N$ to form the matrix $A_m \in \mathbb{R}^{k \times k}$ where $k = \ceil{\frac{N+1-m}{2}}$, $m\in \{0,1,\dots,N\}$.}
\State{\texttt{EXTRACT.} Extract the coefficients of ${\bf f}$ for each Fourier mode and truncate to find the coefficient vectors ${\bf f}_{m,j} \in \mathbb{R}^k$, $m \in \{0,1,\dots, N\}$, $j \in \{0,1\}$, $(m,j) \neq (0,0)$.
}
\State{\texttt{SOLVE.} Solve (in parallel) the $2N+1$ sparse linear systems of size $k \times k$, $A_{m} \mathbf{u}_{m,j} = \mathbf{f}_{m,j}$.}
\State{\texttt{INTERLACE.} Interlace the coefficient vectors $\mathbf{u}_{m,j}$ to construct $\mathbf{u}$.}
\State{$u(x,y) \approx {\bf W}^{\rho,(1,1)}(x,y) \mathbf{u}$.}
\end{algorithmic}
\end{algorithm}

\begin{remark}
We note that in Algorithms \labelcref{alg:scaled-and-shifted} and \labelcref{alg:2d-annuli}, one solves $2N+1$ linear systems where $N$ is the truncation degree. With Chebyshev--Fourier series, the size of all the $2N+1$ linear systems is $(N+2)\times(N+2)$ with two dense rows and a leading nonadiagonal band. Whereas the Zernike annular polynomial induced linear systems are tridiagonal with decreasing size with increasing $m$, starting at the size $\ceil{(N+1)/2} \times \ceil{(N+1)/2}$ and ending at the size $1 \times 1$ as also remarked in \cref{fig:poisson-conditioning}. 
\end{remark}

\begin{remark}[Non-constant coefficients]
\label{rem:ncc}
NCCs can be incorporated into \cref{alg:2d-annuli} by utilizing either \cref{prop:jacobi-matrix-r2} or \cref{prop:jacobi-matrix-x-y}. For instance consider the coefficient $\lambda(r^2) \approx \sum_{n=0}^{N_\lambda} \lambda_n T^{\rho}_n(r^2)$ with $N_\lambda \in \mathbb{N}_0$ small, $\lambda_n \in \mathbb{R}$ and $T_n^{\rho}$ denotes the $n$-th degree first-kind Chebyshev polynomials scaled to the interval $[\rho,1]$. Then in order to incorporate $\lambda(r^2)$, into \cref{alg:2d-annuli}, in line 2 we instead assemble the matrices $4 t^{-1}  D^{t,(1,1,m)}_{\mathrm{c},(0,0,m+1)}   D^{t,(0,0,m+1)}_{\mathrm{ab},(1,1,m)}+ t^{-2}  \Lambda_{m} R^{t,(1,1,m)}_{(0,0,m)} R^{t,(0,0,m)}_{\mathrm{ab}, (1,1,m)}$ where 
\begin{align}
 \Lambda_{m} = \sum_{n=0}^{N_\lambda} \lambda_n T^\rho_n(I-t^{-1}X_{t,(1,1,m)}). \label{eq:ncc1}
\end{align}
We note that the sum in \cref{eq:ncc1} can be efficiently evaluated via Clenshaw's algorithm and $\Lambda_m$ has bandwidth $N_\lambda + 1$. 

An NCC dependent solely on $r^2$ preserves the Fourier mode decoupling. This is not true if the NCC depends on $x$ and $y$.
\end{remark}

\section{Spectral element method} 
\label{sec:spectral-element}
In this section we construct a spectral element method for the strong formulation of \cref{eq:helmholtz}. The disk mesh is constructed such that the first cell is the innermost disk, and the subsequent cells are concentric annuli stacked around the disk as visualised in \cref{fig:disk-annulus}. A mesh for the annulus omits the innermost disk cell. 
\begin{figure}[h!]
\centering
\includegraphics[width =0.2 \textwidth]{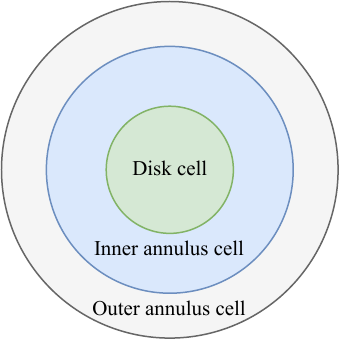}
\caption{The first three cells of a mesh for the spectral element method on a disk. The disk and annuli cells may vary in thickness.}\label{fig:disk-annulus}
\end{figure}

For ease of notation, we consider a two-element method although the methodology extends to an arbitrary number of elements. In the disk and annulus cell, we utilise the Zernike and Zernike annular polynomials, respectively. Consider a mesh with an inner disk cell of radius $\rho$ and the surrounding annulus with inner radius $\rho$ and outer radius 1. Let $\tilde{\bf Z}^{(0)}(r,\theta) \coloneqq {\bf Z}^{(0)}(r/\rho,\theta)$. The goal is to approximate the solution $u(x,y)$ of the Helmholtz equation \cref{eq:helmholtz} with the expansion 
\begin{align}
u(x,y) = 
\begin{pmatrix}
\tilde{\bf Z}^{(0)}(x,y) &  {\bf Z}^{\rho,(0,0)}(x,y) 
\end{pmatrix}
\begin{pmatrix}
\mathbf{u}_0\\
\mathbf{u}_1
\end{pmatrix}.
\end{align}
For any Fourier mode $(m,j)$ we have the lowering relationships
\begin{align*}
\tilde{\bf Z}^{(0)}_{m,j} &= \tilde{\bf Z}^{(2)}_{m,j} R^{(2)}_{m,(0)}, \;\; \Delta \tilde{\bf Z}^{(0)}_{m,j} = \tilde{\bf Z}^{(2)}_{m,j} \tilde{D}^{(2)}_{m,(0)}\\
 {\bf Z}^{\rho,(0,0)}_{m,j} &= {\bf Z}^{\rho,(2,2)}_{m,j} R^{\rho,(2,2)}_{m,(0,0)}, \;\; \Delta {\bf Z}^{\rho,(0,0)}_{m,j} = {\bf Z}^{\rho,(2,2)}_{m,j} D^{\rho,(2,2)}_{m,(0,0)}.
\end{align*}
Note that $\tilde{D}^{(2)}_{m,(0)}$ has a bandwidth of one, $R^{(2,2)}_{m,(0,0)}$ and $D^{\rho,(2,2)}_{m,(0,0)}$ have a bandwidth of three, and $R^{\rho,(2,2)}_{m,(0,0)}$ has a bandwidth of five. 

The Helmholtz operator with Dirichlet boundary conditions takes the following form for each Fourier mode,  $m \in \mathbb{N}_0$, $j \in \{0,1\}$, $(m,j) \neq (0,0)$;
\begin{align}
A_{m} = \begin{pmatrix}
0 & {\bf Z}_{m,0}^{\rho,(0,0)}(1,0) & 0 \\
\tilde{\bf Z}_{m,0}^{(0)}(\rho,0) & - {\bf Z}_{m,0}^{\rho,(0,0)}(\rho,0)  &  1-\mathrm{e}^{-(m+1)} \\
\frac{\mathrm{d}}{\mathrm{d}r} \tilde{\bf Z}_{m,0}^{(0)}(\rho,0) & -\frac{\mathrm{d}}{\mathrm{d}r} {\bf Z}_{m,0}^{\rho,(0,0)}(\rho,0) & 0  \\
\tilde{D}^{(2)}_{m,(0)} + \lambda R^{t,(2)}_{m,(0)} & 0 &  0 \\
0 & D^{\rho,(2,2)}_{m,(0,0)} + \lambda R^{\rho,(2,2)}_{m,(0,0)} &  E_m  
\end{pmatrix}.
\label{eq:Am}
\end{align}
The first row in $A_m$ enforces the boundary condition at $r=1$. As we are discretising the equation in strong form, we are required to enforce continuity of the expansion at $r=\rho$ between cells as well as the outward normal of the derivative, $\partial_n = \frac{\mathrm{d}}{\mathrm{d}r}$, of the expansion. This is achieved by the second and third rows, respectively. The bottom left $2 \times 2$ block matrix contains the recurrence relations between the coefficients and has bandwidth three (one if $\lambda =0$) in the top left block  and bandwidth five (three if $\lambda =0$)  in the bottom right block. The final column is derived from a tau-method and ensures that, after a square truncation, $A_m$ is well-conditioned. 

Let $\tilde{A}_m$ denote the matrix $A_m$ without the last column. Given an even truncation polynomial degree $N$, a discretisation truncates $\tilde{A}_m$ to a rectangular matrix of size $(2k+3) \times (2k+2)$ where $k = \ceil{\frac{N+1-m}{2}}$. The $(1,2)$, $(2,1)$, $(2,2)$, $(3,1)$ and $(3,2)$ blocks are truncated at the size $1 \times (k+1)$. The $(4,1)$ and $(5,2)$ blocks are truncated at the size $k \times (k+1)$. Unfortunately, this truncation leads to a rectangular overdetermined system. This is a well-understood phenomenon when dealing with continuity and boundary conditions in spectral discretisations of problems formulated in strong form. Removing a row from either the $(4,1)$ or $(5,2)$ block leads to severe ill-conditioning for increasing $N$ and $m$ and artificial numerical pollution in the solutions. The final column in $A_m$ resolves this issue and may be understood via the tau-method, c.f.~\cite{Lanczos1938, Ortiz1969} and \cite[Sec.~B]{Burns2020}. 

 In the remainder of this subsection we motivate the final column in $A_m$ via the tau method. For those unfamiliar with tau methods, we note that it is sufficient to consider the purely algebraic description above for the implementation.  Consider the Poisson equation (\cref{eq:helmholtz} with $\lambda = 0)$.  For each Fourier mode, we augment the Poisson equation with the three tau-functions
\begin{align}
\begin{split}
&\tau_{0,m,j}(x, y)  = c_{0,m,j} Z^{\rho, (2,2)}_{N_m,m,j}(x,y), \;\;
 \tau_{1,m,j}(x, y) = c_{1, m,j} Z^{\rho, (2,2)}_{N_m+2,m,j}(x,y), \\ 
& \tau_{2,m,j}(x, y)   = c_{2, m,j} \tilde{Z}^{(2)}_{N_m+2,m,j}(x,y),
\end{split} 
\end{align}
where $N_m = N$ if $m$ is even, otherwise $N_m = N-1$. Thus for each Fourier mode we wish to solve
\begin{align}
\begin{split}
&\Delta \tilde{\bf Z}^{(0)}_{m,j}(x,y) \mathbf{u}_{0, m, j} + \Delta {\bf Z}^{\rho,(0,0)}_{m,j}(x,y) \mathbf{u}_{1, m, j} + \sum_{i=0}^2 c_{i,m,j} \tau_{i,m,j}(x,y) \\
&  = \tilde{\bf Z}^{(2)}_{m,j}(x,y) \tilde{D}^{(2)}_{m,(0)} \mathbf{u}_{0, m, j} +  c_{2, m,j}  \tilde{Z}^{(2)}_{N_m+2,m,j}(x,y)\\
&  \indent + {\bf Z}^{\rho,(2,2)}_{m,j}(x,y)  D^{\rho,(2,2)}_{m,(0,0)} \mathbf{u}_{1, m, j} + \sum_{i=0}^1 c_{i,m,j} Z^{\rho, (2,2)}_{N_m+2i,m,j}(x,y)
= f(x,y),
%& \indent +  c_{m,j,0} \tau_{m,j,0}(x,y) 
%+c_{m,j,1}\tau_{m,j,1}(x,y) 
%+c_{m,j,2}\tau_{m,j,2}(x,y) 
%= f(x,y),
\end{split}
\label{eq:tau1}
\end{align}
with the boundary condition, where $x^2+y^2=1$,
\begin{align}
{\bf Z}^{\rho,(0,0)}_{m,j}(x,y) \mathbf{u}_{1, m, j} &= 0,
\label{eq:tau2}
\end{align}
and the continuity conditions, where $x^2+y^2=\rho^2$,
\begin{align}
\tilde{\bf Z}^{(0)}_{m,j}(x,y) \mathbf{u}_{0, m, j} - {\bf Z}^{\rho,(0,0)}_{m,j}(x,y) \mathbf{u}_{1, m, j} +  (1-\mathrm{e}^{-(m+1)}) c_{0,m,j} &= 0, \label{eq:tau3} \\
\frac{\mathrm{d}}{\mathrm{d}r} \tilde{\bf Z}^{(0)}_{m,j}(x,y) \mathbf{u}_{0, m, j} - \frac{\mathrm{d}}{\mathrm{d}r} {\bf Z}^{\rho,(0,0)}_{m,j}(x,y) \mathbf{u}_{1, m, j} &= 0. \label{eq:tau4}
\end{align}
The final two tau-functions, $i\in\{1,2\}$, are those chosen in a typical ultraspherical method and result in two columns and associated rows with only a one on the diagonal. Thus these may be immediately row eliminated. However, $\tau_{0,m,j}(x,y)$ must be included as part of the linear system, resulting in the (truncated) column vector $E_m = (0, \cdots, \;\; 0, \; \; 1)^\top \in \mathbb{R}^k$. The inclusion of the first tau-function in the continuity condition is atypical. Here it significantly improves the conditioning of $A_m$ as $m \to \infty$. This is because, at $x^2+y^2=\rho^2$,  $\tilde{Z}^{(0)}_{n,m,j}(x,y) \approx \mathcal{O}(1)$, whenever $(n-m) \approx \mathcal{O}(1)$, whereas ${Z}^{\rho,(0,0)}_{n,m,j}(x,y) \approx \mathcal{O}(\rho^m)$. Left unchecked this introduces ill-conditioning for increasing $m$.

 Let ${\bf 0}$ denote the infinite zero matrix and let the subscripts $_{n,k}$ denote truncating an infinite matrix at the $n$th row and $k$th column. We use $_n$ to denote truncating an infinite vector at the $n$th row. Then, more explicitly, after truncating the expansion of the right-hand side at degree $N$, \cref{eq:tau1} induces, for each Fourier mode:
\begin{align}
\begin{split}
&\left( \tilde{Z}^{(2)}_{m,m,j} \; \cdots \; \tilde{Z}^{(2)}_{N_m+2,m,j} | Z^{\rho, (2,2)}_{m,m,j} \cdots Z^{\rho, (2,2)}_{N_m+2,m,j} \right)\\
&\times
\begin{pmatrix}
[\tilde{D}^{(2)}_{m,(0)}]_{k, k+1} & {\bf 0}_{k, k+1}  &  {\bf 0}_{k, 1}& {\bf 0}_{k, 1} & {\bf 0}_{k, 1} \\
{\bf 0}_{1, k+1} & {\bf 0}_{1, k+1} & 0 & 0 & 1\\
{\bf 0}_{k, k+1} & [D^{\rho,(2,2)}_{m,(0,0)}]_{k, k+1} & [E_m]_{k} &  {\bf 0}_{k, 1} &  {\bf 0}_{k, 1}\\
 {\bf 0}_{1, k+1} &  {\bf 0}_{1, k+1}  & 0 & 1 & 0\\
\end{pmatrix}
\begin{pmatrix}
[{\bf u}_{0,m,j}]_k\\
[{\bf u}_{1,m,j}]_k\\
c_{0,m,j}\\
c_{1,m,j}\\
c_{2,m,j}
\end{pmatrix}\\
& \indent = \left( \tilde{Z}^{(2)}_{m,m,j} \; \cdots \; \tilde{Z}^{(2)}_{N_m,m,j} | Z^{\rho, (2,2)}_{m,m,j} \cdots Z^{\rho, (2,2)}_{N_m,m,j} \right) \begin{pmatrix}
[{\bf f}_{0,m,j}]_k\\
[{\bf f}_{1,m,j}]_k
\end{pmatrix}
\end{split} \label{eq:tau5}
\end{align}
The boundary condition \cref{eq:tau2} induces at $x^2+y^2= 1$:
\begin{align}
\left(Z^{\rho, (0,0)}_{m,m,j} \cdots Z^{\rho, (0,0)}_{N_m,m,j} \right) [{\bf u}_{1,m,j}]_k =0, \label{eq:tau6}
\end{align}
and the continuity conditions \cref{eq:tau3} and \cref{eq:tau4} induce at $x^2+y^2 = \rho^2$
\begin{align}
\begin{split}
&\left( \tilde{Z}^{(0)}_{m,m,j} \; \cdots \; \tilde{Z}^{(0)}_{N_m,m,j} | -Z^{\rho, (0,0)}_{m,m,j} \cdots -Z^{\rho, (0,0)}_{N_m,m,j} \right) 
\begin{pmatrix}
[{\bf u}_{0,m,j}]_k\\
[{\bf u}_{1,m,j}]_k
\end{pmatrix}\\
&\indent+ c_{0,m,j} (1-\E^{-(m+1)}) = 0,
\end{split} \label{eq:tau8}
\end{align}
and
\begin{align}
\frac{\mathrm{d}}{\mathrm{d}r}\left( \tilde{Z}^{(0)}_{m,m,j} \; \cdots \; \tilde{Z}^{(0)}_{N_m,m,j} | -Z^{\rho, (0,0)}_{m,m,j} \cdots -Z^{\rho, (0,0)}_{N_m,m,j} \right) 
\begin{pmatrix}
[{\bf u}_{0,m,j}]_k\\
[{\bf u}_{1,m,j}]_k
\end{pmatrix} = 0. \label{eq:tau9}
\end{align}
Concatenating \cref{eq:tau5}--\cref{eq:tau9} together, and row eliminating $c_{1,m,j}$ and $c_{2,m,j}$, results in $A_m$.

%In this work, we simply solve the least squares problem and found that for sufficiently large $n$ the residuals are of the order of $\mathcal{O}(10^{-15})$ or smaller. The algorithm for the spectral element method is given in \cref{alg:spectral-element}.
%\begin{enumerate}
%\itemsep=0pt
%\item 
%\end{enumerate}
%For any Fourier mode $(m,j)$ we have the lowering relationships
%\begin{align*}
%{\bf Z}^{\rho, (0,0)}_{m,j} &= {\bf Z}^{\rho, (2,2)}_{m,j} R^{t,(2,2)}_{m,(0,0)}, \;\; \Delta {\bf Z}^{\rho, (0,0)}_{m,j} = {\bf Z}^{\rho, (2,2)}_{m,j} D^{t,(2,2)}_{m,(0,0)}\\
% {\bf Z}^{(0,0, 0)}_{m,j} &= {\bf Z}^{(2,2,0)}_{m,j} R^{t,(2,2,0)}_{m,(0,0,0)}, \;\; \Delta {\bf Z}^{(0,0, 0)}_{m,j} = {\bf Z}^{(2,2,0)}_{m,j} D^{t,(2,2,0)}_{m,(0,0,0)}.
%\end{align*}
%Note that $D^{t,(2,2)}_{m,(0,0)}$ has a bandwidth of one, $R^{t,(2,2)}_{m,(0,0)}$ and $D^{t,(2,2,0)}_{m,(0,0,0)}$ have a bandwidth of three, and $D^{t,(2,2,0)}_{m,(0,0,0)}$ has a bandwidth of five. The solve is then decomposed into a parallelisable sequence of one-dimensional solves as follows:
%\begin{enumerate}
%\setcounter{enumi}{1}
%\itemsep=0pt
%\item 
%\end{align*}
%\item ;
%\item Solve $A_m {\bf u}_{m, j} = \left(0, \; 0, \; 0, \; {\bf f}_{m,j}^\top \right)^\top$, $m = 0,1,2,\dots$, $j=0,1$;
%\item Split and interlace ${\bf u}_{m, j}$ to form the coefficient vectors ${\bf u}_Z$ and ${\bf u}_R$.
%\end{enumerate}

\begin{remark}
A similar tau-method is used for the spectral element method where the basis is the Chebyshev--Fourier series on the annuli cells. In this case the column vector $E_m$ remains the same but the  $1-\mathrm{e}^{-(m+1)}$  in the continuity condition row is omitted.  The spectral element method considered in \cite{anders2023} (and implemented in \texttt{Dedalus} \cite{Burns2020}) is similar to the one discussed here with a Zernike discretisation for the disk cell and a Chebyhsev--Fourier series for the annulus cell. They utilise the test basis ${\bf C}^{(2)} / r^2$ in the radial direction (which is equivalent to multiplying the equation \cref{eq:helmholtz} by $r^2$ but allows for better automation) and a different tau term in the disk (in terms of the original trial basis $\tilde{\bf Z}^{(0)}$) that had better performance for the problems they considered. 
\end{remark}

\begin{algorithm}[ht]
\caption{Spectral element method}
\label{alg:spectral-element}
\begin{algorithmic}[1]
\State{\texttt{EXPAND.} Restrict $f(x,y)$ to the disk and annulus cells and expand in the bases $\tilde{\bf Z}^{(2)}(x,y)$ and ${\bf Z}^{\rho,(2,2)}(x,y)$ to find the coefficient vectors $\mathbf{f}_0$ and $\mathbf{f}_1$, respectively.}

The solve is then decomposed into a parallelisable sequence of one-dimensional solves:

\State{\texttt{ASSEMBLE.} Assemble and truncate $A_m$ as defined in \cref{eq:Am} to form a square matrix of size  $(2k+3) \times (2k+3)$ where $k = \ceil{\frac{N+1-m}{2}}$, $m \in \{0,1,\dots,N\}$, where $N$ is the truncation degree.}
\State{\texttt{EXTRACT.} Extract the coefficients from ${\bf f}_0$ and ${\bf f}_1$  associated with the $(m, j)$-mode to form the coefficient vector ${\bf f}_{m,j} \in \mathbb{R}^{2k}$, $m \in \{0,1,\dots, N\}$, $j \in \{0,1\}$, $(m,j) \neq (0,0)$.
}
\State{\texttt{SOLVE.} Solve (in parallel) for $m \in  \{0,1,\dots, N\}$, $j \in \{0,1\}$, $(m,j) \neq (0,0)$,
\begin{align}
A_m 
\begin{pmatrix}
{\bf u}_{m, j}\\
c_{0,m,j}
\end{pmatrix}
= 
\begin{pmatrix}
0&
0&
0&
{\bf f}_{m,j}^\top
\end{pmatrix}^\top.
\end{align}
}
\State{\texttt{INTERLACE.} Split and interlace ${\bf u}_{m, j}$ to form the coefficient vectors ${\bf u}_0$ and ${\bf u}_1$.}
\end{algorithmic}
\end{algorithm}

\section[PDEs in annuli]{Numerical examples}
\label{sec:examples}

\textbf{Code availability:} For reproducibility, an implementation of the optimal complexity algorithms for semiclassical Jacobi polynomials may be found in  \texttt{ClassicalOrthogonalPolynomials.jl} \cite{ClassicalPoly.jl2023} and \\ \texttt{SemiclassicalOrthogonalPolynomials.jl} \cite{SemiPoly.jl2023}. The Zernike annular polynomials are implemented in the package \texttt{AnnuliOrthogonalPolynomials.jl} \cite{AnnPoly.jl2023}. Their fast analysis and synthesis operators are implemented in \texttt{FastTransforms.jl} \cite{FastTransforms.jl2023}. Scripts to generate the plots and solutions of the numerical examples in this section may be found in \texttt{AnnuliPDEs.jl} \cite{annulipdes.jl} which has been archived on Zenodo \cite{annulipdes-zenodo}.

In the examples we consider the problem of solving PDEs in the two-dimensional annulus or disk, in particular the Poisson and Helmholtz equations with a homogeneous Dirichlet boundary condition.  We consider the sparse spectral methods: Chebyshev--Fourier series (via \cref{alg:scaled-and-shifted}), Zernike annular polynomials (via \cref{alg:2d-annuli}) and spectral element methods (via \cref{alg:spectral-element}). We emphasise that, for all the methods, the blocks decouple and one solves over each Fourier mode separately, reducing the two-dimensional solve to $2N+1$ one-dimensional solves, when truncating at the polynomial degree $N$. 

%We now discretise the Helmholtz equation \'a la the Ultraspherical method \cite{Olver2013}:
%\begin{enumerate}
%\itemsep=0pt
%\item Expand $f$ in the ${\bf Z}^{\rho,(1,1,0)}$-basis and obtain the coefficient vector $\mathbf{f}$;
%\item Per \cref{prop:2d-Lap} and \cref{prop:2d-low}, (lazily) assemble the matrices $A_m = D^{t,(0,0,m+1)}_{\mathrm{c},(1,1,m)}   D^{t,(1,1,m)}_{\mathrm{ab},(0,0,m+1)}+ \lambda R^{t,(0,0,m)}_{(1,1,m)} R^{t,(1,1,m)}_{\mathrm{ab}, (0,0,m)}$, $m=0,1,\dots$; 
%\item Extract the coefficients of ${\bf f}$ for each Fourier mode, ${\bf f}_{m,j}$ $m=0,1,\dots$, $j=0,1$. 
%\item Solve (in parallel) the sparse linear systems $A_{m} \mathbf{u}_{m,j} = \mathbf{f}_{m,j}$; 
%\item Interlace the coefficient vectors $\mathbf{u}_{m,j}$ to construct $\mathbf{u}$. 
%\item $u(x,y) \approx {\bf W}^{\rho,(1,1,0)}(x,y) \mathbf{u}$. 
%\end{enumerate}

\textbf{Error measurement:} In all examples we measure the $\ell^\infty$-norm error of the approximation on each cell as evaluated on a heavily oversampled generalised Zernike annular grid \cref{eq:zernike-grid}.

\subsection{Forced Helmholtz equation}
\label{sec:forced-helmholtz}
In this example we consider the forced Helmholtz equation \cref{eq:helmholtz} on the annuli domains with inner radii $\rho = 0.2$, $0.5$, and $0.8$, respectively, with  a non-constant coefficient $\lambda(r) = 80^2 r^2$,  and the data $f(x,y) = \sin(100x)$. Due to the large positive choice of $\lambda$  (as $r\to1$),  the solution supports oscillations that are, in general, difficult to resolve.  The non-constant coefficient induces a bounded operator which is incorporated into the solver as discussed in \cref{rem:ncc}.  Moreover, the right-hand side features a high Fourier mode component. As the explicit solution is unavailable, we measure the errors against over-resolved approximations with coefficients that have long decayed to \texttt{Float64} machine precision.
 
We provide plots of the solutions in \cref{fig:forced-helmholtz} as well as the corresponding convergence plots of the Chebyshev--Fourier series (via \cref{alg:scaled-and-shifted}) and weighted Zernike annular discretisation (via \cref{alg:2d-annuli}) in \cref{fig:forced-helmholtz-convergence}.
\begin{figure}[h!]
\centering
\subfloat[$\rho=0.2$]{\includegraphics[width =0.32 \textwidth]{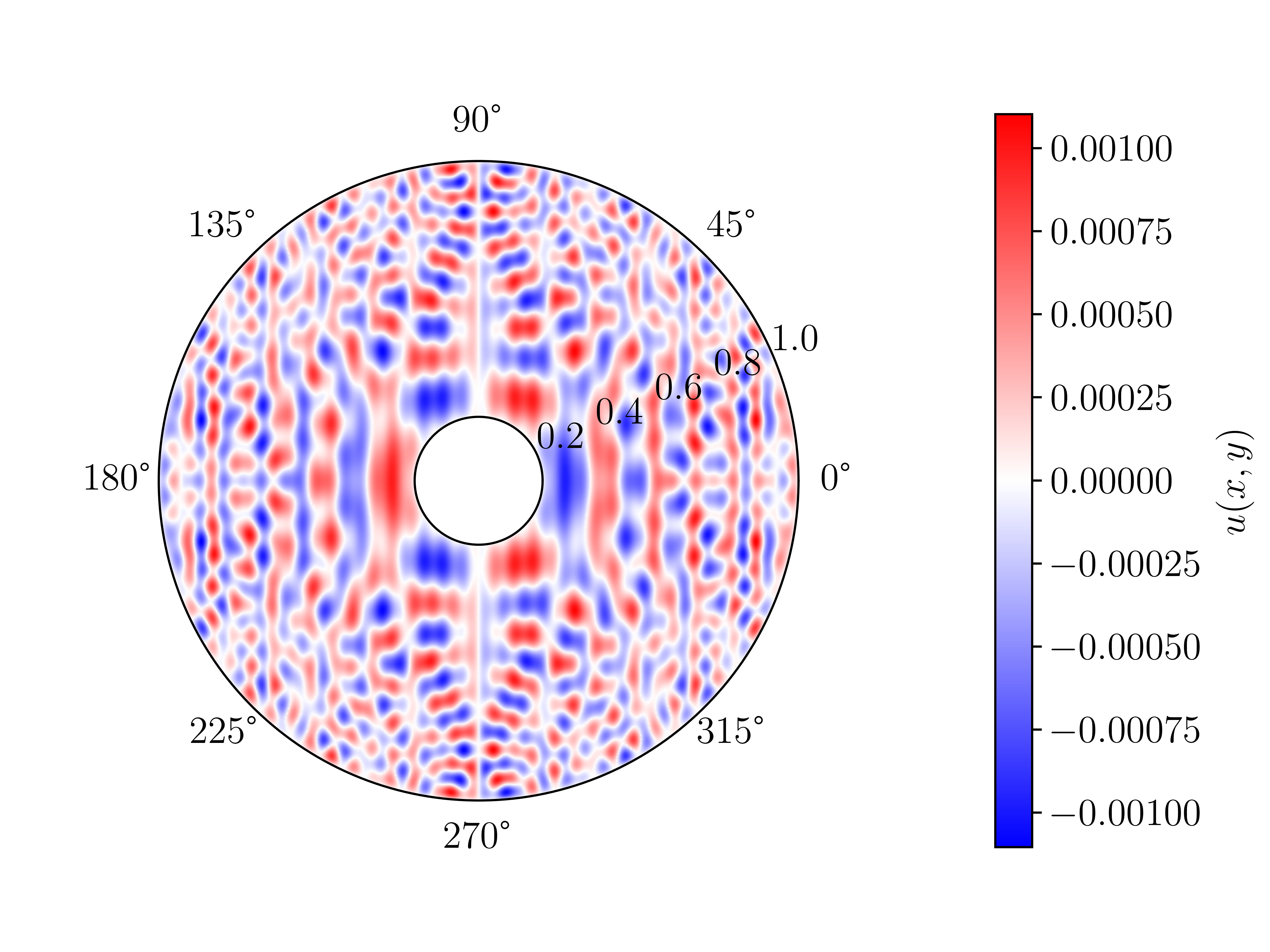}}
\subfloat[$\rho=0.5$]{\includegraphics[width =0.32 \textwidth]{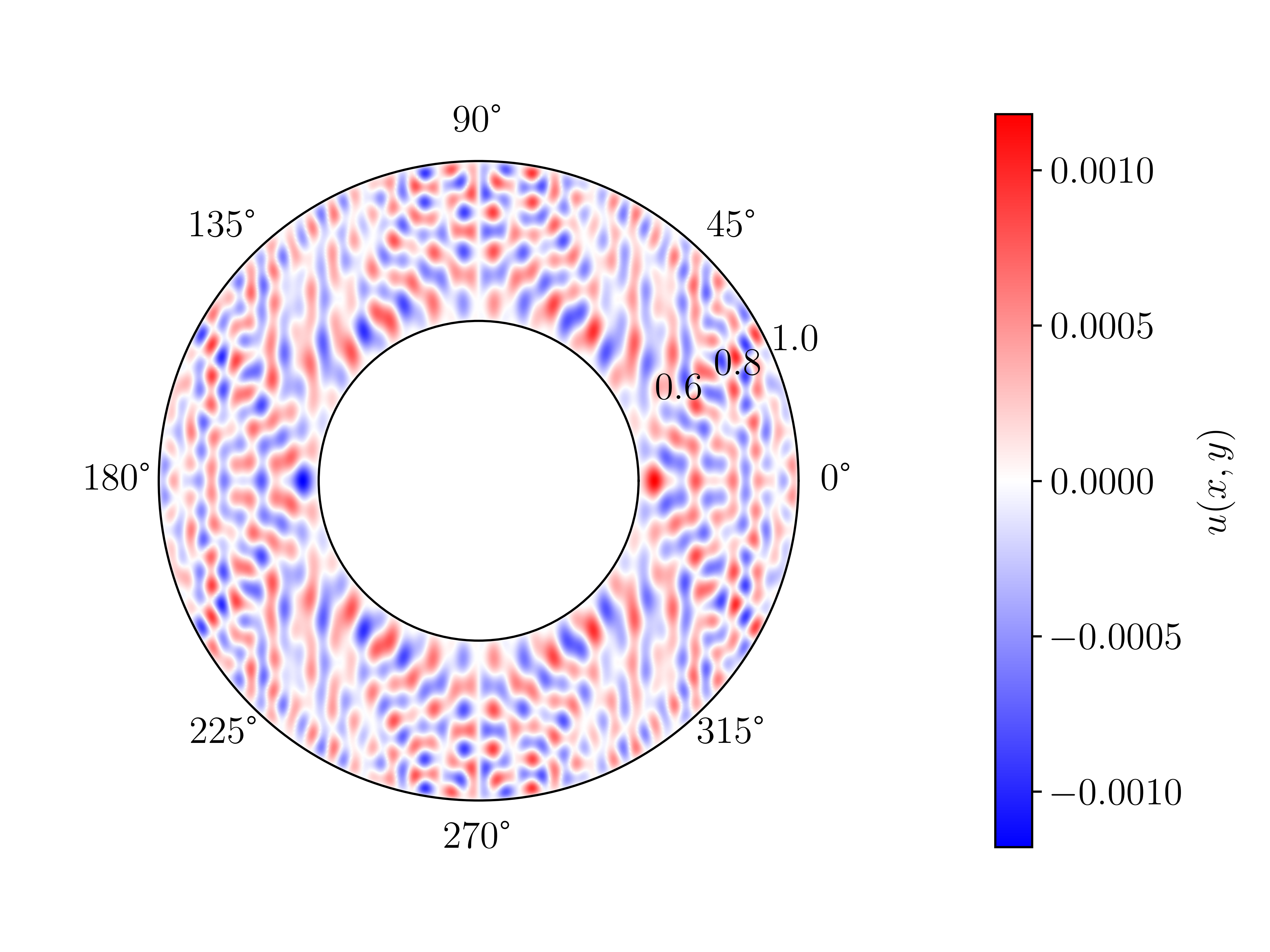}}
\subfloat[$\rho=0.8$]{\includegraphics[width =0.32 \textwidth]{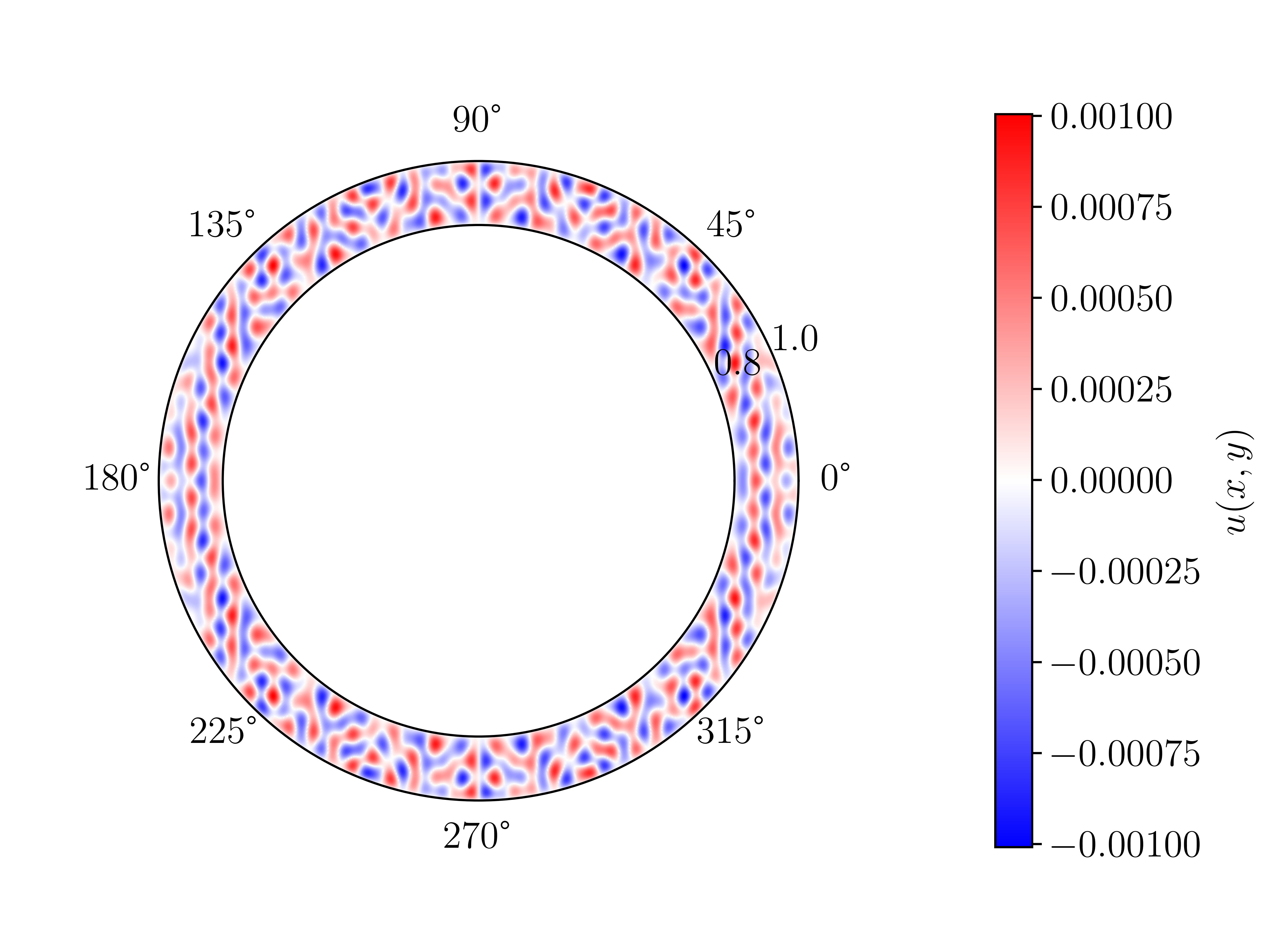}}
\caption{Plots of the solutions of \cref{eq:helmholtz} with  $\lambda(r) =80^2r^2$  and the right-hand side $f(x,y) = \sin(100x)$ on the annuli domains with inradii $\rho = 0.2$, $0.5$, and $0.8$, respectively (\cref{sec:forced-helmholtz}).}\label{fig:forced-helmholtz}
\end{figure}

\begin{figure}[h!]
\centering
\subfloat[$\rho=0.2$]{\includegraphics[width =0.32 \textwidth]{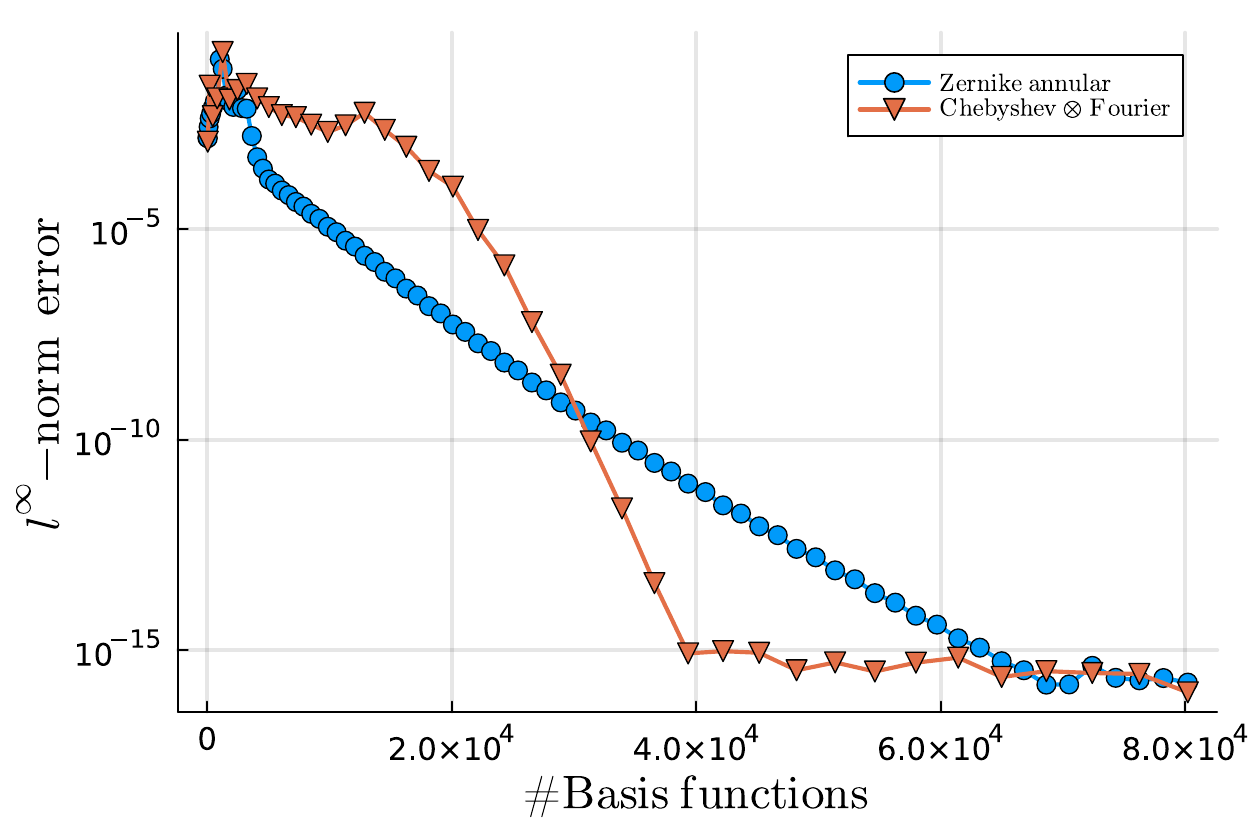}}
\subfloat[$\rho=0.5$]{\includegraphics[width =0.32 \textwidth]{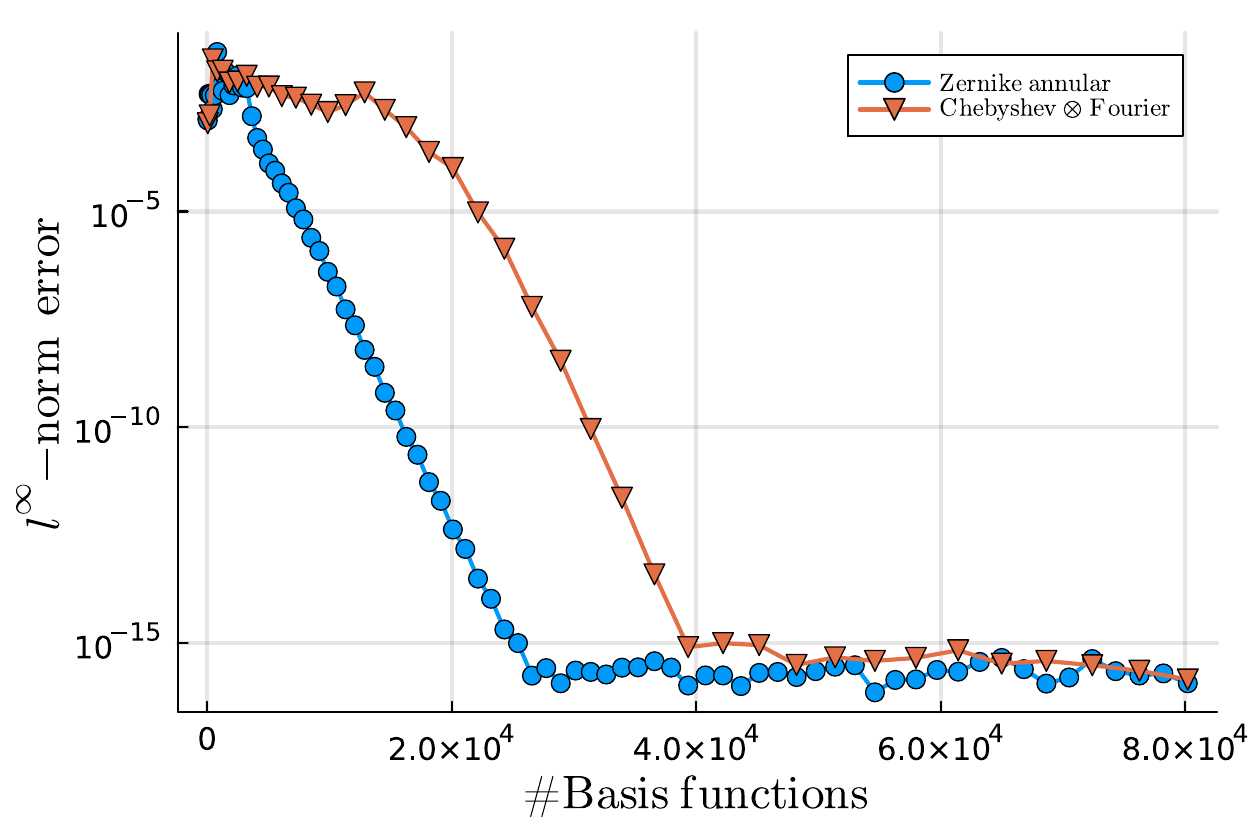}}
\subfloat[$\rho=0.8$]{\includegraphics[width =0.32 \textwidth]{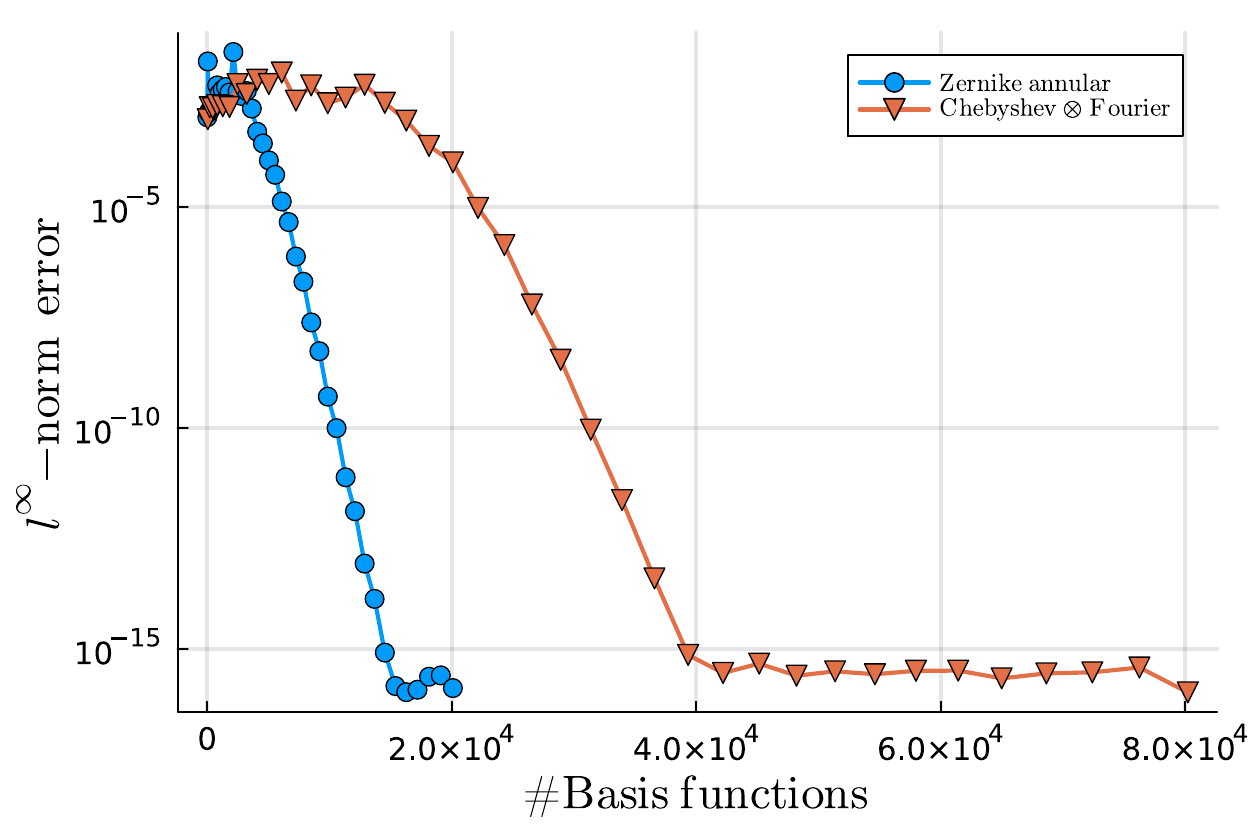}}
\caption{Convergence of the spectral methods for solving \cref{eq:helmholtz} with  $\lambda(r) =80^2r^2$  and the right-hand side $f(x,y) = \sin(100x)$ on the annuli domains with inradii $\rho = 0.2$, $0.5$, and $0.8$, respectively. The Zernike annular polynomial approximation of the solution converges slower than the Chebyshev--Fourier series analogue for $\rho=0.2$  but considerably faster when $\rho=0.5$ and $0.8$.}\label{fig:forced-helmholtz-convergence}
\end{figure}

The weighted Zernike annular discretisation converges much quicker than the Chebyshev--Fourier series when $\rho = 0.5$ and $0.8$ and considerably slower when $\rho=0.2$. In fact, the convergence profile of the Chebyshev--Fourier series is roughly independent of $\rho$, reaching \texttt{Float64} machine precision with  39,339 coefficients (truncation degree $N=139$)  on all three domains. In contrast, the convergence of the weighted Zernike annular method largely depends on $\rho$. When the the inradius is large, the Zernike annular discretisation requires fewer basis functions to resolve functions with large Fourier mode components and fluctuating behaviour near the outer boundary. The degradation in performance as $\rho \to 0$ is likely caused by the proximity of the inner boundary to the logarithmic singularity of the solutions at the origin (outside of the domain). The cause of this is due to the quadratic variable transformation in the radial direction for the Zernike annular discretisation, while the Chebyshev--Fourier series uses an affine variable transformation. The former brings the logarithmic singularity at the origin much closer in the Bernstein ellipse metric than the latter, cf.~\cite[Fig.~4]{Boyd2011}. 

We conclude by noting that, in practical examples, the Zernike annular polynomials are likely to be used as part of a spectral element method (such as in Sections \labelcref{sec:disc-data} and \labelcref{sec:example:disc-ann}). Thus typically $0 \ll \rho < 1$ as it is highly unlikely one would utilise annuli cells with a very small inner radius relative to the outer boundary.

\subsection{Poisson equation with a modified Gaussian bump}

Consider \cref{eq:helmholtz} with $\lambda = 0$, $\rho = 0.2$, and, for some coefficients $a \in \mathbb{R}$, $b,c \in [\rho, 1]$,
\begin{align}
f(x,y) = -4a\E^{-a((x-b)^2 + (y-c)^2)}(-a(b^2 - 2bx + c^2 - 2cy + x^2 + y^2) + 1).
\label{eq:manufactured-gaussian}
\end{align}
The right-hand side $f$ is a smooth function and emulates a modified Gaussian bump. For sufficiently positive $a \gg 0$, the solution is approximately $u(x,y) \approx \E^{-a((x-b)^2 + (y-c)^2)}$. In \cref{fig:manufactured-gaussian-bump} we plot the solution as well as the convergence plot of the two spectral methods against the number of basis functions utilised in the discretisation. We see that the weighted Zernike annular polynomial discretisation converges the quickest.

\begin{figure}[h!]
\centering
\includegraphics[width =0.49 \textwidth]{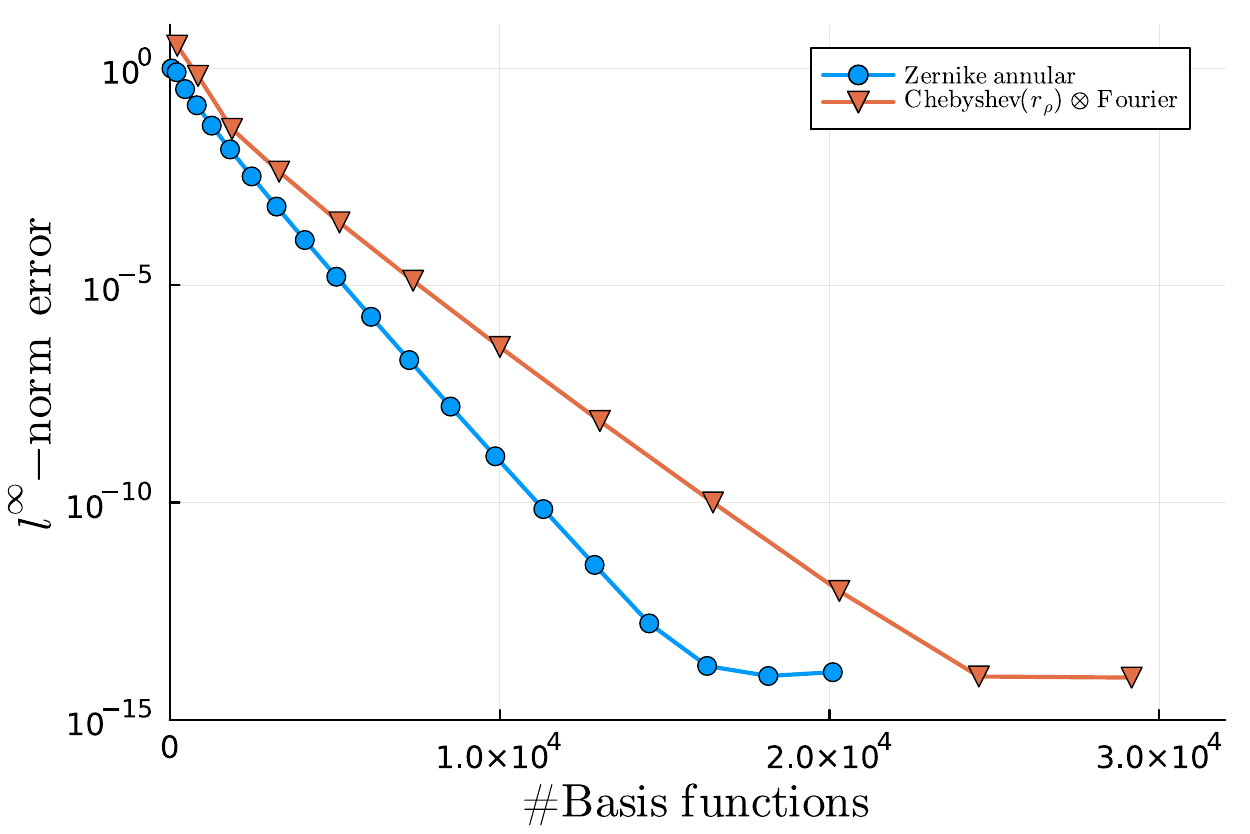}
\includegraphics[width =0.49 \textwidth]{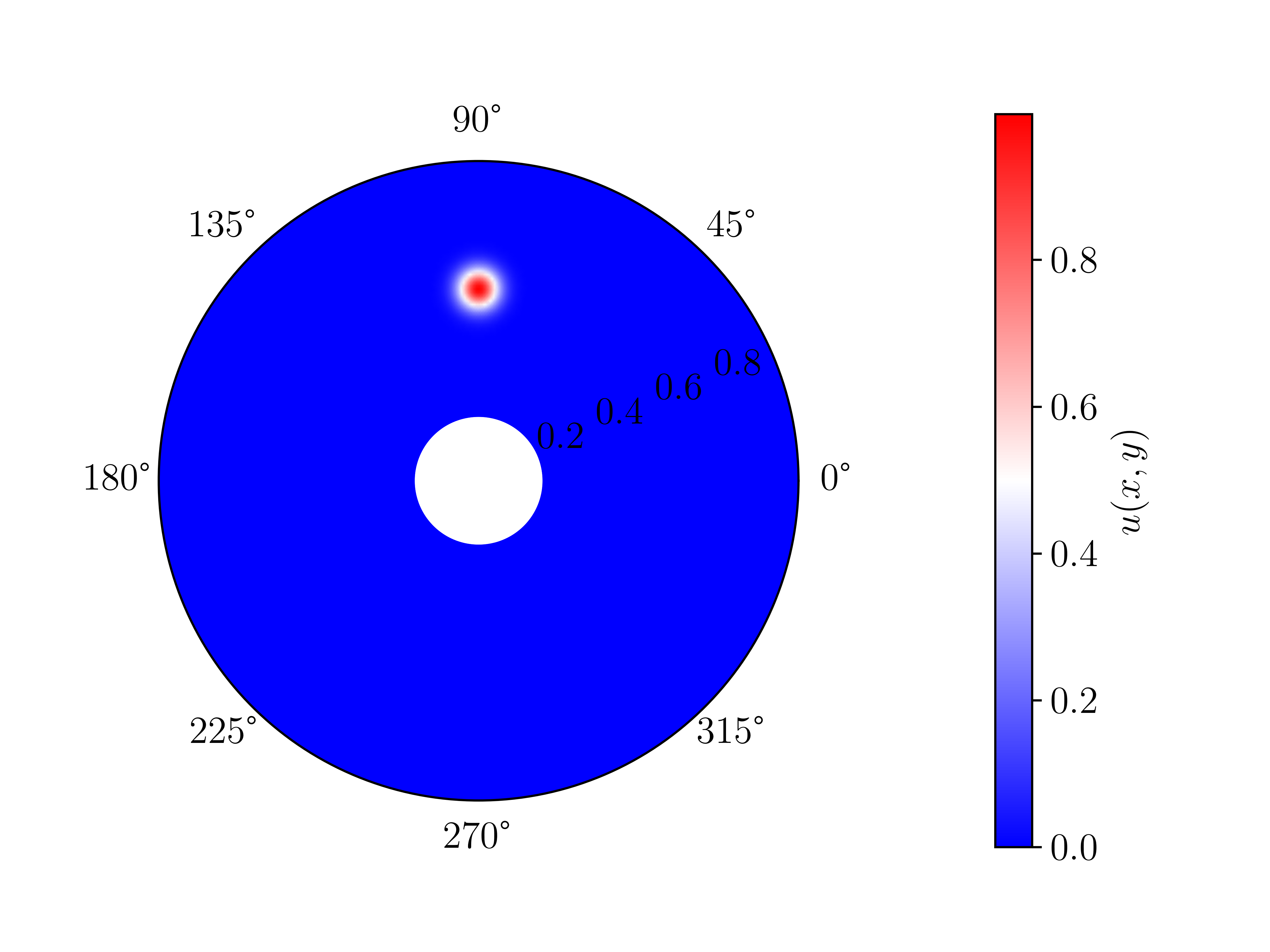}
\caption{Convergence of the spectral methods for solving \cref{eq:helmholtz} with $\lambda =0$, $\rho = 0.2$, and the right-hand side \cref{eq:manufactured-gaussian} with $a = 250$, $b=0$, and $c=0.6$ (left). Plot of the solution (right). The Zernike annular polynomial approximation of the solution, with a smooth right-hand side, converges  faster than the Chebyshev--Fourier series analogue.}\label{fig:manufactured-gaussian-bump}
\end{figure}

\subsection{Discontinuous variable coefficients and data on a disk}
\label{sec:disc-data}
Consider a variable-coefficient Helmholtz equation \cref{eq:helmholtz} such that the coefficient $\lambda(r)$ and the right-hand side have discontinuities in the radial direction. In this example we show that, although a traditional spectral method struggles to resolve the jumps, the spectral element method of \cref{sec:spectral-element} performs particularly well. Let the domain be the unit disk $\Omega_0$. Define the jump function $\kappa(r)$, $\kappa_0, \kappa_1 \in \mathbb{R}$,
\[
\kappa(r) \coloneqq
\begin{cases}
\kappa_0 & \text{if} \; 0 \leq r < \rho,\\
\kappa_1 & \text{if} \; \rho \leq r \leq 1.
\end{cases}
\]
Consider the continuous solution:
\begin{align}
\begin{split}
&u(x,y) = \left( \sum_{i=1}^5 d_i \E^{-a_i((x-b_i)^2 + (y-c_i)^2)} \right)\\
&\indent\times\begin{cases}
\kappa_0r^2/4 + (\kappa_1 - \kappa_0)\rho^2/4 - \kappa_1/4 + (\kappa_0 - \kappa_1)\rho^2\log(\rho)/2 & \text{if} \; 0 \leq r \leq \rho, \\
 \kappa_1r^2/4 - \kappa_1/4 + (\kappa_0 - \kappa_1)\rho^2\log(r)/2 & \text{if} \; \rho < r \leq 1.
\end{cases}
\end{split} \label{eq:u-element}
\end{align}
Throughout this example we fix $d_i = 1$ and $a_i = -10i$ for all $i$. Moreover, for $i \in \{1,2,3,4\}$, $(b_i,c_i) = (\rho\cos(\theta_i), \rho\sin(\theta_i))$ where $\theta_i = 0, \pi/2, \pi/3, 5\pi/4$, respectively, and $(b_5,c_5) = (0.9\cos(3\pi/4), 0.9\sin(3\pi/4))$.  Furthermore, $\kappa(r) = \kappa_0 = 100$ if $0\leq r \leq \rho$ and $\kappa(r) = \kappa_1 = 1$ if $\rho <r\leq 1$ with $\rho = 1/2$. We consider two equations, with the homogeneous Dirichlet boundary condition $u|_{\partial \Omega_0} = 0$,
\begin{align}
\Delta u(x,y) &= f_1(x,y), \label{eq:spectral-poisson} \\
(\Delta + \kappa(r)) u(x,y) &= f_2(x,y). \label{eq:spectral-helmholtz}
\end{align}
The first equation \cref{eq:spectral-poisson} is the Poisson equation posed on the unit disk with the right-hand side
\begin{align}
\begin{split}
f_1(x,y) 
&= \sum_{i=1}^5 d_i \E^{-a_i((x-b_i)^2 + (y-c_i)^2)}\\
& \indent \times \left(-4a_i(-a_i(b_i^2 - 2b_ix + c_i^2 - 2c_iy + x^2 + y^2) + 1) + \kappa(r) \right).
\end{split}
\label{eq:f-element1}
\end{align}
Note that $f_1(x,y)$ has discontinuities in the radial direction at $r=1/2$. The second equation \cref{eq:spectral-helmholtz} is a variable-coefficient Helmholtz equation where the Helmholtz coefficient is $\lambda(r) = \kappa(r)$. Hence, the coefficients of the equation also contain jumps in the radial direction. The corresponding right-hand side is
\begin{align}
f_2(x,y) =  f_1(x,y) + \kappa(r) u(x,y).
\label{eq:f-element2}
\end{align}

In \cref{fig:poisson-conditioning-2-element} we plot the conditioning and size of the Laplacian matrices for increasing Fourier mode $m$ with a truncation degree $N=100$. We observe that the two-element method of Zernike and Zernike annular polynomials has, on average, better conditioning and smaller matrix size. The discontinuous right-hand sides $f_1$ and $f_2$, together with the continuous solution \cref{eq:u-element}, are those provided in \cref{sec:introduction} in \cref{fig:spectral-element-plots}.  In \cref{fig:spectral-element-convergence} we plot the convergence of various spectral methods.  For the Poisson equation \cref{eq:spectral-poisson} we compare the spectral methods: (1) a one-element weighted Zernike spectral method (\cref{alg:2d-annuli}), (2) a two-element method Zernike (annular) method (\cref{alg:spectral-element}) and (3) a two-element method via a Zernike discretisation on the disk cell and a Chebyshev--Fourier series on the outer annulus cell. In the two-element methods, the inner cell is the disk $\{ 0 \leq r \leq 1/2\}$ and the outer cell is the annulus $\{ 1/2\leq r \leq 1\}$. For the variable-coefficient Helmholtz equation we only plot the convergence of (2) and (3). A one-element method would require resolving the function $\kappa(r)$ over the whole disk, requiring high-degree polynomials, and would result in systems with large bandwidths leading to impractical methods. Whereas in the spectral element method, $\kappa(r)$ is a constant in each cell and the problem effectively reduces to a non-variable Helmholtz problem on each cell. We avoid any strategy that utilises a Chebyshev--Fourier series discretisation in a disk cell due to well-studied issues at the origin \cite{Boyd2011, VasilDisk, Wilber2017}.

The two-element Zernike (annular) method of \cref{alg:spectral-element} performs the best. This is followed by the two-element Zernike--Chebyshev--Fourier series method. The one-element method struggles to reduce the accuracy below $\mathcal{O}(10^{-4})$.
\begin{figure}[h!]
\centering
\subfloat[Condition number.]{\includegraphics[width =0.35\textwidth]{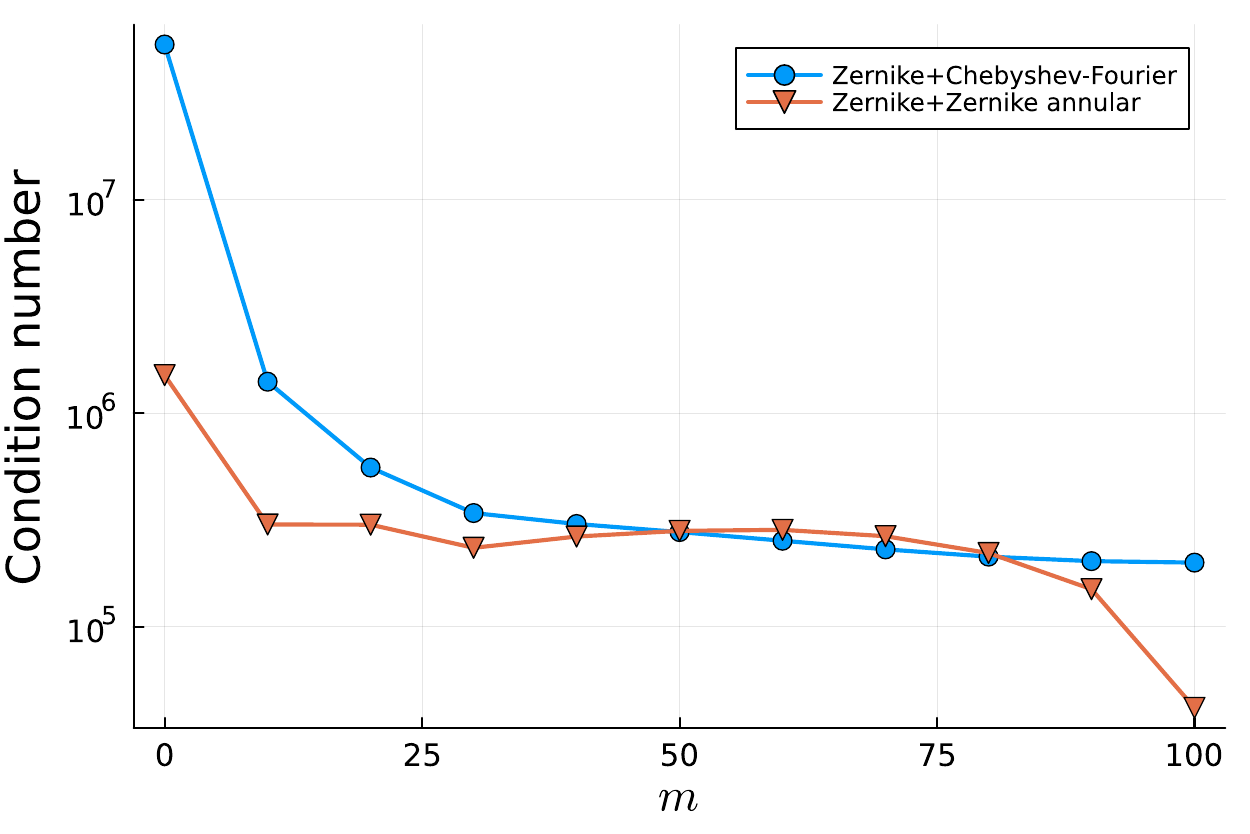}}
\subfloat[Matrix size.]{\includegraphics[width =0.35\textwidth]{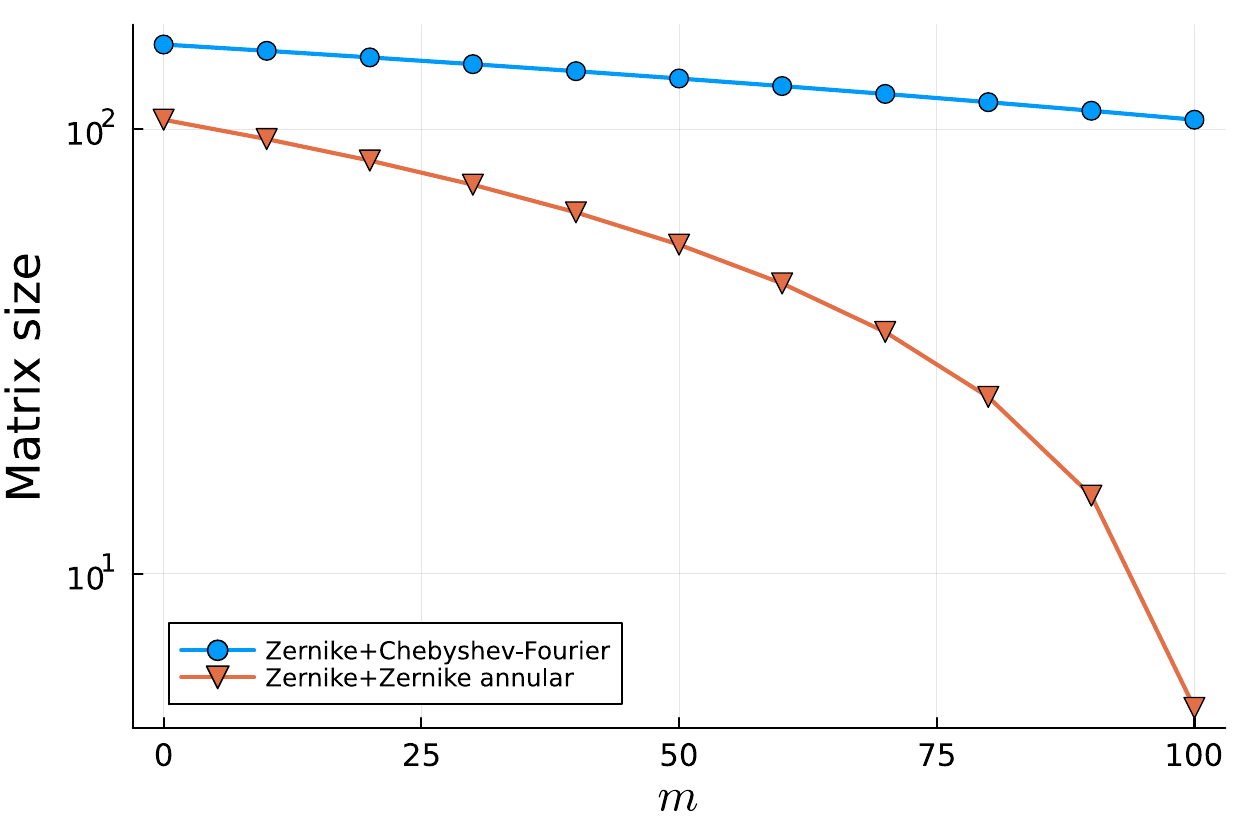}}
\caption{The condition number and matrix size of the Laplacian matrix, for increasing Fourier mode $m$, of the two-element methods with a Zernike discretisation on the disk cell and either a Chebyshev--Fourier series or a Zernike annular discretisation on the annulus cell (inner radius $\rho=1/2$) on the domain $\Omega_{0}$ with truncation degree $N=100$. We observe that the Laplacian matrices for the two-element Zernike (annular) polynomials is better conditioned, on average, and smaller in size.}\label{fig:poisson-conditioning-2-element}
\end{figure}

\begin{figure}[h!]
\centering
\subfloat[Convergence of \cref{eq:spectral-poisson}.]{\includegraphics[width =0.49 \textwidth]{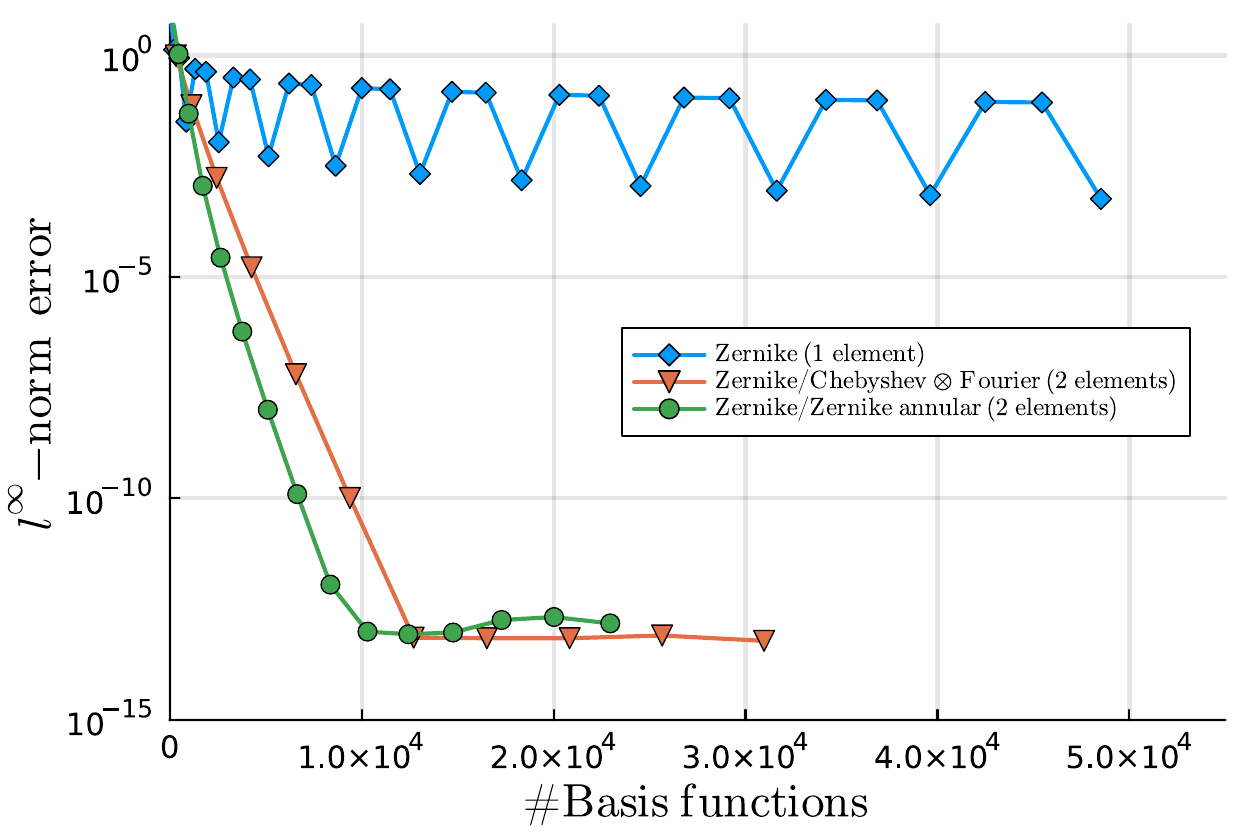}}
\subfloat[Convergence of \cref{eq:spectral-helmholtz}.]{\includegraphics[width =0.49 \textwidth]{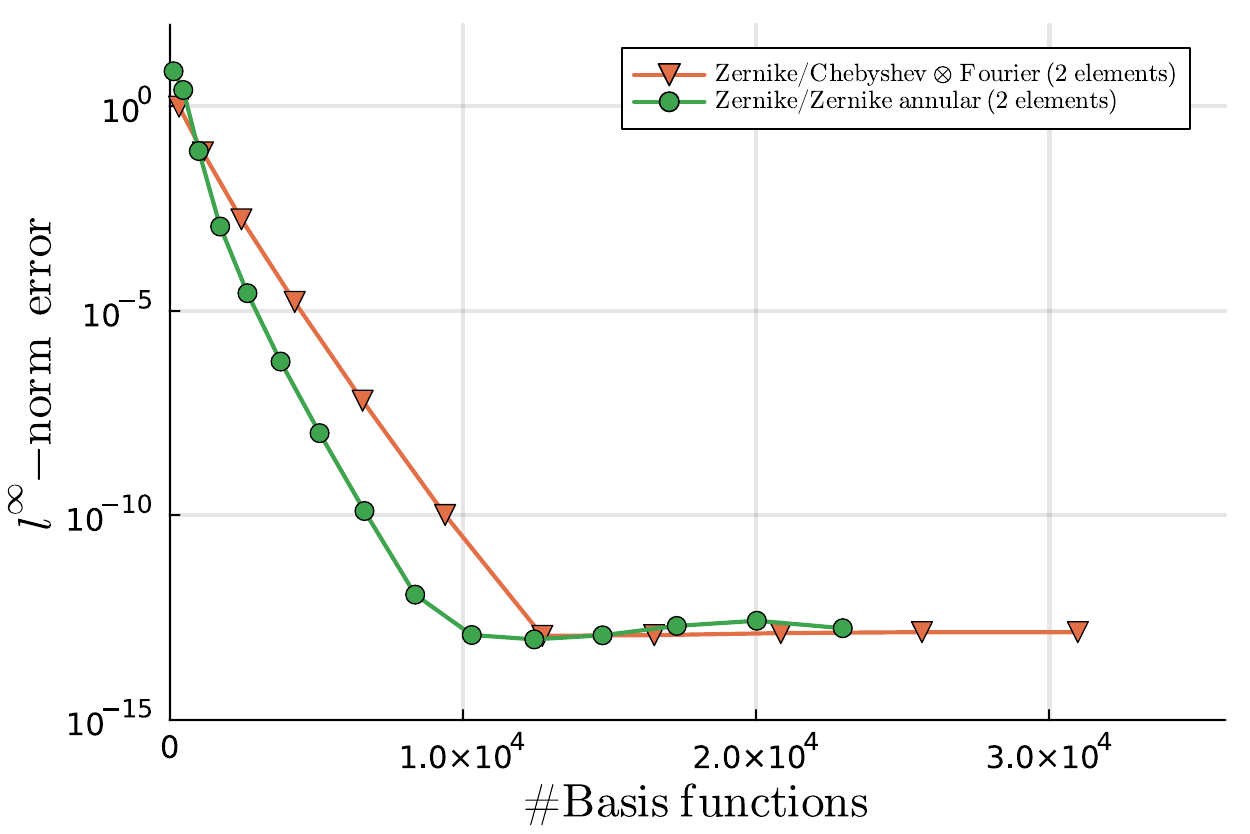}}
\caption{Convergence of the spectral methods for solving \cref{eq:spectral-poisson} and \cref{eq:spectral-helmholtz}. For both right-hand sides, the two-element methods with the boundary of the cells at the radial discontinuity performs significantly better than the one-element analogue. Moreover, the Zernike (annular) two-element methods performs the best overall.}\label{fig:spectral-element-convergence}
\end{figure}

\subsection{Discontinuous variable coefficients and data on an annulus}
\label{sec:example:disc-ann}
We study a similar problem to the one described in \cref{sec:disc-data} but on the annulus domain $\Omega_{1/10}$. We consider the Helmholtz equation \cref{eq:helmholtz} with the right-hand side \cref{eq:f-element2}. The parameters are $\rho = 1/10$, $\alpha = 1/2$, and for $i \in \{1,2,3,4\}$, $d_i = 1$, $a_i = -10i$, $(b_i,c_i) = (\alpha \cos(\theta_i), \alpha\sin(\theta_i))$ where $\theta_i = 0, \pi/2, \pi/3, 5\pi/4$, respectively. For $i=5$ we pick $(d_5, a_5, b_5,c_5) = (10, 80, -0.95, 0)$. Finally $\lambda(r) = \kappa(r)$ where $\kappa(r) = \kappa_0 = 100$ if $\rho \leq r \leq \alpha$ and $\kappa(r) = \kappa_1 = 1$ if $\alpha <r \leq 1$. The right-hand side has a radial discontinuity at $r=1/2$.

We compare three spectral methods: (1) a two-element method Zernike annular method, (2) a two-element method via a Chebyshev--Fourier series on the inner annulus cell and a Zernike annular discretisation on the outer annulus cell and (3) a two-element Chebyshev--Fourier series. The inner cell is the annulus $\{ 1/10 \leq r \leq 1/2\}$ and the outer cell is the annulus $\{ 1/2\leq r \leq 1\}$. The discontinuous right-hand, the continuous solution and the convergence plots are depicted in \cref{fig:annulus-element-convergence}. The quickest convergence is achieved by (1) followed by (2) and then (3).

\begin{figure}[h!]
\centering
\subfloat[Right-hand side.]{\includegraphics[width =0.32 \textwidth]{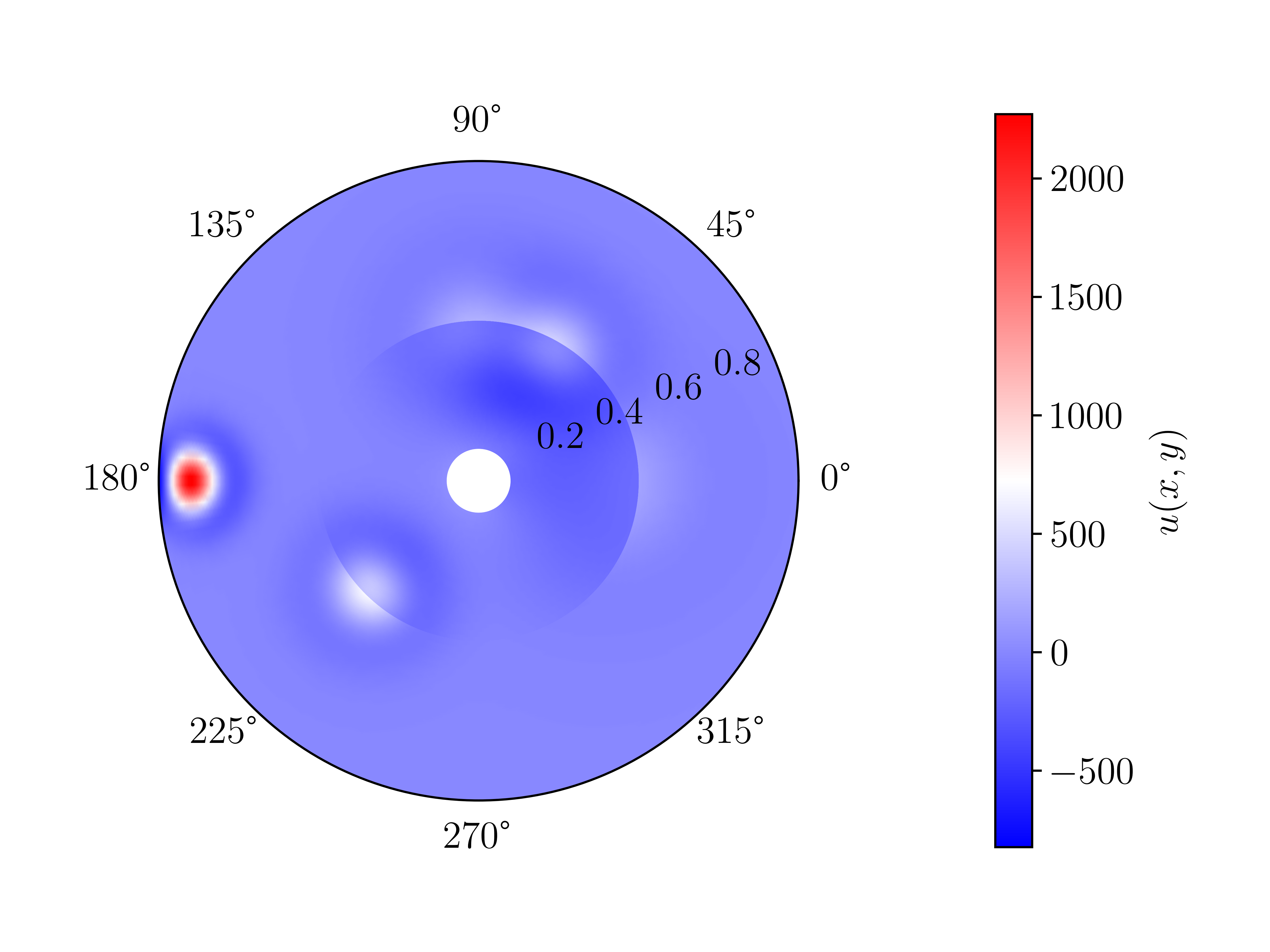}}
\subfloat[Solution.]{\includegraphics[width =0.32 \textwidth]{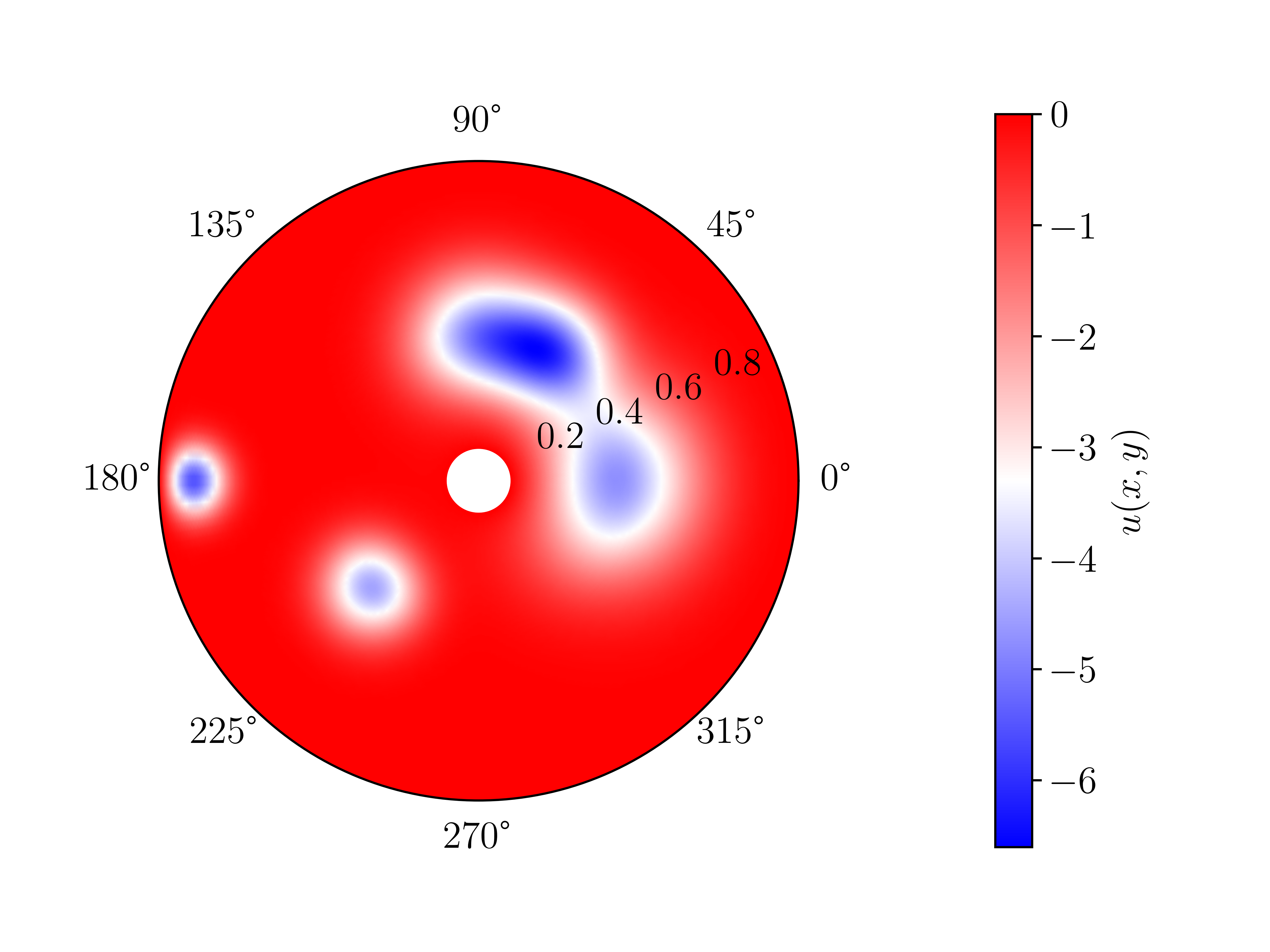}}
\subfloat[Convergence.]{\includegraphics[width =0.32 \textwidth]{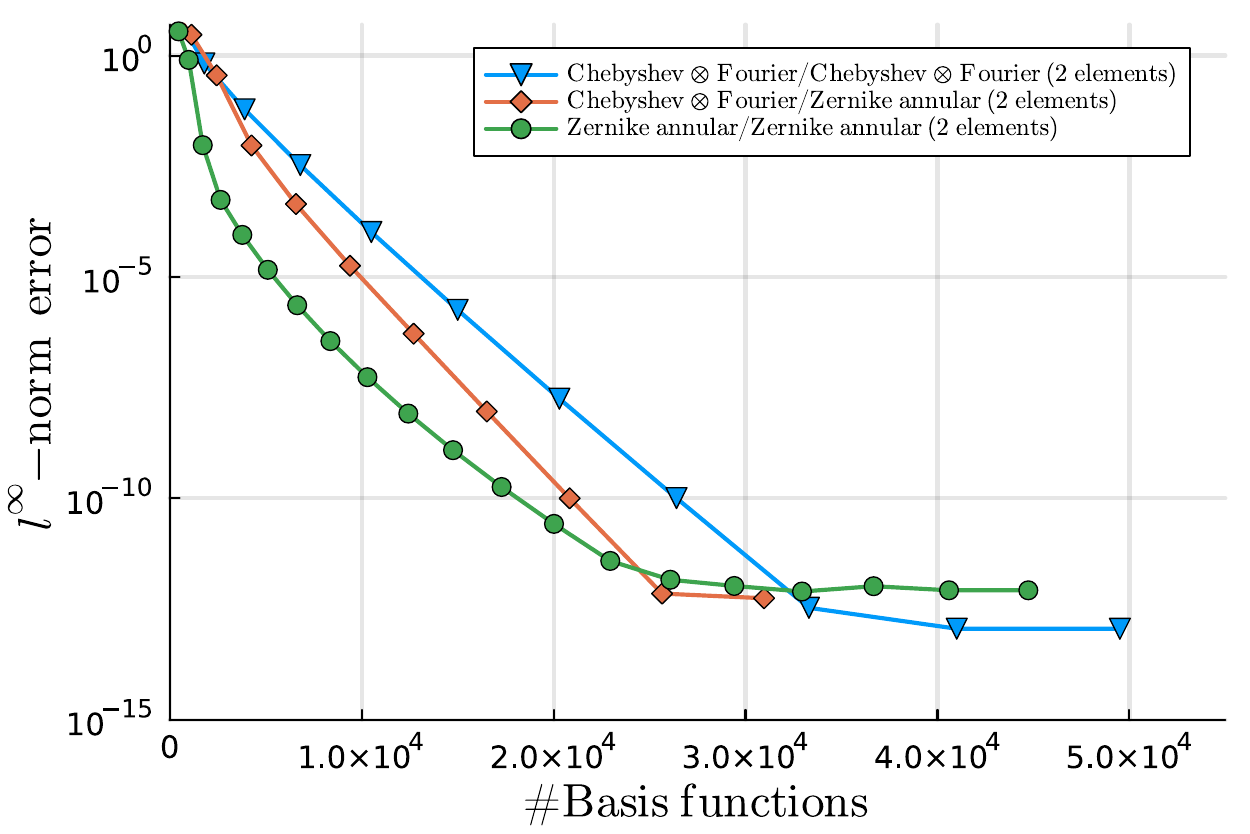}}
\caption{Convergence of the spectral methods for solving the example in \cref{sec:example:disc-ann}. The domain is the annulus with inner radius $\rho=1/10$. The right-hand side has a radial discontinuity at $r=1/2$. The mesh consists of two annuli cells which meet at $r=1/2$. The Zernike annular two-element method performs the best overall, followed by the Chebyshev--Fourier/Zernike annular method and lastly the Chebyshev--Fourier discretisation on both cells.}\label{fig:annulus-element-convergence}
\end{figure}

\section{Conclusions}
\label{sec:conclusions}

In this work we detailed an optimal complexity algorithm for computing connection and differentiation matrices as well a quasi-optimal complexity algorithm for the analysis and synthesis operators of hierarchies of semiclassical Jacobi polynomials. These allowed us to develop similar optimal complexity computations for the generalised Zernike annular polynomials. With the generalised polynomials, we construct several sparse spectral methods for solving PDEs posed on the annulus and disk. In particular we focused on the scaled-and-shifted Chebyshev--Fourier series and the Zernike annular polynomials. Akin to similar observations by Boyd and Yu \cite{Boyd2011} we observed that, for problems that feature a high Fourier mode, the Zernike annular polynomials often converge faster. A key note is that the Helmholtz operator discretised with generalised Zernike annular polynomial results in matrices of bandwidth five as opposed to the bandwidth of nine given by the scaled-and-shifted Chebyshev--Fourier series. We also constructed a spectral element method for the disk and annulus where the cells are an inner disk (omitted if the domain is an annulus) and subsequent concentric annuli of varying thickness. We used this spectral element method to solve the Helmholtz equation with a right-hand side and a variable coefficient with discontinuities in the radial direction. The spectral element method converged quickly to machine precision whereas one cell counterparts did not reduce the error below $\mathcal{O}(10^{-4})$ for the truncation degrees considered. Moreover, the Laplacian has better conditioning when a Zernike annular basis is used.

Unlike the Chebyshev--Fourier series, generalised Zernike annular polynomials may be used to discretise the Helmholtz equation in weak form resulting in a banded and sparse discretisation that preserves symmetry. Constructing a sparse spectral element method that utilises this approach will the subject of future work. We conclude by noting that the definition of the Zernike annular polynomials as given in this work naturally extends to three dimensions. This is achieved by simply considering the three-dimensional spherical harmonics in \cref{def:def:2Dannuli}. Hence, it is possible to construct sparse spectral methods for spherical shells in three dimensions.

\section*{Acknowledgments}
We are grateful to Tom H.~Koornwinder for providing us with previous literature on Zernike annular polynomials,  to Keaton Burns for discussions on tau methods, and to James Bremer for the discussion concerning solvers for radially symmetric potentials. 

%\clearpage
\bibliographystyle{siamplain}
\bibliography{references}
\end{document}